\documentclass[11pt]{amsart}
\usepackage{amsfonts}
\usepackage{amsmath}
\usepackage{amsthm}
\usepackage{amssymb}  
\usepackage{bm}
\usepackage[dvips]{graphicx}
\usepackage{caption}
\usepackage{color}

\usepackage{bigints}

\numberwithin{equation}{section}

\allowdisplaybreaks

\theoremstyle{plain}
\newtheorem{theorem}{Theorem}[section]
\newtheorem{lemma}[theorem]{Lemma}
\newtheorem{proposition}[theorem]{Proposition}

\theoremstyle{definition}
\newtheorem{definition}[theorem]{Definition}
\newtheorem{remark}[theorem]{Remark}
\newtheorem{example}[theorem]{Example}

\def\beqn{\begin{equation}}
\def\beqn*{$$}
\def\eeqn{\end{equation}}

\newcommand{\bx}{{\bf x}}
\newcommand{\by}{{\bf y}}

\def\P{\mathbb{P}}
\def\E{\mathbb{E}}

\def\B{\mathcal B}

\newcommand{\X}{{\mathcal{X}}}
\newcommand{\Y}{{\mathcal{Y}}}

\newcommand{\one}{{\bf 1}}

\newcommand{\al}{\alpha}

\newcommand{\C}{\check{C}}

\newcommand{\Gknt}{G_{k,n}(t)}

\newcommand{\Z}{\mathcal Z}

\newcommand{\iid}{\mathrm{i. i. d}}
\newcommand{\vc}{\text{\v{C}}}
\newcommand{\cc}{\check{C}}
\newcommand{\M}{\mathcal{M}}
\newcommand{\N}{\mathbb{N}}

\newcommand{\R}{\mathbb{R}}

\newcommand{\cs}{C^{*}}

\newcommand{\icj}{{(i,j)}}

\newcommand{\ijk}{{(i,j,K)}}
\newcommand{\ind}[1]{\mathbf{1}\big\{#1\big\}}

\newcommand{\ijp}{{(i,j,+)}}
\newcommand{\ijm}{{(i,j,-)}}
\newcommand{\basicsup}{\sup_{0 \le t \le 1}}
\newcommand{\ijupa}{{(i,j,\uparrow)}}
\newcommand{\ijdoa}{{(i,j,\downarrow)}}
\newcommand{\ktop}{{(k+2,1,+)}}
\newcommand{\aRrho}{R_{q_m} + a(R_{q_m})\rho}

\begin{document}

\bibliographystyle{abbrv}

\title[Functional SLLN for Betti numbers]
{Functional strong law of large numbers for Betti numbers in the tail}
\author{Takashi Owada}
\address{Department of Statistics\\
Purdue University \\
West Lafayette, 47907, USA}
\email{owada@purdue.edu}
\author{Zifu Wei}
\address{Department of Statistics\\
Purdue University \\
West Lafayette, 47907, USA}
\email{wei296@purdue.edu}

\thanks{This research was partially supported by NSF grant DMS-1811428.}

\keywords{Functional strong law of large numbers, extreme value theory, Betti number, topological crackle.  \vspace{.5ex}}

\begin{abstract}
The objective of this paper is to investigate the layered structure of topological complexity in the tail of a probability distribution. We establish the functional strong law of large numbers for Betti numbers, a basic quantifier of algebraic topology, of a geometric complex outside  an open ball of radius $R_n$, such that $R_n\to\infty$ as the sample size $n$ increases. The nature of the obtained law of large numbers is determined by the decay rate of a probability density. It especially depends on whether the tail of a density decays at a regularly varying rate or an exponentially decaying rate. The nature of the limit theorem depends also on how rapidly $R_n$ diverges. In particular, if $R_n$ diverges sufficiently slowly, the limiting function in the law of large numbers is crucially affected by the emergence of arbitrarily large connected components supporting topological cycles in the limit. 
\end{abstract}

\maketitle

\section{Introduction} \label{sec:intro}

The main theme of this paper is a phenomenon called \emph{topological crackle}, which refers to a layered structure of homological elements of different orders. In a practical setting, topological crackle appears in a manifold learning problem as in Figure \ref{f:perils_ht_noise}. Our aim in Figure  \ref{f:perils_ht_noise} is to recover the topology of the circle $S^1$ from the union of balls centered around $100$ uniformly distributed points. 
In Figure \ref{f:perils_ht_noise} (c), Gaussian noise is added to the uniform random sample. Then, the union is similar in shape to the circle $S^1$ and the recovery of its topology is possible. However, if the noise has a ``heavy tailed" Cauchy distribution as in Figure \ref{f:perils_ht_noise} (d), the extraneous homological elements (i.e., two distinct components and a one-dimensional cycle) away from the center of $S^1$, will render the recovery of its topology difficult (or even impossible). This example indicates that a small number of ``topologically essential" components and cycles may drastically change topology at a global level. 

\begin{figure}[t]
\centering
\includegraphics[width = 4in]{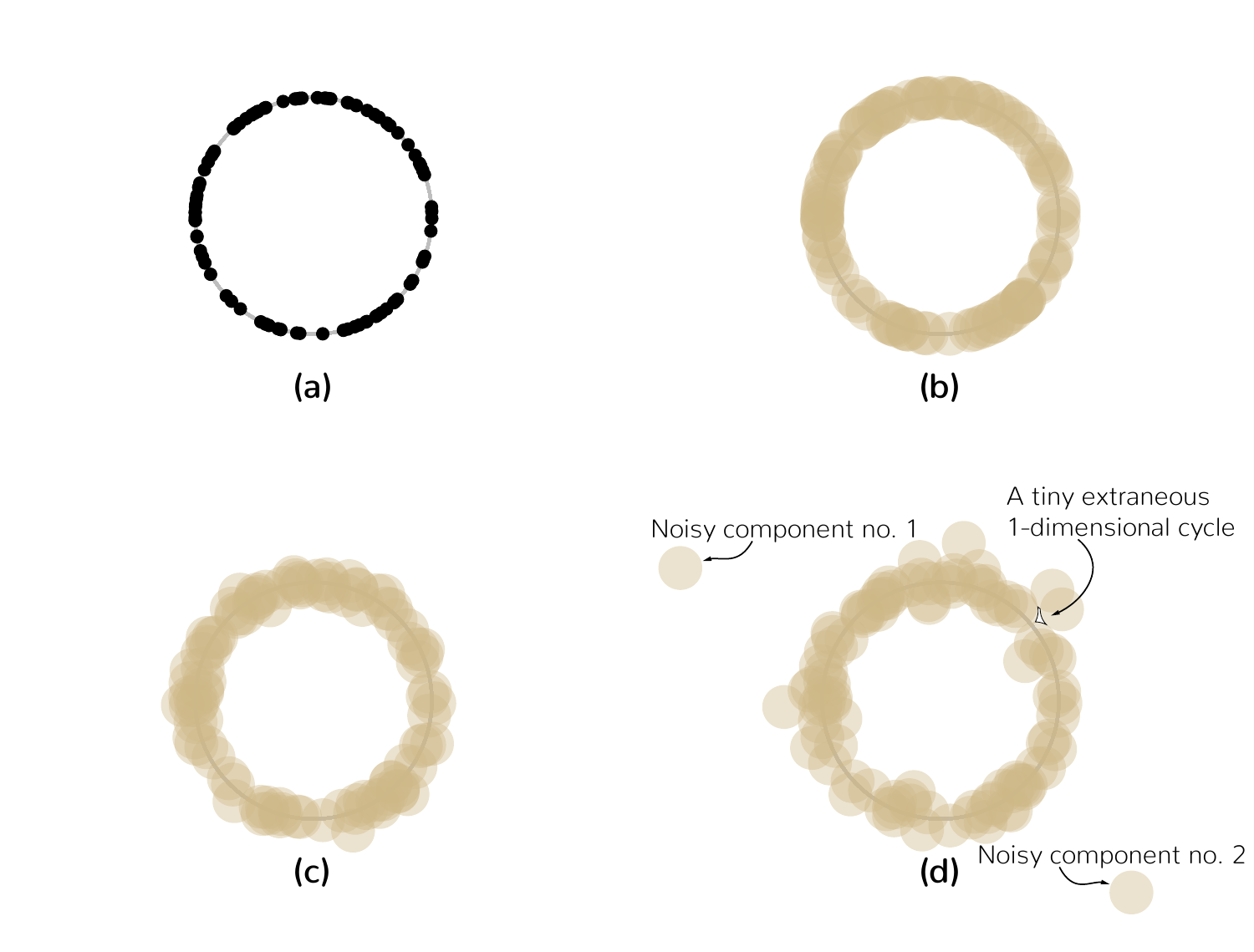}
\caption{\footnotesize{(a) Uniformly distributed points on the circle $S^1$. (b) A family of open balls with radius $0.2$ around the uniform points. (c) Gaussian noise is added to the uniform points via the algorithm in \cite{niyogi:smale:weinberger:2011}. (d) When heavy tailed Cauchy noise is added, several extraneous components and cycles appear. This figure is taken from \cite{thomas:owada:2021b}. }}
\label{f:perils_ht_noise}
\end{figure}

The current work is partially motivated by extreme value theory (EVT), in light of the fact that topological crackle is associated with a ``rare event" occurring in the tail of probability distributions. Beyond the standard literature on EVT  (see, e.g., \cite{resnick:1987,embrechts:kluppelberg:mikosch:1997,dehaan:ferreira:2006}), there have been many attempts so far to understand the geometric and topological features of extreme sample clouds  \cite{balkema:embrechts:2007, schulte:thale:2012, decreusefond:schulte:thaele:2016, adler:bobrowski:weinberger:2014, owada:adler:2017}. This paper contributes to the existing literature by examining  crackle phenomena in a higher dimensional setting. To make our implication more transparent, we consider the power-law density 
\begin{equation}  \label{e:power.law.density}
f(x) = \frac{C}{1+\|x\|^\alpha}, \ \ x\in \R^d, 
\end{equation}
for some $\alpha>d$ and a normalizing constant $C>0$. Let Ann$(K,L)$ be an annulus with inner radius $K$ and outer radius $L$. We then divide $\R^d$ into the layers of annuli at different radii, 
\begin{equation}  \label{e:annuli.power.law}
\R^d = \bigcup_{i=1}^{d+1} \text{Ann} (R_{i,n}, R_{{i-1},n}), 
\end{equation}
where 
\begin{equation}  \label{e:radii.power.law}
R_{i,n} = \begin{cases}
\infty & \ i=0, \\
(Cn)^{\frac{1}{\alpha-d/(i+2)}} & \ i\in\{ 1,\dots,d-1 \}, \\
(Cn)^{1/\alpha} & \ i=d, \\
0 & \ i=d+1. 
\end{cases}
\end{equation}
Figure \ref{f:crackle} visualizes the annuli \eqref{e:annuli.power.law} as a layer of \emph{Betti numbers} of all feasible dimensions.  The Betti number is  a quantifier of topological complexity in view of homological elements such as components and cycles in Figure \ref{f:perils_ht_noise} (d). For every $k\in \{ 1,\dots,d-1 \}$, the $k$th Betti number represents the number of $k$-dimensional cycles (henceforth we call it a $k$-cycle) as a boundary of a $(k+1)$-dimensional body. In Section \ref{sec:setup}, we provide a more rigorous description of Betti numbers under a more precise setup. 
Suppose $\X_n=\{ X_1,\dots,X_n \}\subset \R^d$ denotes independently and identically distributed ($\iid$)  random points in $\R^d$, $d\ge 2$, drawn from \eqref{e:power.law.density}. Let $B(x,r) = \big\{ y\in \R^d: \| y-x\| < r\big\}$  be an open ball of radius $r$ around $x\in \R^d$ ($\|\cdot\|$ is the Euclidean norm), and 
$$
U(r) = \bigcup_{X\in \X_n}B(X, r), \ \ \ r\ge 0, 
$$
be the union of open balls  around the points in $\X_n$. 
Given a growing sequence $R_n\to\infty$ and a positive integer $k\in \{  1,\dots,d-1\}$, the $k$th Betti number in Figure \ref{f:crackle} is  defined as 
\begin{equation}  \label{e:intro.betti}
\beta_{k,n}(t) := \beta_k \Big( \bigcup_{X\in \X_n\cap B(0,R_n)^c} B(X,t/2)\Big), \ \ \ t\ge0.  
\end{equation}
Then, $\beta_{k,n}(t)$ counts the number of $k$-cycles in the outside of an expanding ball $B(0,R_n)$. Note also that \eqref{e:intro.betti} is seen as a stochastic process of right continuous sample paths with left limits. 

\begin{figure}[t]
\includegraphics[scale=0.32]{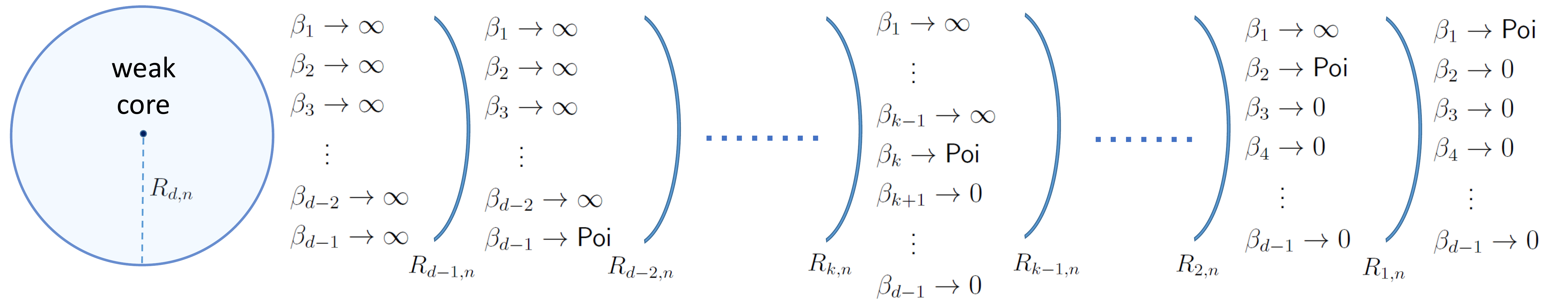}
\caption{\label{f:crackle} \footnotesize{Topological crackle is a layered structure of topological invariants such as Betti numbers.  The $k$th Betti number is denoted as $\beta_k$, and ``\textsf{Poi}'' stands for a Poisson distribution.  } }
\end{figure}

After the pioneering paper of \cite{adler:bobrowski:weinberger:2014}, the layered structure  in Figure \ref{f:crackle} has been intensively studied via the behavior of various topological invariants \cite{owada:adler:2017, owada:2017, owada:2018, owada:bobrowski:2020, thomas:owada:2021b}. In particular, from the viewpoints of \eqref{e:intro.betti}, we have in an asymptotic sense, 
\vspace{3pt}

\begin{itemize}
\item Outside Ann$(R_{1,n},\infty) = B(0,R_{1,n})^c$ there are finitely many $1$-cycles (accordingly, the first order Betti number is approximated by a Poisson distribution), but none of the cycles of dimensions $2,\dots,d-1$. 
\vspace{3pt}

\item Inside Ann$(R_{2,n}, R_{1,n})$ there are infinitely many $1$-cycles and finitely many $2$-cycles, but none of the cycles of dimensions $3,\dots,d-1$. 
\end{itemize}
\vspace{3pt}

In general, for every $k\in \{ 3,\dots,d-1 \}$, 
\vspace{3pt}

\begin{itemize}
\item Inside Ann$(R_{k,n}, R_{k-1,n})$ there are infinitely many cycles of dimensions $1,\dots,k-1$ and finitely many $k$-cycles, but none of the cycles of dimensions $k+1,\dots,d-1$, 
\end{itemize}
\vspace{3pt}

and finally, 
\vspace{3pt}

\begin{itemize}
\item Inside Ann$(R_{d,n}, R_{d-1,n})$ there are infinitely many cycles of \emph{all} dimensions $1,\dots,d-1$. 
\end{itemize}
\vspace{3pt}

In the literature (e.g., \cite{owada:2017}), the innermost ball $B(0,R_{d,n}) = B\big(0, (Cn)^{1/\alpha}\big)$ is called a \emph{weak core}. A weak core is a centered ball, in which random points are highly densely distributed and the homology of the union of balls becomes nearly trivial as $n\to\infty$, i.e., almost all cycles of any dimension inside of a weak core will  be asymptotically filled in.  See Section \ref{sec:setup} for a  formal definition of a weak core. 

As expected from  Figure \ref{f:crackle}, it is not surprising that the stochastic features of \eqref{e:intro.betti} drastically vary, depending on  how rapidly $R_n$ diverges. Among many of the related studies  \cite{owada:adler:2017, owada:2017, owada:2018, owada:bobrowski:2020, thomas:owada:2021b}, the authors of \cite{owada:adler:2017} explored the case $R_n=  R_{k,n} = (Cn)^{\frac{1}{\alpha-d/(k+2)}}$, so that there appear at most finitely many $k$-cycles outside of $B(0,R_n)$ as $n\to\infty$. Consequently, \eqref{e:intro.betti} converges weakly to a Poisson distribution. If $R_n$ diverges more slowly, such that $R_{d,n}\ll R_n \ll R_{k,n}$ as $n\to\infty$, then,  infinitely many $k$-cycles would appear in Ann$(R_n,R_{k,n})$. In this case, \eqref{e:intro.betti} obeys a functional central limit theorem (CLT) (see \cite{owada:2018}). However, the spatial distribution of $k$-cycles  is still sparse. As a result, all the $k$-cycles contributing to the limiting Gaussian process will be of the minimal size (i.e., they all consist of $k+2$ points). If $R_n$ diverges even more slowly, so that $R_n=R_{d,n}=(Cn)^{1/\alpha}$, then $U(t)$ becomes highly connected in the area close to a boundary of a weak core. Although \eqref{e:intro.betti} still follows a CLT under an appropriate normalization,   the components supporting the limiting $k$-cycles will become  arbitrarily large \cite{owada:2018}. In other words, the limiting $k$-cycles  can be supported not only on $k+2$ points, but also on $i$ points for all  $i\ge k+2$. 

The main objective of this paper is to establish the functional strong law of large numbers (SLLN) for \eqref{e:intro.betti} in the space $D[0,1]$ of right continuous functions defined on $[0,1]$ with left limits. This study is relevant only when the behavior of \eqref{e:intro.betti} is governed by a CLT. In this sense, the current study is viewed as a natural continuation of the work in \cite{owada:2018}. We will see that  the nature of the functional SLLN for \eqref{e:intro.betti}, including the scaling constants and the properties of the limiting functions, differs according to the growth rate of $R_n$. 
In Section \ref{sec:fslln.heavy}, we demonstrate two distinct functional SLLNs, depending on whether $R_{d,n}\ll R_n \ll R_{k,n}$ or $R_n=R_{d,n}$, when the density has a regularly varying tail as in \eqref{e:power.law.density}. 

In the literature, the previous studies  \cite{goel:duy:tsunoda:2019, yogeshwaran:subag:adler:2017, thomas:owada:2021a, hiraoka:shirai:trinh:2018} also proved the SLLNs
 for topological invariants such as Betti numbers and the Euler characteristic. However, all of these studies treated topology of an entire space $\R^d$. Then, the topological invariants are crucially impacted by   densely scattered random points closer to the origin, especially those inside of a weak core. As a consequence, the resulting SLLN is robust to the choice of the density $f$, and the topological invariants  grow proportionally to the sample size. 
In the context of topological crackle, however, the limit theorems for topological invariants heavily depend on the decay rate of a probability density. For instance, according to \cite{adler:bobrowski:weinberger:2014, owada:adler:2017, owada:bobrowski:2020}, the layered structure in Figure \ref{f:crackle} appears only when the distribution for $\X_n$ has a tail at least as heavy as that of an exponential distribution. Therefore, if $\X_n$ is drawn from a Gaussian distribution, for each $k \ge 1$, $\beta_{k,n}(t)$ in \eqref{e:intro.betti} simply vanishes as $n\to \infty$.
 
From this point of view, the other main thrust of this paper, which was never studied in \cite{owada:2018}, is to explore how crackle occurs when the density of $\X_n$ has a (sub)exponential tail. One of the simplest cases of such densities is 
\begin{equation}  \label{e:exp.decaying.density}
f(x) =C e^{-\|x\|^\tau/\tau}, \ \ x\in \R^d, 
\end{equation}
for some $\tau\in (0,1]$ and a normalizing constant $C>0$. 
Then, instead of \eqref{e:annuli.power.law} we obtain the ``logarithmic scale" annuli structure 
$$
\R^d = \bigcup_{i=1}^{d+1} \text{Ann} (R_{i,n}, R_{{i-1},n}), 
$$
where 
$$
R_{i,n} = \begin{cases}
\infty & \ i=0, \\
\big(\tau \log n + (i+2)^{-1} (d-\tau) \log (\tau \log n) +\tau \log C\big)^{1/\tau} & \ i\in\{ 1,\dots,d-1 \}, \\
(\tau \log n + \tau \log C)^{1/\tau} & \ i=d, \\
0 & \ i=d+1. 
\end{cases}
$$
In Section \ref{sec:slln.exp}, we establish the functional SLLN for \eqref{e:intro.betti} in the space $D[0,1]$ when the density has an exponentially decaying tail as in \eqref{e:exp.decaying.density}. The asymptotics of \eqref{e:intro.betti} once again depends on the growth rate of $R_n$: a phase transition occurs between the cases $R_{d,n} \ll R_n \ll R_{k,n}$ and  $R_n=R_{d,n}$. 

From an analytic viewpoint, the main challenge of the current study is that the scaling constants for the Betti number may grow  logarithmically, whereas  in the previous studies  \cite{goel:duy:tsunoda:2019, yogeshwaran:subag:adler:2017, thomas:owada:2021a, hiraoka:shirai:trinh:2018}, the scaling constant was equal to the sample size $n$ up to multiplicative constants. More concretely, if the density has an (sub)exponential tail as in \eqref{e:exp.decaying.density}, the radii of annuli are all logarithmic. Accordingly, the scaler for the Betti number will be logarithmic as well; see Example \ref{ex:exp.decaying.density} for more detailed analyses.  In the case of a power-law density,  the radii of annuli are polynomial as shown in  \eqref{e:radii.power.law}. Nevertheless, the scaler for the Betti numbers can still be logarithmic, especially when    $(R_n)$ and $(R_{k,n})$  have the same regular variation exponent (i.e., $\big(\alpha-d/(k+2) \big)^{-1}$), and the difference  between $R_{k,n}$and $R_n$ is at most logarithmic.  Example \ref{ex:power.law.density} below provides more detailed analyses on this point. If the scaler is logarithmic,  a direct application of the Borel-Cantelli lemma, together with the lower-order moment calculations, does not help to establish the required SLLNs. To overcome this difficulty, we shall utilize the concentration inequalities in \cite{bachmann:reitzner:2018}, which themselves were developed for analyzing  Poisson $U$-statistics of the geometric configuration of a point cloud, such as subgraph counts of a random geometric graph. For the application of these concentration bounds, one needs to detect appropriate subsequential upper and lower bounds for various quantities that are used to approximate Betti numbers. The latter approach is a  standard technique  for the theory of geometric graphs; see, e.g., Chapter 3 of the monograph \cite{penrose:2003}. 

This paper is organized as follows. In Section \ref{sec:setup}, we provide a precise setup of the Betti number in \eqref{e:intro.betti} via the notion of a \v{C}ech complex. In Section \ref{sec:main.result}, we develop the  functional SLLNs in two distinct scenarios: one where the distribution has a regularly varying tail and the other where the distribution has an exponentially decaying tail. For each of the distributional contexts, the behavior of \eqref{e:intro.betti}  splits into two additional distinct cases, according to how rapidly $R_n$ diverges. All the proofs are deferred to Section \ref{sec:proofs} and the Appendix. 

As a final remark, our proposed methods are applicable to other geometric complexes beyond a \v{C}ech complex. In particular, for a Vietoris-Rips complex, all the results can be carried over by a simple modification of the scaler of the SLLN. 

\section{Setup} \label{sec:setup}

We begin with introducing fundamental concepts towards proving the main functional SLLNs. Let $\mathcal X_n=\{X_1, \dots, X_n\}$ be a set of $n$ $\iid$ $\R^d$-valued random variables with probability density $f$.  
For the rigorous description of the Betti number in \eqref{e:intro.betti},  we need to introduce a higher-dimensional notion of graphs, called the \emph{geometric complex}. Among many varieties of geometric complexes (see \cite{ghrist:2014}), we especially focus on the \emph{\v{C}ech complex}.
\begin{definition}
Given a set $\X=\{ x_1,\dots,x_n \}$ of points in $\R^d$ and a positive number $r>0$, the \v{C}ech complex $\C(\X,r)$ is defined as follows.
\begin{enumerate}
\item The $0$-simplices are the points in $\X$.
\item For each $m\ge 1$, $[ x_{i_0},\dots,x_{i_m} ]\subset \X$ forms an $m$-simplex if $\bigcap_{j=0}^{m} B(x_{i_j},r/2) \neq \emptyset$.
\end{enumerate}
\end{definition}
The main advantage of the \v{C}ech complex is its homotopy equivalence to the union of balls $U(r/2)$. This fundamental result is known as the Nerve lemma (e.g., Theorem 10.7 in \cite{bjorner:1995}).

The objective of this paper is to study the ``extreme-value behavior" of the \v{C}ech complex generated by points far away from the origin. More concretely, for any sequence  $(R_n, n \geq 1)$ of positive numbers with $R_n\to\infty$, and a positive number $t>0$, 
we consider a \vc ech complex $\cc (\mathcal X_n\cap B(0, R_n)^c, t)$ built over random points in $\X_n$ lying outside of $B(0,R_n)$. Then, the collection of \vc ech complexes
\begin{equation}  \label{e:cech.comp}
\big(  \cc (\mathcal X_n\cap B(0, R_n)^c, t), \, t \ge 0 \big)
\end{equation}
induces a filtration in  parameter $t$; that is, the \vc ech complexes under consideration are non-decreasing, 
$$
\cc(\X_n\cap B(0, R_n)^c, s) \subset \cc(\mathcal X_n\cap B(0, R_n)^c, t)
$$
for all $0<s\le t< \infty$. Given a positive integer $k\in \{ 1,\dots,d-1 \}$ that remains fixed henceforth, the $k$th Betti number of  \eqref{e:cech.comp} is denoted as 
\begin{equation}  \label{e:Betti.tail}
\beta_{k,n}(t) := \beta_k \Big(  \cc(\mathcal X_n\cap B(0, R_n)^c, t)  \Big) = \beta_k \bigg(  \bigcup_{X\in \X_n \cap B(0,R_n)^c} B(X, t/2)\bigg),  \ \ \ t\ge 0, 
\end{equation}
where the second equality is a consequence of the Nerve lemma. Then, \eqref{e:Betti.tail} can be viewed as a non-negative integer-valued stochastic process in parameter $t$ that has right continuous sample paths with left limits. 

To derive the required functional SLLN, we need a more explicit expression of $\beta_{k,n}(t)$. For this purpose, we express $\beta_{k,n}(t)$ in the same way as the previous studies in \cite{kahle:meckes:2013, owada:2018}. Define for $\Y=(y_1,\dots,y_i)\in(\R^d)^i$, $i\geq k+2$, $j\geq 1$, and $t\geq 0$,
\begin{equation} \label{e:def.htij}
h_t^\icj(\Y):= \one \big\{ \beta_k(\cc(\Y, t))=j, \, \cc(\Y, t)  \text{ is connected}\big\},
\end{equation}
and
\begin{equation}  \label{e:def.hntij}
h_{n,t}^\icj(\Y):= h_t^\icj(\Y)\, \one \big\{\M(\Y)\geq R_n\big\},
\end{equation}
with
$$
\M(\Y):= \min_{1\leq \ell\leq i}\lVert y_\ell \rVert. 
$$
For any subset $\Z\supset\Y$ consisting of finitely many $d$-dimensional real vectors, we define
$$
g_t^\icj(\Y,\Z):= h_t^\icj(\Y)\, \one \big\{\cc(\Y, t) \text{ is a connected component of } \cc(\Z, t)\big\}, 
$$
and also, 
\begin{equation} \label{e:def.gntij}
g_{n,t}^\icj(\Y,\Z):= h_{n,t}^\icj(\Y)\, \one \big\{\cc(\Y, t) \text{ is a connected component of } \cc(\Z, t)\big\}, 
\end{equation}
and
\begin{equation}  \label{e:def.Jknij}
J_{k,n}^\icj(t):= \sum_{\Y\subset\X_n, \, |\Y|=i}g_{n,t}^\icj(\Y,\X_n).
\end{equation}
Then, we can express $\beta_{k,n}(t)$ as
\begin{equation}\label{e:def.betsum}
\beta_{k,n}(t)=\sum_{i=k+2}^n\sum_{j\geq 1}j J_{k,n}^\icj (t). 
\end{equation}

Before concluding this section, we shall formally define the notion of a weak core. 
\begin{definition}  \label{def.weak.core}
Let $f$ be a spherically symmetric probability density in $\R^d$. A weak core is a centered ball $B(0,R_n^{(w)})$ such that $nf(R_n^{(w)}\theta) \to 1$ as $n\to\infty$, for every (equivalently, some) $\theta \in S^{d-1}$,  where $S^{d-1}$ is the $(d-1)$-dimensional unit sphere in $\R^d$
\end{definition}
As mentioned in the Introduction, if $f$ has  a specific power-law tail as in \eqref{e:power.law.density},  it is  easy to see that $R_n^{(w)}=R_{d,n}=(Cn)^{1/\alpha}$. Similarly, if $f$ is given by \eqref{e:exp.decaying.density}, one can take  $R_n^{(w)}=R_{d,n}=(\tau \log n + \tau \log C)^{1/\tau}$. See \cite{owada:2017} for more detailed information about a weak core.

\section{Functional strong law of large numbers}  \label{sec:main.result}
We develop two main theorems in this section. In the below, we denote by $RV_\gamma$ the collection of regularly varying functions or sequences (at infinity) of exponent $\gamma\in \R$. 

\subsection{Regularly varying tail case}  \label{sec:fslln.heavy}
First, we explore the case that the density $f$ is spherically symmetric and has a regularly varying tail (at infinity) of exponent $-\alpha$ with $\alpha>d$. Namely, we assume that 
\begin{equation} \label{e:hvytl}
\lim_{r\to\infty}\frac{f(rt\theta)}{f(r\theta)}=t^{-\alpha}
\end{equation}
for every $t>0$ and $\theta\in S^{d-1}$. With spherical symmetry 
of $f$, one can denote $f(r):= f(r\theta)$, $r\ge 0$ for any fixed $\theta \in S^{d-1}$. 
Then, \eqref{e:hvytl} can be written as $f\in RV_{-\alpha}$.

From the literature of topological crackle \cite{adler:bobrowski:weinberger:2014, owada:adler:2017, owada:2018}, it is known that 
$$
\E\big[ \beta_{k,n}(t)\big]\sim C  \rho_n, \ \ \text{as } n\to\infty,
$$
where $C$ is a finite and positive constant and 
$$
\rho_n:= n^{k+2}R_n^d f(R_n)^{k+2}. 
$$
Since the main theme of this paper is a (functional) SLLN for $\beta_{k,n}(t)$, we treat only the case $\rho_n \to \infty$. Under this assumption, Theorem \ref{t:fslln.hvt} below splits the behavior of $\beta_{k,n}(t)$ into two distinct regimes: 
\begin{align*}
&(i) \ \ nf(R_n)\to0, \ \ \ n\to\infty,  \\
&(ii) \ \ nf(R_n) \to \lambda, \ \ n\to\infty, \ \text{for some } \lambda \in (0,\infty). 
\end{align*}
Clearly,  $(R_n)$ in case $(i)$ diverges more rapidly than that in case $(ii)$. 
In case $(i)$, the \v{C}ech complex outside of $B(0,R_n)$ is so sparse  that all the $k$-cycles remaining in the limit of the SLLN will be the simplest one formed by $k+2$ vertices. In contrast, if $R_n$ is determined by condition $(ii)$, the ball $B(0,R_n)$ agrees with the weak core (see Definition \ref{def.weak.core}) up to a proportionality constant. Then, the random points are highly connected to one another in the area sufficiently close to the weak core, and  the limiting function in the SLLN will be much more complicated, because of many connected components formed on $i$ vertices for any $i\ge k+2$. In the special case when $f$ is given by \eqref{e:power.law.density}, condition $(i)$ together with $\rho_n\to\infty$ is equivalent to $R_{d,n}\ll R_n \ll R_{k,n}$ where $R_{d,n}$ and $R_{k,n}$ are defined in \eqref{e:radii.power.law}. Furthermore, condition $(ii)$  implies that $R_n$ is equal to $R_{d,n}$ up to multiplicative factors.

Before stating the main theorem,  for $i\ge k+2$ and $j\ge1$ we define
\begin{equation} \label{e:def.mukij}
\mu_k^{(i,j)}(t; \lambda):= s_{d-1}\int_1^\infty\rho^{d-1-\alpha i}\int_{(\R^d)^{i-1}}h_t^{(i,j)}(0,\by)e^{-\lambda\rho^{-\alpha}t^d \text{vol}\big(\B(\{0,\by\};1)\big)}d\by d\rho, \ \ t\ge 0, \ \lambda\ge0, 
\end{equation}
where $s_{d-1}$ is surface area of the $(d-1)$-dimensional unit sphere and $\by =(y_1,\dots,y_{i-1}) \in (\R^d)^{i-1}$, and 
$$
\B \Big( \big\{ x_1,\dots,x_i \big\}; r  \Big) = \bigcup_{j=1}^i B(x_j, r), \ \ \ x_1,\dots, x_i \in \R^d, 
$$
is the union of open balls of radius $r$  around the points in $\{ x_1,\dots,x_i \}$. Moreover, vol$(B)$ is the volume of a subset $B\subset \R^d$, and $\omega_d$ is the volume of the unit ball in $\R^d$. In the special case $i=k+2$, $j=1$ and $\lambda =0$, we can simplify \eqref{e:def.mukij} as 
$$
\mu_k^{(k+2,1)}(t; 0) = \frac{s_{d-1}}{\alpha (k+2)-d}\, \int_{(\R^d)^{k+1}} h_t^{(k+2,1)}(0,\by) d\by, \ \ t\ge 0. 
$$
In the below, for two sequences $(a_n)$ and $(b_n)$, we write $a_n=\Omega (b_n)$ if there exists a positive constant $C>0$ such that $a_n/b_n \ge C$ for all $n\ge 1$. 

\begin{theorem}  \label{t:fslln.hvt}
$(i)$ Suppose that $R_n$ is a regularly varying sequence of positive exponent, such that $nf(R_n)\to 0$ as $n\to\infty$ and 
\begin{equation}  \label{e:cond:rhon}
\rho_n  = \Omega \big( (\log n)^\eta \big) \ \ \text{ for some } \eta>0. 
\end{equation}
Then, we have, as $n\to\infty$, 
$$
\frac{\beta_{k,n}(t)}{\rho_n} \to \frac{\mu_k^{(k+2,1)}(t; 0)}{(k+2)!}, \ \ \text{a.s.~in } D[0,1]. 
$$
$(ii)$ Suppose that $nf(R_n) \to \lambda$ as $n\to\infty$, for some $\lambda\in \big( 0,(e\omega_d)^{-1} \big)$. Then, as $n\to\infty$, 
\begin{equation}  \label{e:fslln.heavy.weak.core}
\frac{\beta_{k,n}(t)}{R_n^d} \to \mu_k (t; \lambda), \ \ \text{a.s.~in } D[0,1], 
\end{equation}
where 
$$
\mu_k (t; \lambda) := \sum_{i=k+2}^\infty \sum_{j \ge 1} j \frac{\lambda^i}{i!}\, \mu_k^\icj (t; \lambda) <\infty. 
$$
\end{theorem}

\begin{remark} 
In Theorem \ref{t:fslln.hvt} $(ii)$ above, there is a technical assumption that $\lambda < (e\omega_d)^{-1}$. As shown in Section \ref{sec:proof.fslln.hvt}, however, the  SLLN holds for \emph{every} $\lambda>0$, in the case of the truncated Betti number 
\begin{equation}  \label{e:truncated.Betti.main.section}
\beta_{k,n}^{(M)}(t) := \sum_{i=k+2}^M \sum_{j\ge 1}j J_{k,n}^{(i,j)} (t), \ \ \ \ M =k+2,k+3, \dots. 
\end{equation}
In fact, the upper bound condition for $\lambda$ is applied only when we show that the difference between \eqref{e:def.betsum} and \eqref{e:truncated.Betti.main.section} vanishes a.s.~when they are scaled by $R_n^d$. 
\end{remark}

\begin{example}  \label{ex:power.law.density}
We consider the power-law density
$$
f(x)=\frac{C}{1+\|x \|^\alpha}, \ \ x\in \R^d,
$$
for some $\alpha>d$ and a normalizing constant $C>0$. 
In this example, we consider three distinct sequences of growing radii, 
\begin{align*}
&(i) \ \ R_n = (\log n)^{-\xi} n^{\frac{1}{\alpha-d/(k+2)}} \ \ \text{for some } \xi>0, \\
&(ii) \ \ R_n = n^{\frac{1}{\al-d/b}} \ \ \text{ for some } k+2 < b < \infty, \ \ \\
&(iii) \ \ R_n = (cn)^{1/\alpha} \ \ \text{for some } c>C e \omega_d. 
\end{align*}
Note that $(R_n)$ in cases $(i)$ and $(ii)$ satisfies $R_{d,n}\ll R_n \ll R_{k,n}$ as $n\to\infty$, where $R_{d,n}$ and $R_{k,n}$ are defined in \eqref{e:radii.power.law}. In case $(iii)$, $R_n$ is equal to $R_{d,n}$ up to multiplicative factors. 

Since $(R_n)$ in case $(i)$ grows fastest, the occurrence of $k$-cycles outside  $B(0,R_n)$ is the least likely of the three regimes. In particular, $(R_n)$ and $(R_{k,n})$ are  ``close" in size, in the sense that both the sequences have the same regular variation exponent. Then, $\beta_{k,n}(t)$ grows only logarithmically. In fact, one can readily check that 
$$
\rho_n \sim C^{k+2} (\log n)^{\xi(\alpha (k+2)-d)}, \ \ \ n\to\infty, 
$$
and Theorem \ref{t:fslln.hvt} $(i)$ yields that 
\begin{equation}  \label{e:fslln.heavy.ex1}
\frac{\beta_{k,n}(t)}{(\log n)^{\xi(\alpha (k+2)-d)}} \to \frac{C^{k+2}\mu_k^{(k+2,1)}(t; 0)}{(k+2)!}, \ \ n\to\infty, \ \ \text{a.s. in } D[0,1]. 
\end{equation}

If $R_n$ diverges more slowly as in case $(ii)$, then, far more $k$-cycles would occur in Ann$(R_n,R_{k,n})$, and $\beta_{k,n}(t)$ begins to grow polynomially. More concretely, we have 
$$
\rho_n \sim C^{k+2} n^{\frac{d(1-(k+2)/b)}{\al-d/b}}, \ \ n\to\infty, 
$$
and Theorem \ref{t:fslln.hvt} $(i)$ gives that 
\begin{equation}  \label{e:fslln.heavy.ex2}
\frac{\beta_{k,n}(t)}{n^{\frac{d(1-(k+2)/b)}{\al-d/b}}} \to \frac{C^{k+2} \mu_k^{(k+2,1)}(t; 0)}{(k+2)!}, \ \ n\to\infty, \ \ \text{a.s. in } D[0,1].  
\end{equation}

Finally, if $R_n$ diverges even more slowly as in case $(iii)$, then $\beta_{k,n}(t)$ grows even faster and Theorem \ref{t:fslln.hvt} $(ii)$ concludes that 
\begin{equation}  \label{e:fslln.heavy.ex3}
\frac{\beta_{k,n}(t)}{n^{d/\alpha}} \to c^{d/\alpha} \mu_k(t; C/c), \ \ \ n\to\infty, \ \ \text{a.s.~in } D[0,1]. 
\end{equation}
The limiting function above becomes more complicated than those in \eqref{e:fslln.heavy.ex1} and \eqref{e:fslln.heavy.ex2}, because of the emergence of arbitrarily large connected components  supporting $k$-cycles in the limit. 
\end{example}

\subsection{Exponentially decaying tail case}  \label{sec:slln.exp}
The aim of this section is to explore the limiting behavior of \eqref{e:def.betsum} when $f$ has an exponentially decaying tail. As a generalization of \eqref{e:exp.decaying.density}, we consider the density
\begin{equation}\label{e:exp.density}
f(x)=C\exp\big\{-\psi\big(\lVert x\rVert\big)\big\}, \qquad x\in\R^d, 
\end{equation}
where $C>0$ is a normalizing constant and $\psi:[0,\infty)\to [0,\infty)$ is a regularly varying function (at infinity) with index $\tau \in (0,1]$. Because of the spherical symmetry of \eqref{e:exp.density}, we can define $f(r):=f(r\theta)$, $r\ge 0$ for any fixed $\theta$. 

Before continuing, we need to put additional assumptions on $\psi$. First, it is assumed to be twice differentiable such that $\psi'(x)>0$ for all $x>0$, and $\psi'(x)$ is eventually non-increasing; that is, there is a $x_0>0$ so that $\psi'(x)$ is non-increasing for all $x\ge x_0$. Let $a(z):= 1/\psi'(z)$, $z>0$. It then follows from Proposition 2.5 in \cite{resnick:2007} that $a\in \text{RV}_{1-\tau}$. In \cite{owada:adler:2017} it was shown that topological crackle occurs if and only if 
\begin{equation}  \label{e:limit.a}
c:= \lim_{z\to \infty} a(z) \in (0,\infty]. 
\end{equation}
From this viewpoint, the case $\tau>1$ is irrelevant for the current study. Indeed, if $\tau>1$, it holds that $a(z) \to 0$ as $z\to \infty$ and topological crackle does not occur. Moreover, if $\tau\in (0,1)$, then \eqref{e:limit.a} is automatic with $c=\infty$. We thus need to assume \eqref{e:limit.a} only when $\tau=1$. 

Under this setup,  it was shown in the literature  \cite{adler:bobrowski:weinberger:2014, owada:adler:2017, owada:bobrowski:2020}  that 
$$
\E \big[ \beta_{k,n}(t) \big] \sim C\eta_n, \ \  \ n\to\infty,
$$
where $C$ is a finite and positive constant and 
$$
\eta_n := n^{k+2}a(R_n)R_n^{d-1} f(R_n)^{k+2}. 
$$
By the same reasoning as in Section \ref{sec:fslln.heavy}, we focus only on the case $\eta_n\to\infty$ as $n\to\infty$. As before, the asymptotics of $\beta_{k,n}(t)$ is divided into  two different regimes: 
\begin{align*}
&(i) \ \ nf(R_n)\to0, \ \ \ n\to\infty,  \\
&(ii) \ \ nf(R_n) \to \lambda, \ \ n\to\infty, \ \text{for some } \lambda \in (0,\infty). 
\end{align*}
For the analyses of case $(i)$, we shall restrict $(R_n)$ to a slightly narrower class, 
\begin{equation}  \label{e:specific.Rn}
R_n = \psi^\leftarrow (\log n +b\log \log n), \ \  \text{for some }   b \in \Big(0, \, \frac{d-\tau}{ \tau (k+2) } \Big), 
\end{equation}
where $\psi^\leftarrow (x) = \inf \big\{ y:\psi(y) \ge x \big\}$, $x>0$,  
is the (left continuous) inverse of $\psi$.  
Note that condition $(i)$ requires that a constant $b$ in \eqref{e:specific.Rn} must be positive. Moreover, if one takes $b > \frac{d-\tau}{\tau(k+2)}$, then $\eta_n$ no longer diverges as $n\to\infty$; hence, we need to restrict the range of $b$ as in \eqref{e:specific.Rn}.

Finally, for the description of the limiting function in the SLLN, we need the following functions: for  $i\ge k+2$ and $j\ge 1$, 
\begin{align}
\xi_k^{(i,j)}(t; \lambda):=&\int_{S^{d-1}}\int_0^\infty \int_{(\R^d)^{i-1}}h_t^\icj(0,\by)\, e^{-\rho i-c^{-1}\sum_{\ell=1}^{i-1}\langle\theta,y_\ell\rangle}  \prod_{\ell=1}^{i-1} \one \big\{ \rho+c^{-1}\langle \theta, y_\ell\rangle \ge 0 \big\} \label{e:def.xikij}  \\
&\qquad \qquad  \times\exp\Big\{-\lambda e^{-\rho}\int_{\B(\{0,\by\};  t)} e^{-c^{-1} \langle \theta, z\rangle}\,  dz \Big\}d\by\, d\rho\, J(\theta)\, d\theta, \ \ t \ge 0, \ \lambda\ge 0, \notag 
\end{align}
where $\by=(y_1,\dots,y_{i-1}) \in (\R^d)^{i-1}$ and $\langle \cdot, \cdot \rangle$ is the  Euclidean inner product. Further, $J(\theta)$ is the Jacobian 
$$
J(\theta) = \sin^{d-2}(\theta_1)\sin^{d-3}(\theta_2)\dots \sin (\theta_{d-2}). 
$$

\begin{theorem} \label{t:fslln.expdt}
$(i)$ Suppose that $R_n = \psi^\leftarrow (\log n + b\log \log n)$ 
for some $0 < b < \frac{d-\tau}{ \tau (k+2) }$. 
Then, as $n\to\infty$, 
$$
\frac{\beta_{k,n}(t)}{\eta_n} \to \frac{\xi_k^{(k+2,1)}(t; 0)}{(k+2)!}, \ \ \ \text{a.s.~in }  D[0,1]. 
$$
$(ii)$ Suppose that $nf(R_n) \to \lambda$ as $n\to\infty$, for some $\lambda\in \big( 0,(e\omega_d)^{-1} \big)$. If $d=2$, we restrict the range of $\tau$ to $(0,1)$. 
Then, as $n\to\infty$, 
$$
\frac{\beta_{k,n}(t)}{a(R_n)R_n^{d-1}} \to \xi_k(t; \lambda), \ \ \text{a.s.~in } D[0,1], 
$$
where
$$
\xi_k(t; \lambda):=\sum_{i=k+2}^\infty\sum_{j\ge 1}j\, \frac{\lambda^i}{i !}\, \xi_k^{(i,j)}(t; \lambda) <\infty. 
$$
\end{theorem}

\begin{example}  \label{ex:exp.decaying.density}
We consider the density
$$
f(x) = Ce^{-\|x\|^\tau/\tau},  \ \ \ x\in \R^d, 
$$
for some $\tau\in (0,1]$ and  a normalizing constant $C>0$. Then, $\psi(z)=z^\tau /\tau$ and $a(z)=z^{1-\tau}$ for $z>0$. 
We first set
$$
R_n = \psi^\leftarrow (\log n +b\log \log n) = (\tau \log n + b\tau \log \log n)^{1/\tau}
$$
for some $b\in \big( 0,\frac{d-\tau}{\tau(k+2) }\big)$. It is then straightforward to calculate that 
$$
\eta_n \sim \tau^{(d-\tau)/\tau} C^{k+2} (\log n)^{\frac{d-\tau}{\tau}-b(k+2)}, 
$$
where $\frac{d-\tau}{\tau} -b(k+2)>0$. 
Thus, Theorem \ref{t:fslln.expdt} $(i)$ shows that, as $n\to\infty$, 
\begin{equation}  \label{e:fslln.exp.ex1}
\frac{\beta_{k,n}(t)}{(\log n)^{\frac{d-\tau}{\tau}-b(k+2)}} \to \frac{\tau^{(d-\tau)/\tau}C^{k+2} \xi_k^{(k+2,1)}(t; 0)}{(k+2)!}, \ \ \text{a.s.~in } D[0,1].  
\end{equation}

Next, we take $R_n = (\tau \log n + \log c_1)^{1/\tau}$ 
for some $c_1>(C e\omega_d)^\tau$. It then holds that $nf(R_n)=c_1^{-1/\tau}C<(e\omega_d)^{-1}$. 
Since $a(R_n)R_n^{d-1} \sim (\tau \log n)^{(d-\tau)/\tau}$ as $n\to\infty$, Theorem \ref{t:fslln.expdt} $(ii)$ indicates that 
$$
\frac{\beta_{k,n}(t)}{( \log n )^{\frac{d-\tau}{\tau}}} \to \tau^{(d-\tau)/\tau}\xi_k(t; c_1^{-1/\tau} C), \ \ \text{a.s.~in } D[0,1]. 
$$
As in \eqref{e:fslln.heavy.ex3} of Example \ref{ex:power.law.density}, the limiting function above is  affected by many  $k$-cycles  on arbitrarily large connected components, whereas the limit in \eqref{e:fslln.exp.ex1} is determined by the $k$-cycles supported exactly on $k+2$ points. 
\end{example}

\section{Proofs}  \label{sec:proofs}

\subsection{Preliminaries}

Before commencing the proof, we need to introduce additional functions and objects pertaining to the indicator \eqref{e:def.htij}. For $i \ge k+2$, $j \ge 1$, and a connected simplicial complex $K$ that has $i$ vertices, 
$$
h_t^\ijk(\Y) :=\one \big\{\cc(\Y, t)\cong K, \, \beta_k(K)=j \big\}, \ \ \Y=(y_1,\dots,y_{k+2}) \in (\R^d)^{k+2},
$$
where $\cong$ denotes isomorphism between simplicial complexes. With this notation we can interpret $h_t^\icj(\Y)$ as 
\begin{equation}\label{e:def.htijsumK}
h_t^\icj(\Y)=\sum_{K:|K|=i} h_t^\ijk(\Y), 
\end{equation}
where the sum is taken over all connected simplicial complexes on $i$ vertices. From \eqref{e:def.htijsumK}, $h_t^\icj(\Y)$ can be   decomposed as 
$$
h_t^\icj (\Y) = h_t^\ijp (\Y) - h_t^\ijm (\Y), 
$$
where 
\begin{align}
h_t^\ijp (\Y) &:= \sum_{K: |K|=i} h_t^\ijk (\Y) +  \sum_{K: |K|=i} \sum_{\substack{L: |L|=i, \\ L \supsetneq K}}\one \big\{ \cc(\Y, t) \cong L \big\}, \label{e:def.htijp}\\
h_t^\ijm (\Y) &:=   \sum_{K: |K|=i} \sum_{\substack{L: |L|=i, \\ L \supsetneq K}}\one \big\{ \cc(\Y, t) \cong L \big\}. \label{e:def.htijm}
\end{align}
In the above expressions, the inner sums in \eqref{e:def.htijp} and \eqref{e:def.htijm} are taken over all connected complexes $L$ built over $i$ points that contain $K$ as its \emph{proper} subcomplex. Since there are at most finitely many isomorphism classes of such complexes, these sums consist of at most a finite number of terms. Consequently, for any $i\ge k+2$ and $j\ge 1$, \eqref{e:def.htijp} and \eqref{e:def.htijm} are bounded functions. Notice further that \eqref{e:def.htijp} and \eqref{e:def.htijm} are non-decreasing functions in $t$. Namely, for all $0 < s \le t <\infty$ and $\Y\in (\R^d)^{k+2}$, 
\begin{equation}  \label{e:monotonicity.ind}
h_s^\ijp (\Y) \le h_t^\ijp (\Y), \ \ \ \ \ h_s^\ijm (\Y) \le h_t^\ijm (\Y). 
\end{equation}
In the special case $i=k+2$ and $j=1$, it is easy to check that for all $x_1,\dots,x_{k+2}\in \R^d$, 
$$
h_t^{(k+2,1,+)} (x_1,\dots,x_{k+2}) = \one \Big\{ \bigcap_{j=1,j\neq j_0}^{k+2}B(x_j,t)\neq\emptyset 
\text{ for all } j_0\in\{1,\dots, k+2\}\Big\},
$$
and 
$$
h_t^{(k+2,1,-)}(x_1,\dots,x_{k+2})=\one \Big\{\bigcap_{j=1}^{k+2}B(x_j, t)\neq\emptyset\Big\}.
$$
In addition to the monotonicity \eqref{e:monotonicity.ind}, the following properties of $h_t^{(i,j,\pm)}$ are important for our analyses. 
\begin{itemize}
\item $h_t^{(i,j,\pm)}$ are shift-invariant, i.e.,  
\begin{equation}  \label{e:shift.invariance}
h_t^{(i,j,\pm)}(\Y) = h_t^{(i,j,\pm)}(\Y+x), \ \ \Y\in (\R^d)^{k+2}, \ x\in \R^d. 
\end{equation}
\item $h_t^{(i,j,\pm)}$ are locally determined, i.e., there exists $L$ (depending only on $i$) such that 
\begin{equation}  \label{e:locally.determined}
h_t^{(i,j,\pm)}(\Y)=0, \ \ \text{whenever diam}(\Y)>Lt, 
\end{equation}
where diam$(\Y)=\max_{x,y \in \Y}\|x-y\|$. 
\end{itemize}
Finally, we define various functions analogous to those defined at \eqref{e:def.hntij}, \eqref{e:def.gntij}, \eqref{e:def.Jknij},  \eqref{e:def.mukij},  \eqref{e:def.xikij}, respectively.
\begin{align}
h_{n,t}^{(i,j,\pm)}(\Y)&:= h_t^{(i,j,\pm)}(\Y)\, \one \big\{\M(\Y)\geq R_n\big\}, \notag \\
g_{n,t}^{(i,j,\pm)}(\Y,\Z)&:= h_{n,t}^{(i,j,\pm)}(\Y)\, \one  \big\{\cc(\Y, t) \text{ is a connected component of } \cc(\Z, t)\big\},\label{e:def.gntijpm} \\
J_{k,n}^{(i,j,\pm)}(t) &:= \sum_{\Y\subset\X_n, \, |\Y|=i}g_{n,t}^{(i,j,\pm)}(\Y,\X_n),  \label{e:def.Jknijpm} \\
\mu_k^{(i,j, \pm)}(t; \lambda) &:= s_{d-1}\int_1^\infty\rho^{d-1-\alpha i}\int_{(\R^d)^{i-1}}h_t^{(i,j, \pm)}(0,\by)e^{-\lambda\rho^{-\alpha}t^d \text{vol}\big(\B(\{0,\by\};1)\big)}d\by d\rho, \label{e:def.mukijpm} \\
\xi_k^{(i,j, \pm)}(t; \lambda):=&\int_{S^{d-1}}\int_0^\infty \int_{(\R^d)^{i-1}}h_t^{(i,j,\pm)}(0,\by)\, e^{-\rho i-c^{-1}\sum_{\ell=1}^{i-1}\langle\theta,y_\ell \rangle}  \prod_{\ell=1}^{i-1} \one \big\{ \rho+c^{-1}\langle \theta, y_\ell\rangle \ge 0 \big\}  \label{e:def.xikijpm} \\
&\qquad \times\exp\Big\{-\lambda e^{-\rho}\int_{\B(\{0,\by\};  t)} e^{-c^{-1} \langle \theta, z\rangle}\,  dz \Big\}d\by\, d\rho\, J(\theta)\, d\theta.  \notag  
\end{align}

One of the main ideas for proving the SLLN for \eqref{e:def.betsum} is to detect appropriate subsequential upper and lower bounds for the quantities that are used to approximate the Betti number \eqref{e:def.betsum}. After detecting such bounds, we apply the Borel-Cantelli lemma to the obtained bounds. This is a standard approach for proving the SLLN for graph related statistics, such as subgraph and component counts, in random geometric graphs; see, e.g., Theorems 3.18 and 3.19 in \cite{penrose:2003}. 

Specifically, we take a constant $\gamma\in(0,1)$ and define 
\begin{equation}  \label{e:def.vm}
v_m = \lfloor e^{m^\gamma} \rfloor, \ \ m=0,1,2,\dots.
\end{equation}
Then, for each $n\in \N$, there exists a unique $m=m(n)\in \N$ such that $v_m \le n <v_{m+1}$. 
Given a sequence $(R_n, \, n \ge 1)$ growing to infinity, we define 
\begin{equation}  \label{e:def.pmqm}
p_m:=\arg\max\{v_m\leq \ell\leq v_{m+1}: R_\ell\}, \ \ q_m:=\arg\min\{v_m\leq \ell\leq v_{m+1}: R_\ell\}. 
\end{equation}
For the density $f$ satisfying \eqref{e:hvytl}, we set 
\begin{align}  
\begin{split} \label{e:def.rmsm}
r_m &:= \arg \max \{ v_m \le \ell \le v_{m+1}: R_\ell^d f(R_\ell)^{k+2} \},   \\
s_m &:= \arg \min \{ v_m \le \ell \le v_{m+1}: R_\ell^d f(R_\ell)^{k+2} \}, 
\end{split}
\end{align}
and for the density $f$ in \eqref{e:exp.density}, we define 
\begin{align}  
\begin{split}  \label{e:def.bmcm}
b_m &:= \arg \max \big\{ v_m\le \ell \le v_{m+1}: a(R_\ell) R_\ell^{d-1}  \big\},   \\
c_m &:= \arg \min \big\{ v_m\le \ell \le v_{m+1}: a(R_\ell) R_\ell^{d-1}  \big\}, 
\end{split}
\end{align}
and 
\begin{align}
\begin{split}  \label{e:def.emgm} 
e_m &:= \arg \max \big\{ v_m \le \ell \le v_{m+1}:  a(R_\ell)R_\ell^{d-1} f(R_\ell)^{k+2} \big\},   \\  
g_m &:=  \arg \min \big\{ v_m \le \ell \le v_{m+1}:  a(R_\ell)R_\ell^{d-1} f(R_\ell)^{k+2} \big\}. 
\end{split}
\end{align}

In the sequel we provide the proofs of Theorems \ref{t:fslln.hvt} and \ref{t:fslln.expdt}. Throughout the proof, $\cs$ denotes a generic positive constant,  which is independent of $n$ and may vary between and within the lines. For two sequences $(a_n)$ and $(b_n)$, we write $a_n\sim b_n$ if $a_n/b_n\to 1$ as $n\to\infty$.

\subsection{Proof of Theorem \ref{t:fslln.hvt}}  \label{sec:proof.fslln.hvt}

For ease of the description, we first prove Part $(ii)$ and then proceed to the proof of Part $(i)$. The discussion for Part $(i)$ is more delicate, requiring to  make use of the concentration bounds in \cite{bachmann:reitzner:2018}, which is stated in Proposition \ref{p:concentration} below. 
All the technical results necessary for our proof are provided in Sections \ref{sec:ingredients} and \ref{sec:technical.lemma.heavy} of the Appendix. 

\begin{proof}[Proof of Theorem \ref{t:fslln.hvt} $(ii)$]

Note first that \eqref{e:fslln.heavy.weak.core} is implied by 
$$
\sup_{0\leq t\leq 1}\Big\lvert \frac{\beta_{k,n}(t)}{R_n^d}-\mu_k(t; \lambda)\Big\rvert \to 0, \ \ \ n\to\infty, \ \ \text{a.s.}
$$ 
We introduce the truncated  version of \eqref{e:def.betsum}; that is, for every $M\in \N$, 
\begin{equation}  \label{e:truncated.betti.M}
\beta_{k,n}^{(M)}(t):=\sum_{i=k+2}^M \sum_{j\geq 1}j J_{k,n}^\icj (t). 
\end{equation}
Analogously, we also define $\mu_k^{(M)}(t; \lambda)$ by the same truncation as above: 
$$
\mu_k^{(M)}(t; \lambda) := \sum_{i=k+2}^M \sum_{j\geq 1}j  \frac{\lambda^i}{i!}\, \mu_k^\icj(t; \lambda). 
$$
Then, it is easy to see that, for every $M\in \N$, 
\begin{align*}
\sup_{0\leq t\leq 1}\Big\lvert \frac{\beta_{k,n}(t)}{R_n^d}-\mu_{k}(t; \lambda)\Big\rvert &\leq \sup_{0\leq t\leq 1}\Big\{ \frac{\beta_{k,n}(t)}{R_n^d} - \frac{\beta^{(M)}_{k,n}(t)}{R_n^d}\Big\} + \sup_{0 \le t \le 1} \Big| \frac{\beta_{k,n}^{(M)}(t)}{R_n^d} - \mu_k^{(M)}(t; \lambda) \Big| \\
& \quad + \sup_{0\le t \le 1} \big( \mu_k(t; \lambda) - \mu_k^{(M)}(t; \lambda) \big). 
\end{align*}
From this it is sufficient to show that 
\begin{align}
&\lim_{n\to\infty} \sup_{0\le t \le 1}  \Big| \frac{\beta_{k,n}^{(M)}(t)}{R_n^d} - \mu_k^{(M)}(t; \lambda) \Big| = 0, \ \ \text{a.s. for all } M\in \N,  \label{e:suff.cond.heavy1}\\
&\lim_{M\to\infty} \limsup_{n\to\infty} \sup_{0\leq t\leq 1}\Big\{ \frac{\beta_{k,n}(t)}{R_n^d} - \frac{\beta^{(M)}_{k,n}(t)}{R_n^d}\Big\} =0, \ \ \text{a.s.}, \label{e:suff.cond.heavy2}\\
&\lim_{M\to\infty} \limsup_{n\to\infty} \sup_{0\le t \le 1} \big( \mu_k(t; \lambda) - \mu_k^{(M)}(t; \lambda) \big) =0. \label{e:suff.cond.heavy3}
\end{align}
Of the three requirements above, we first deal with \eqref{e:suff.cond.heavy1}. Since $M$ is finite, it is enough to demonstrate that, for every $i \ge k+2$ and $j\ge 1$, 
$$
\basicsup \Big| \, \frac{J_{k,n}^\icj(t)}{R_n^d} -\frac{\lambda^i}{i!}\, \mu_k^\icj (t; \lambda) \, \Big| \to 0, \ \ n\to\infty, \ \ \text{a.s.}
$$
Clearly, this convergence is obtained by 
\begin{align}
&\basicsup \Big| \, \frac{J_{k,n}^\ijp(t)}{R_n^d} -\frac{\lambda^i}{i!}\, \mu_k^\ijp (t; \lambda) \, \Big| \to 0, \ \ \text{a.s.}, \label{e:compwise.as.convergence} \\
&\basicsup \Big| \, \frac{J_{k,n}^{(i,j,-)}(t)}{R_n^d} -\frac{\lambda^i}{i!}\, \mu_k^{(i,j,-)} (t; \lambda) \, \Big| \to 0, \ \ \text{a.s.}, \notag
\end{align}
where $J_{k,n}^{(i,j,\pm)}(t)$ and $\mu_k^{(i,j,\pm)}(t; \lambda)$ are defined in \eqref{e:def.Jknijpm} and \eqref{e:def.mukijpm} respectively. 
We prove only the former convergence. One can actually handle the latter  in the same way. 
For this purpose, we extend \eqref{e:def.gntijpm},  \eqref{e:def.Jknijpm}, and \eqref{e:def.mukijpm} by adding an extra time parameter. Namely, for $\Y\in (\R^d)^i$, a finite subset $\Z\supset \Y$ of points in $\R^d$, $(t, s) \in [0,1]^2$, and $\lambda \ge 0$,  
\begin{align}
g^\ijp_{n,t,s}(\Y,\Z) &:= h_{n,t}^\ijp(\Y)\, \one \big\{ \cc(\Y, s) \text{ is a connected component of } \cc(\Z, s) \big\} \notag \\
&=h_{t}^\ijp(\Y)\, \one \big\{  \M(\Y)\ge R_n \big\}\, \one \big\{\B(\Y;s/2)\cap\B(\Z \setminus\Y;s/2)=\emptyset\big\}, \notag \\
J_{k,n}^\ijp (t,s)  &:=\sum_{\Y\subset\X_n, \,\lvert \Y \rvert =i}g^\ijp_{n,t,s}(\Y,\X_n), \label{e:def.Jknijpts} \\
\mu_k^\ijp (t,s; \lambda) &:= s_{d-1}\int_1^\infty\rho^{d-1-\alpha i}\int_{(\R^d)^{i-1}}h_t^\ijp(0,\by)e^{-\lambda\rho^{-\alpha}s^d \text{vol}\big(\B(\{0,\by\};1)\big)}d\by d\rho.   \label{e:def.mukijpts}
\end{align}
Clearly, we have that $g_{n,t,t}^\ijp(\Y, \Z) = g_{n,t}^\ijp(\Y,\Z)$, $J_{k,n}^\ijp (t,t) = J_{k,n}^\ijp (t)$, and $\mu_k^\ijp (t,t; \lambda) = \mu_k^\ijp (t; \lambda)$. 
It  immediately follows from \eqref{e:monotonicity.ind} that \eqref{e:def.Jknijpts} and \eqref{e:def.mukijpts} are non-decreasing in $t$ and non-increasing in $s$. Moreover, \eqref{e:def.mukijpts} is a continuous function in $(t,s)\in [0,1]^2$. Hence, 
according to Lemma \ref{l:pointwise.to.funcl} $(ii)$ in the Appendix, \eqref{e:compwise.as.convergence} follows from the pointwise SLLN, 
\begin{equation}  \label{e:pointwise.slln.Jkn}
\frac{J_{k,n}^\ijp(t,s)}{R_n^d} \to \frac{\lambda^i}{i!}\, \mu_k^\ijp (t,s;\lambda), \ \ n\to\infty,  \ \ \text{a.s.}, 
\end{equation}
for every $t,s \in [0,1]$. 

To show \eqref{e:pointwise.slln.Jkn}, note that for every $n\in \N$, there exists a unique $m=m(n)\in \N$ such that $v_m \le n <v_{m+1}$, where $v_m$ is defined in \eqref{e:def.vm}. 
Using $p_m$ and $q_m$ in  \eqref{e:def.pmqm} as well as $v_m$, we define 
\begin{equation}\label{e:def.upbd}
T_m^\ijupa(t,s):=\sum_{\substack{\Y\subset\X_{v_{m+1}} \\ \lvert \Y \rvert = i}}h^\ijp_t(\Y)\ind{\M(\Y)\geq R_{q_m}}\ind{\B(\Y;s/2)\cap\B(\X_{v_m}\setminus\Y;s/2)=\emptyset},
\end{equation}
\begin{equation}\label{e:def.lobd}
T^\ijdoa_m(t,s):=\sum_{\substack{\Y\subset\X_{v_m} \\ \lvert \Y \rvert = i}}h^\ijp_t(\Y)\ind{\M(\Y)\geq R_{p_m}}\ind{\B(\Y;s/2)\cap\B(\X_{v_{m+1}}\setminus\Y;s/2)=\emptyset}.
\end{equation} 
Since $\X_{v_m}\subset\X_n\subset\X_{v_{m+1}}$ and $R_{q_m}\leq R_n\leq R_{p_m}$ for all $n\in \N$, one can immediately derive that 
$$
\frac{T^\ijdoa_m(t,s)}{R_{p_m}^d}\le \frac{J_{k,n}^\ijp (t,s)}{R_n^d}\leq\frac{T^\ijupa_m(t,s)}{R_{q_m}^d}, 
$$
for all $n\in \N$. 
It follows from Lemma \ref{l:first.second.moments.heavy} $(i)$ that 
$$
\limsup_{n\to \infty} \frac{J_{k,n}^\ijp(t,s)}{R_n^d} \le \frac{\lambda^i}{i !}\, \mu_k^\ijp(t,s; \lambda) + \limsup_{m\to\infty} \frac{T_m^\ijupa (t,s)-\E[T_m^\ijupa (t,s)]}{R_{q_m}^d}, \ \ \text{a.s.}, 
$$
and 
$$
\liminf_{n\to \infty} \frac{J_{k,n}^\ijp(t,s)}{R_n^d} \ge \frac{\lambda^i}{i !}\, \mu_k^\ijp(t,s; \lambda) + \liminf_{m\to\infty} \frac{T_m^\ijdoa (t,s)-\E[T_m^\ijdoa (t,s)]}{R_{p_m}^d}, \ \ \text{a.s.}
$$
From these bounds, we now need to show that as $m\to\infty$, 
\begin{align}
&R_{q_m}^{-d}\, \Big(  T_m^\ijupa (t,s)-\E[T_m^\ijupa (t,s)] \Big) \to 0,  \ \ \text{a.s.}, \label{e:upper.weak.core.heavy} \\
&R_{p_m}^{-d}\, \Big(  T_m^\ijdoa (t,s)-\E[T_m^\ijdoa (t,s)] \Big) \to 0,  \ \ \text{a.s.} \label{e:lower.weak.core.heavy}
\end{align}
For \eqref{e:upper.weak.core.heavy}, Chebyshev's inequality and Lemma \ref{l:first.second.moments.heavy} $(i)$ yield that, for every $\epsilon>0$, 
\begin{align*}
\sum_{m=1}^\infty \P \Big( \big|\,  T_m^\ijupa (t,s)-\E[T_m^\ijupa (t,s)] \, \big| >\epsilon R_{q_m}^d  \Big) \le \frac{1}{\epsilon^2} \sum_{m=1}^\infty \frac{\text{Var}\big( T_m^\ijupa(t,s) \big)}{R_{q_m}^{2d}} \le C^* \sum_{m=1}^\infty \frac{1}{R_{q_m}^d}. 
\end{align*}
From the assumption that $nf(R_n) \to \lambda$ as $n\to\infty$, one can readily check that  $R_n\in \text{RV}_{1/\alpha}$; therefore, 
\begin{equation}  \label{e:R.asym.heavy}
R_{q_m}\ge C^* q_m^{1/(2\alpha)}\ge \cs v_m^{1/(2\alpha)}\ge \cs e^{m^\gamma/(3\alpha)}, 
\end{equation}
and we thus conclude that $\sum_m 1/R_{q_m}^d \le C^*\sum_m e^{-dm^\gamma/(3\alpha)}<\infty$. Now, the Borel-Cantelli lemma ensures \eqref{e:upper.weak.core.heavy}. The proof of \eqref{e:lower.weak.core.heavy} is analogous by virtue of Lemma \ref{l:first.second.moments.heavy} $(i)$. Now, \eqref{e:upper.weak.core.heavy} and \eqref{e:lower.weak.core.heavy} complete the proof of \eqref{e:pointwise.slln.Jkn}. 

Next we turn to condition \eqref{e:suff.cond.heavy2}. Since the $k$th Betti number of any simplicial complex on $i$ vertices is bounded above by the number of $k$-simplices, which is further bounded by ${i \choose k+1}$, we have that 
\begin{align}
\frac{1}{R_n^d}\big( \beta_{k,n}(t)-\beta^{(M)}_{k,n}(t)\big) &=\frac{1}{R_n^d}\sum_{i=M+1}^n\sum_{j\geq 1}j J_{k,n}^\icj(t) \label{e:Wmmt}   \\
&\le \frac{1}{R_n^d}\sum_{i=M+1}^n{i \choose k+1}\sum_{\Y\subset\X_n, \lvert\Y\rvert=i}\sum_{j\geq 1}h^\icj_{n,t}(\Y)\, \notag \\ 
&\le \frac{1}{R_n^d}\sum_{i=M+1}^n i^{k+1}\sum_{\Y\subset\X_n, \lvert\Y\rvert=i}\ind{\cc(\Y, t) \text{ is connected},\, \M(\Y)\geq R_n} \notag \\ 
&\le \frac{1}{R_{q_m}^d}\sum_{i=M+1}^\infty i^{k+1}\hspace{-10pt} \sum_{\Y\subset\X_{v_{m+1}}, \lvert\Y\rvert=i}\hspace{-10pt}\ind{\cc(\Y, t) \text{ is connected},\, \M(\Y)\geq R_{q_m}}\notag \\ 
&=:\frac{V_{m,M}(t)}{R_{q_m}^d}. \notag
\end{align}
By construction, $V_{m,M}(t)$ is a non-decreasing function in $t$. Thus, 
$$
\sup_{0\leq t\leq 1}\Big\{  \frac{\beta_{k,n}(t)}{R_n^d} - \frac{\beta^{(M)}_{k,n}(t)}{R_n^d}\Big\} \le \frac{V_{m,M}(1)}{R_{q_m}^d}, 
$$
and we now need to show that 
$$
\lim_{M\to\infty} \limsup_{m\to\infty} R_{q_m}^{-d} \, V_{m,M}(1) = 0, \ \ \ \text{a.s.}
$$
By Lemma \ref{l:first.second.moments.heavy} $(ii)$, it is sufficient to demonstrate that for every $M\in \N$, 
\begin{equation}  \label{e:centered.SLLN}
\frac{V_{m,M}(1)-\E\big[ V_{m,M} (1)\big]}{R_{q_m}^d} \to 0, \ \ \text{as } m\to\infty, \ \ \text{a.s.}
\end{equation}
For every $\epsilon>0$, it follows from Lemma \ref{l:first.second.moments.heavy} $(ii)$ and \eqref{e:R.asym.heavy} that 
\begin{align*}
&\sum_{m=1}^\infty \P \Big( \Big|\, V_{m,M}(1) -\E\big[ V_{m,M}(1) \big] \, \Big| >\epsilon R_{q_m}^d \Big) \\
&\le \frac{1}{\epsilon^2}\, \sum_{m=1}^\infty \frac{\text{Var}\big( V_{m,M}(1) \big)}{R_{q_m}^{2d}} \le C^* \sum_{m=1}^\infty \frac{1}{R_{q_m}^d}  \le C^* \sum_{m=1}^\infty e^{-dm^\gamma / (3\alpha)} <\infty. 
\end{align*}
Thus, the Borel-Cantelli lemma ensures \eqref{e:centered.SLLN}. Finally, we assert that the argument similar to (or even easier than) that for \eqref{e:suff.cond.heavy2} can establish \eqref{e:suff.cond.heavy3}. Now, the entire proof has been completed. 
\end{proof}

\begin{proof}[Proof of Theorem \ref{t:fslln.hvt} $(i)$] 
Our main idea  is to justify that  $\beta_{k,n}(t)$ in \eqref{e:def.betsum} can be approximated by 
\begin{equation}  \label{e:def.Gkn}
G_{k,n}(t) := \sum_{\substack{\Y\subset \X_n, \\ |\Y|=k+2}} h_{n,t}^{(k+2,1)} (\Y), \ \ \ t \ge 0. 
\end{equation}
This implies that    the \v{C}ech complexes are  asymptotically distributed sparsely with many separate connected components, with each consisting of only $k+2$ points. Note that ``$k+2$" is the minimum required number of points for a non-trivial $k$th Betti number of a \v{C}ech complex. 

As in the proof of Theorem \ref{t:fslln.hvt} $(ii)$, our goal is to show that 
$$
\basicsup \Big| \, \frac{\beta_{k,n}(t)}{\rho_n} - \frac{\mu_k^{(k+2,1)}(t; 0)}{(k+2)!} \, \Big| \to 0, \ \ \ n\to\infty, \ \ \text{a.s.}
$$
This clearly follows if one can establish that 
\begin{align}
&\basicsup \Big| \, \frac{G_{k,n}(t)}{\rho_n} - \frac{\mu_k^{(k+2,1)}(t; 0)}{(k+2)!} \, \Big| \to 0, \ \ \ n\to\infty, \ \ \text{a.s.}, \label{e:suff.cond.crackle1.heavy}\\
&\basicsup \Big| \, \frac{\beta_{k,n}(t)}{\rho_n} - \frac{G_{k,n}(t)}{\rho_n}  \, \Big| \to 0, \ \ \ n\to\infty, \ \ \text{a.s.} \label{e:suff.cond.crackle2.heavy}
\end{align}
For \eqref{e:suff.cond.crackle1.heavy}, we define
\begin{equation}  \label{e:def.Gknpmt}
G_{k,n}^\pm(t) := \sum_{\substack{\Y\subset \X_n, \\ |\Y|=k+2}} h_{n,t}^{(k+2,1, \pm)} (\Y), \ \ \ t \ge 0, 
\end{equation}
and demonstrate that 
\begin{equation}  \label{e:Gkn.pm.funcl.SLLN}
\basicsup \Big| \, \frac{G_{k,n}^+(t)}{\rho_n} - \frac{\mu_k^{(k+2,1,+)}(t; 0)}{(k+2)!} \, \Big| \to 0,  \ \ \text{a.s.,} \ \  \basicsup \Big| \, \frac{G_{k,n}^-(t)}{\rho_n} - \frac{\mu_k^{(k+2,1,-)}(t; 0)}{(k+2)!} \, \Big| \to 0,  \ \ \text{a.s.}, 
\end{equation}
where $\mu_k^{(k+2,1,\pm)}(t; 0)$ is defined at \eqref{e:def.mukijpm}. 
In the below, we show the result of the ``+" case only. 
One can immediately show that 
$$
\mu_k^{(k+2,1,+)}(t; 0) = t^{d(k+1)} \mu_k^{(k+2,1,+)}(1; 0), 
$$
which in turn indicates that $\mu_k^{(k+2,1,+)}(t; 0)$ is a continuous and increasing function in $t$. Moreover, $G_{k,n}^+(t)$ is also a non-decreasing function in $t$; hence, by Lemma \ref{l:pointwise.to.funcl} $(i)$, the first convergence in \eqref{e:Gkn.pm.funcl.SLLN} is obtained from the pointwise SLLN, 
$$
\frac{G_{k,n}^+(t)}{\rho_n} \to \frac{\mu_k^{(k+2,1,+)}(t; 0)}{(k+2)!}, \ \ \ n\to\infty, \ \ \text{a.s.}, 
$$ 
for every $t\in [0,1]$. As before, our next task is to detect the subsequential upper and lower bounds for $G_{k,n}^+(t)/\rho_n$. 
Let $(v_m)$ be a sequence defined in \eqref{e:def.vm}. Recall that for every $n \ge \N$, there exists a unique $m=m(n)\in \N$ such that $v_m \le n < v_{m+1}$. 
Next we define 
\begin{align}
S_m^\uparrow (t) &:= \sum_{\substack{\Y\subset \mathcal P_{v_{m+1}} \\ |\Y|=k+2}} h_t^{(k+2,1,+)}(\Y)\, \one\big\{ \M(\Y) \ge R_{q_m}  \big\}, \label{e:def.Smup}\\
S_m^\downarrow (t) &:= \sum_{\substack{\Y\subset \mathcal P_{v_m}, \\ |\Y|=k+2}} h_t^{(k+2,1,+)}(\Y)\, \one\big\{ \M(\Y) \ge R_{p_m}  \big\}, \label{e:def.Smdown}
\end{align}
where $p_m$ and $q_m$ are given in  \eqref{e:def.pmqm}. Furthermore,  
$\mathcal P_{v_{m+1}}$ and $ \mathcal P_{v_m}$ denote Poisson processes in $\R^d$ with intensity measures $v_{m+1}\int_\cdot f(z)dz$ and $v_m \int_\cdot f(z)dz$,  respectively. In other words, one can denote  $\mathcal P_{v_{m+1}} = \{ X_1,\dots,X_{N_{v_{m+1}}} \}$, where $X_1,X_2,\dots$ are i.i.d random points with density $f$, and $N_{v_{m+1}}$ is Poisson distributed with mean $v_{m+1}$, independently of $(X_i, \, i \ge 1)$. 

Since $\rho_n \ge v_m^{k+2} R_{s_m}^d f(R_{s_m})^{k+2}$, $\X_n \subset \X_{v_{m+1}}$, and $R_n \ge R_{q_m}$ (see \eqref{e:def.rmsm} for the definition of $s_m$), one can bound $G_{k,n}^+(t)/\rho_n$ in a way that 
\begin{align} 
\frac{G_{k,n}^+(t)}{\rho_n}  &\le \big( v_{m}^{k+2}R_{s_m}^df(R_{s_m})^{k+2} \big)^{-1} \sum_{\substack{\Y\subset \X_{v_{m+1}}, \\ |\Y|=k+2}} h_t^\ktop(\Y)\, \one \big\{ \M(\Y) \ge R_{q_m} \big\} \label{e:initial.upper.bdd.Gknrhon}\\
&= \frac{S_m^\uparrow(t)}{v_{m}^{k+2}R_{s_m}^df(R_{s_m})^{k+2}} +  \big( v_{m}^{k+2}R_{s_m}^df(R_{s_m})^{k+2} \big)^{-1} \notag \\
&\qquad \qquad \qquad \qquad \qquad \times \bigg\{  \sum_{\substack{\Y\subset \X_{v_{m+1}}, \\ |\Y|=k+2}} h_t^\ktop(\Y)\, \one \big\{ \M(\Y) \ge R_{q_m} \big\}  -S_m^\uparrow(t)\bigg\}, \ \ n\ge 1. \notag
\end{align}  
By Lemmas \ref{l:RV.seq.heavy} and \ref{l:difference.binom.poisson}, 
\begin{align}
&\big( v_{m}^{k+2}R_{s_m}^df(R_{s_m})^{k+2} \big)^{-1}  \bigg\{  \sum_{\substack{\Y\subset \X_{v_{m+1}}, \\ |\Y|=k+2}} h_t^\ktop(\Y)\, \one \big\{ \M(\Y) \ge R_{q_m} \big\}  -S_m^\uparrow(t)\bigg\}  \label{e:diff.binom.poisson}\\
&=\Big( \frac{v_{m+1}}{v_m} \Big)^{k+2} \Big( \frac{R_{q_m}}{R_{s_m}} \Big)^{d} \bigg( \frac{f(R_{q_m})}{f(R_{s_m})} \bigg)^{k+2} \big( v_{m+1}^{k+2}R_{q_m}^df(R_{q_m})^{k+2} \big)^{-1}  \notag \\
&\qquad \qquad \qquad \qquad \times \bigg\{  \sum_{\substack{\Y\subset \X_{v_{m+1}}, \\ |\Y|=k+2}} h_t^\ktop(\Y)\, \one \big\{ \M(\Y) \ge R_{q_m} \big\}  -S_m^\uparrow(t)\bigg\} \to 0, \ \ m\to\infty, \ \ \text{a.s.} \notag  
\end{align}
Now, by Lemmas \ref{l:RV.seq.heavy} and \ref{l:first.second.asymptotics} as well as \eqref{e:diff.binom.poisson}, 
\begin{align}
\limsup_{n\to\infty} \frac{G_{k,n}^+(t)}{\rho_n} &\le \limsup_{m\to\infty} \frac{S_m^\uparrow(t)}{v_{m}^{k+2}R_{s_m}^df(R_{s_m})^{k+2}}\label{e:upper.Gkn.rhon} \\
&=\limsup_{m\to\infty} \Big( \frac{v_{m+1}}{v_m} \Big)^{k+2} \Big( \frac{R_{q_m}}{R_{s_m}} \Big)^d \bigg( \frac{f(R_{q_m})}{f(R_{s_m})} \bigg)^{k+2} \frac{S_m^\uparrow(t)}{v_{m+1}^{k+2}R_{q_m}^df(R_{q_m})^{k+2}} \notag \\
&\le \limsup_{m\to\infty} \frac{S_m^\uparrow(t)}{v_{m+1}^{k+2}R_{q_m}^df(R_{q_m})^{k+2}} \notag \\
&\le \frac{\mu_k^\ktop(t; 0)}{(k+2)!} + \limsup_{m\to\infty} \frac{S_m^\uparrow(t)-\E[S_m^\uparrow(t)]}{v_{m+1}^{k+2}R_{q_m}^df(R_{q_m})^{k+2}}, \ \ \text{a.s.} \notag 
\end{align}
As for the lower bound of $G_{k,n}^+(t)/\rho_n$, noting that $\rho_n \le v_{m+1}^{k+2} R_{r_m}^d f(R_{r_m})^{k+2}$, $\X_n\supset \X_{v_m}$, and $R_n \le R_{p_m}$ (see \eqref{e:def.rmsm} for the definition of $r_m$), we obtain that 
\begin{align*} 
\frac{G_{k,n}^+(t)}{\rho_n}  &\ge  \frac{S_m^\downarrow(t)}{v_{m+1}^{k+2}R_{r_m}^df(R_{r_m})^{k+2}} +  \big( v_{m+1}^{k+2}R_{r_m}^df(R_{r_m})^{k+2} \big)^{-1} \\
&\qquad \qquad \qquad \qquad \qquad \times \bigg\{  \sum_{\substack{\Y\subset \X_{v_{m}}, \\ |\Y|=k+2}} h_t^\ktop(\Y)\, \one \big\{ \M(\Y) \ge R_{p_m} \big\}  -S_m^\downarrow(t)\bigg\}. 
\end{align*}  
Repeating the same argument as in \eqref{e:upper.Gkn.rhon} and using Lemmas \ref{l:RV.seq.heavy}, \ref{l:first.second.asymptotics}, and \ref{l:difference.binom.poisson}, 
\begin{equation}  \label{e:lower.bdd.Gkn.rho}
\liminf_{n\to\infty} \frac{G_{k,n}^+(t)}{\rho_n} \ge \frac{\mu_k^\ktop(t; 0)}{(k+2)!} + \liminf_{m\to\infty} \frac{S_m^\downarrow(t)-\E[S_m^\downarrow(t)]}{v_{m}^{k+2}R_{p_m}^df(R_{p_m})^{k+2}}, \ \ \text{a.s.}
\end{equation}
From \eqref{e:upper.Gkn.rhon} and \eqref{e:lower.bdd.Gkn.rho}, it now suffices to demonstrate that, for every $\epsilon>0$, 
\begin{align*}
\limsup_{m\to\infty} \frac{S_m^\uparrow(t)-\E[S_m^\uparrow(t)]}{v_{m+1}^{k+2}R_{q_m}^df(R_{q_m})^{k+2}} &\le \epsilon, \ \ \text{a.s.,} \ \ \text{and }   \ \ \ \liminf_{m\to\infty} \frac{S_m^\downarrow(t)-\E[S_m^\downarrow(t)]}{v_{m}^{k+2}R_{p_m}^df(R_{p_m})^{k+2}} \ge -\epsilon, \ \ \text{a.s.} 
\end{align*}
According to the Borel-Cantelli lemma, we need to show that 
\begin{align}
&\sum_{m=1}^\infty \P \big( S_m^\uparrow(t) >  \E[S_m^\uparrow(t)] + \epsilon v_{m+1}^{k+2} R_{q_m}^d f(R_{q_m})^{k+2} \big) < \infty,  \label{e:BC.upper} \\
&\sum_{m=1}^\infty \P \big( S_m^\downarrow(t) <   \E[S_m^\downarrow(t)] - \epsilon v_{m}^{k+2} R_{p_m}^d f(R_{p_m})^{k+2}   \big) < \infty.  \label{e:BC.lower}
\end{align}
For the proof of \eqref{e:BC.upper},  note that $S_m^\uparrow(t)$ is a Poisson $U$-statistics of order $k+2$, satisfying \eqref{e:indicator.condition1} and \eqref{e:indicator.condition2}, for which an underlying Poisson point process has a finite intensity measure $v_{m+1}\int_{\cdot \, \cap B(0,R_{q_m})^c}f(z)dz$.  From this observation, one can appeal to Proposition \ref{p:concentration} to get that 
\begin{align*}
&\P \big( S_m^\uparrow(t) >  \E[S_m^\uparrow(t)] + \epsilon v_{m+1}^{k+2} R_{q_m}^d f(R_{q_m})^{k+2} \big) \\
&\le \exp \bigg\{ -C^* \Big[ \Big(  \E[S_m^\uparrow(t)]  + \epsilon  v_{m+1}^{k+2} R_{q_m}^d f(R_{q_m})^{k+2}   \Big)^{1/(2k+4)} - \big( \E[S_m^\uparrow(t)] \big)^{1/(2k+4)}  \Big]^2  \bigg\}. 
\end{align*}
By Lemmas \ref{l:RV.seq.heavy} and \ref{l:first.second.asymptotics}, and \eqref{e:cond:rhon}, 
\begin{align*}
&\Big[ \Big(  \E[S_m^\uparrow(t)]  + \epsilon  v_{m+1}^{k+2} R_{q_m}^d f(R_{q_m})^{k+2}   \Big)^{1/(2k+4)} - \big( \E[S_m^\uparrow(t)] \big)^{1/(2k+4)}  \Big]^2\\
&\ge C^* \big( v_{m+1}^{k+2} R_{q_m}^d f(R_{q_m})^{k+2} \big)^{1/(k+2)} \ge C^* \big( v_m^{k+2} R_{v_m}^d f(R_{v_m})^{k+2} \big)^{1/(k+2)} \\
&\ge C^*(\log v_m)^{\eta/(k+2)} \ge C^* m^{\gamma \eta / (k+2)}. 
\end{align*}
In conclusion, 
$$
\P \big( S_m^\uparrow(t) >  \E[S_m^\uparrow(t)] + \epsilon v_{m+1}^{k+2} R_{q_m}^d f(R_{q_m})^{k+2} \big)  \le \exp\big\{ -C^*m^{\gamma \eta/(k+2)} \big\}, 
$$
so that $\sum_m\exp\{ -C^*m^{\gamma \eta/(k+2)} \} <\infty$ as desired. 

We proceed to the proof of \eqref{e:BC.lower}. 
Applying the second concentration inequality in Proposition \ref{p:concentration}, it holds that 
\begin{align*}
&\P \big( S_m^\downarrow(t) <   \E[S_m^\downarrow(t)] - \epsilon v_{m}^{k+2} R_{p_m}^d f(R_{p_m})^{k+2}   \big) \le \exp \bigg\{ - \frac{C^*\big[  \epsilon v_{m}^{k+2}R_{p_m}^df(R_{p_m})^{k+2}  \big]^2}{ \text{Var}(S_m^\downarrow(t)) }  \bigg\}. 
\end{align*}
Because of Lemmas \ref{l:RV.seq.heavy} and \ref{l:first.second.asymptotics}, as well as \eqref{e:cond:rhon}, 
\begin{align*}
\big(v_{m}^{k+2}R_{p_m}^df(R_{p_m})^{k+2}  \big)^2 \big( \text{Var}(S_m^\downarrow(t)) \big)^{-1} &\ge C^* v_{m}^{k+2}R_{p_m}^df(R_{p_m})^{k+2}  \ge C^* v_m^{k+2} R_{v_m}^d f(R_{v_m})^{k+2} \\
&\ge C^* (\log v_m)^\eta \ge C^*m^{\gamma \eta}. 
\end{align*}
We now conclude that 
\begin{align*}
\P \big( S_m^\downarrow(t) <   \E[S_m^\downarrow(t)] - \epsilon v_{m}^{k+2} R_{p_m}^d f(R_{p_m})^{k+2}   \big) \le \exp\big\{ -C^*m^{\gamma \eta} \big\}, 
\end{align*}
so that $\sum_m e^{-C^*m^{\gamma \eta}}<\infty$. This completes the proof of \eqref{e:BC.lower}, and thus, \eqref{e:suff.cond.crackle1.heavy} follows. 

The next step is to establish \eqref{e:suff.cond.crackle2.heavy}. 
By virtue of Lemma \ref{l:upper.lower.betti}, we obtain that 
\begin{equation}  \label{e:bound.betti.Gkn}
\big| \beta_{k,n}(t) -\Gknt  \big| \le \Gknt - J_{k,n}^{(k+2,1)} (t)+ \binom{k+3}{k+1} L_{k,n}(t), 
\end{equation}
where 
$$
L_{k,n}(t):= \sum_{\substack{\Y\subset \X_n, \\ |\Y|=k+3}} \one \big\{  \cc (\Y, t) \text{ is connected}, \, \M(\Y)\ge R_n \big\}. 
$$
One can further bound the right hand side in \eqref{e:bound.betti.Gkn} as follows: 
\begin{align*}
&\Gknt -J_{k,n}^{(k+2,1)}(t) + \binom{k+3}{k+1} L_{k,n}(t) \\
&\le \sum_{\Y\subset \X_n, \, |\Y|=k+2} \one \big\{ \C(\Y, t) \text{ is connected}, \, \M(\Y)\ge R_n \big\}    \\
&\qquad \qquad \qquad \times \one \big\{ \C(\Y, t) \text{ is not a connected component of } \C(\X_n, t) \big\}  + \binom{k+3}{k+1} L_{k,n}(t)   \\
&\le \bigg( k+3 + \binom{k+3}{k+1} \bigg) L_{k,n}(t) \le 10k^2 L_{k,n}(t).  
\end{align*}
Observing that $\X_n \subset \X_{v_{m+1}}$, $R_n\ge R_{q_m}$ and $\one \big\{ \C(\Y, t) \text{ is connected} \big\}$ is a non-decreasing function in $t$, 
\begin{align*}
 L_{k,n}(t) &\le \sum_{\substack{\Y\subset \X_{v_{m+1}}, \\ |\Y|=k+3}} \one \big\{ \C(\Y, 1) \text{ is connected}, \, \M(\Y)\ge R_{q_m} \big\}. 
\end{align*}
Since $\rho_n \ge v_m^{k+2}R_{s_m}^d f(R_{s_m})^{k+2}$, we have, for every $n\ge 1$, 
\begin{align*}
\basicsup \Big|\, \frac{\beta_{k,n}(t)}{\rho_n} - \frac{G_{k,n}(t)}{\rho_n} \, \Big|
\le \frac{10k^2}{v_m^{k+2}R_{s_m}^d f(R_{s_m})^{k+2}}  \sum_{\substack{\Y\subset \X_{v_{m+1}}, \\ |\Y|=k+3}} \one \big\{ \C(\Y, 1) \text{ is connected}, \, \M(\Y)\ge R_{q_m} \big\}. 
\end{align*}
We now define
\begin{equation}  \label{e:def.Wm}
W_m(t) := \sum_{\substack{\Y \subset \mathcal P_{v_{m+1}}, \\  |\Y|=k+3}} \one \big\{ \C(\Y, t) \text{ is connected}, \, \M(\Y)\ge R_{q_m} \big\}, \ \ t\in [0,1]. 
\end{equation}
This is a Poisson $U$-statistics of order $k+3$, such that $s(\Y) := \one \big\{ \C(\Y, t) \text{ is connected} \big\}$ fulfills conditions \eqref{e:indicator.condition1} and \eqref{e:indicator.condition2}. 
Then, Lemmas \ref{l:RV.seq.heavy} and \ref{l:difference.binom.poisson} guarantee that the proof of \eqref{e:suff.cond.crackle2.heavy} will be complete if one can verify that
\begin{equation}  \label{e:cond.a}
\frac{W_m(1)}{v_{m+1}^{k+2}R_{q_m}^d f(R_{q_m})^{k+2}} \to 0, \ \  \text{a.s.}
\end{equation}
By the Borel-Cantelli lemma, \eqref{e:cond.a} is implied by 
$$
\sum_{m=1}^\infty \P\big( W_m(1) > \epsilon v_{m+1}^{k+2} R_{q_m}^d f(R_{q_m})^{k+2} \big) < \infty, 
$$  
for every $\epsilon>0$. 
Lemmas \ref{l:RV.seq.heavy} and \ref{l:first.second.asymptotics} yield that as $m\to\infty$, 
\begin{align}  
\E\big[ W_m(1) \big] &\le C^*v_{m+1}^{k+3} R_{q_m}^d f(R_{q_m})^{k+3}  \le C^* v_{m+1}^{k+2} R_{q_m}^d f(R_{q_m})^{k+2} v_m f(R_{v_m}) \label{e:EVmiT} \\
&=o \big( v_{m+1}^{k+2} R_{q_m}^d f(R_{q_m})^{k+2}  \big), \notag
\end{align}
where the last equality comes from $v_mf(R_{v_m}) \to 0$ as $m\to\infty$. Now, there exists $N\in \N$, such that for all $m\ge N$, 
\begin{align*}
&\P\big( W_m(1) > \epsilon v_{m+1}^{k+2} R_{q_m}^d f(R_{q_m})^{k+2} \big) \\
&\le \P \Big(W_m(1) -\E\big[ W_m(1)  \big] > \frac{\epsilon}{2}\,  v_{m+1}^{k+2} R_{q_m}^d f(R_{q_m})^{k+2}  \Big) \\
&\le \exp \bigg\{  -C^* \Big[ \Big( \E\big[W_m(1)\big] + \frac{\epsilon}{2}\,  v_{m+1}^{k+2} R_{q_m}^d f(R_{q_m})^{k+2}  \Big)^{1/(2k+6)} - \Big(  \E\big[ W_m(1) \big]\Big)^{1/(2k+6)} \Big]^2 \bigg\}. 
\end{align*}
The last inequality above is a direct result of Proposition \ref{p:concentration}. 
It follows from \eqref{e:EVmiT} that 
$$
\frac{\E\big[ W_m(1) \big]}{v_{m+1}^{k+2}R_{q_m}^d f(R_{q_m})^{k+2}} \to 0, \ \ m\to\infty. 
$$
This, together with Lemma \ref{l:RV.seq.heavy} and \eqref{e:cond:rhon}, gives that 
\begin{align*}
&\Big[ \Big( \E\big[W_m(1)\big] + \frac{\epsilon}{2}\,  v_{m+1}^{k+2} R_{q_m}^d f(R_{q_m})^{k+2}  \Big)^{1/(2k+6)} - \Big(  \E\big[ W_m(1) \big]\Big)^{1/(2k+6)} \Big]^2  \\
&\ge C^*\big( v_{m+1}^{k+2} R_{q_m}^d f(R_{q_m})^{k+2} \big)^{1/(k+3)} \ge C^* \big( v_m^{k+2} R_{v_m}^d f(R_{v_m})^{k+2} \big)^{1/(k+3)} \\
&\ge C^*(\log v_m)^{\eta/(k+3)} \ge C^*m^{\gamma \eta/(k+3)}. 
\end{align*}
Now we get that 
$$
\sum_{m=1}^\infty  \P\big( W_m(1) > \epsilon v_{m+1}^{k+2} R_{q_m}^d f(R_{q_m})^{k+2} \big)  \le C^*\sum_{m=1}^\infty \exp \big\{ -C^*m^{\gamma \eta/(k+3)} \big\} < \infty, 
$$
and the Borel-Cantelli lemma concludes \eqref{e:cond.a}, as desired. 
\end{proof}

\subsection{Proof of Theorem \ref{t:fslln.expdt}}   \label{sec:proof.fslln.exp}

This section is divided into two parts. For ease of the description, the first part is devoted to proving Theorem \ref{t:fslln.expdt} $(ii)$, while the second part treats Theorem \ref{t:fslln.expdt} $(i)$.  Here we exploit the results in Sections \ref{sec:ingredients} and \ref{sec:technical.lemma.exp}.  In particular, the concentration bounds in Proposition \ref{p:concentration} will play a key role in the proof of Theorem \ref{t:fslln.expdt} $(i)$.
As our proof  is similar in nature to that of Theorem \ref{t:fslln.hvt}, we occasionally skip detailed arguments. 

\begin{proof}[Proof of Theorem \ref{t:fslln.expdt} $(ii)$]
We first truncate $\beta_{k,n}(t)$ in the same way as \eqref{e:truncated.betti.M} and define also $\xi_k^{(M)}(t; \lambda)$ by the same truncation. With the same reasoning as in the proof of Theorem \ref{t:fslln.hvt} $(ii)$,  
 the required functional SLLN can be obtained as a result of the following statements. 
\begin{align}
&\lim_{n\to\infty} \sup_{0\le t \le 1}  \Big| \frac{\beta_{k,n}^{(M)}(t)}{a(R_n)R_n^{d-1}} - \xi_k^{(M)}(t; \lambda) \Big| = 0, \ \ \text{a.s. for all } M\in \N,  \label{e:suff.cond.exp1}\\
&\lim_{M\to\infty} \limsup_{n\to\infty} \sup_{0\leq t\leq 1}\Big\{ \frac{\beta_{k,n}(t)}{a(R_n)R_n^{d-1}} - \frac{\beta^{(M)}_{k,n}(t)}{a(R_n)R_n^{d-1}}\Big\} =0, \ \ \text{a.s.}, \label{e:suff.cond.exp2}\\
&\lim_{M\to\infty} \limsup_{n\to\infty} \sup_{0\le t \le 1} \big( \xi_k(t; \lambda) - \xi_k^{(M)}(t; \lambda) \big) =0. \label{e:suff.cond.exp3}
\end{align}
For \eqref{e:suff.cond.exp1}, it suffices to prove that for every $i\ge k+2$ and $j\ge1$, 
\begin{equation*} 
\basicsup \Big|\, \frac{J_{k,n}^\icj (t)}{a(R_n)R_n^{d-1}} -\frac{\lambda^i}{i!}\, \xi_k^\icj (t; \lambda)  \, \Big| \to 0, \ \ n\to\infty, \ \ \text{a.s.}, 
\end{equation*}
which itself is  implied by 
\begin{align}
\basicsup \Big|\, \frac{J_{k,n}^\ijp (t)}{a(R_n)R_n^{d-1}} -\frac{\lambda^i}{i!}\, \xi_k^\ijp (t; \lambda)  \, \Big| \to 0, \ \ \text{a.s.},   \label{e:comp.wise.fslln} \\
\basicsup \Big|\, \frac{J_{k,n}^{(i,j,-)} (t)}{a(R_n)R_n^{d-1}} -\frac{\lambda^i}{i!}\, \xi_k^{(i,j,-)} (t; \lambda)  \, \Big| \to 0, \ \ \text{a.s.}, \notag
\end{align}
where $J_{k,n}^{(i,j,\pm)}(t)$ and $\xi_k^{(i,j,\pm)}(t; \lambda)$ are defined in \eqref{e:def.Jknijpm} and \eqref{e:def.xikijpm} respectively. As before, we focus only on the asymptotics of the ``$+$" part. For this purpose, we again extend $J_{k,n}^\ijp (t)$ above to $J_{k,n}^\ijp (t,s)$ in the same way as \eqref{e:def.Jknijpts}. Additionally, we also define 
\begin{align}
\xi_k^{(i,j, +)}(t, s; \lambda):=&\int_{S^{d-1}}\int_0^\infty \int_{(\R^d)^{i-1}}h_t^\ijp(0,\by)\, e^{-\rho i -c^{-1}\sum_{\ell=1}^{i-1}\langle\theta,y_\ell \rangle} \prod_{\ell=1}^{i-1} \one \big\{ \rho+c^{-1}\langle \theta, y_\ell\rangle \ge 0 \big\} \label{e:def.xikijpts} \\
&\quad \times\exp\Big\{-\lambda e^{-\rho}\int_{\B(\{0,\by\};  s)} e^{-c^{-1} \langle \theta, z \rangle}\,  dz \Big\}d\by\, d\rho\, J(\theta)\, d\theta, \ \ \ t,s \in [0,1]. \notag
\end{align}
Owing to the monotonicity in \eqref{e:monotonicity.ind}, one can see that \eqref{e:def.xikijpts} is non-decreasing in $t$ and non-increasing in $s$. Additionally, \eqref{e:def.xikijpts} is a continuous function on $[0,1]^2$. Thus, by Lemma \ref{l:pointwise.to.funcl} $(ii)$, showing 
\begin{equation}  \label{e:pointwise.SLLN.two.exp}
\frac{J_{k,n}^\ijp (t, s)}{a(R_n)R_n^{d-1}} \to \frac{\lambda^i}{i!}\, \xi_k^\ijp (t, s; \lambda), \ \ \ n\to\infty, \ \ \text{a.s.}
\end{equation}
for every $(t,s)\in [0,1]^2$, will suffice for \eqref{e:comp.wise.fslln}.  

To show \eqref{e:pointwise.SLLN.two.exp}, we take a constant $\gamma\in \big( \tau/(d-\tau),1\big)$ and define $v_m:= \lfloor e^{m^\gamma} \rfloor$, $m=0,1,2,\dots$ as in \eqref{e:def.vm}. Recall that the range of $\tau$ is restricted to $(0,1)$ when $d=2$, so one can always take such $\gamma$. For every $n\in \N$, there exists a unique $m=m(n)\in \N$ such that $v_m \le n <v_{m+1}$. 
Additionally, let $p_m$, $q_m$, $b_m$, and $c_m$ be defined as in \eqref{e:def.pmqm} and \eqref{e:def.bmcm} respectively. 
Defining $T_m^\ijupa (t,s)$ and $T_m^\ijdoa (t,s)$  as in \eqref{e:def.upbd} and \eqref{e:def.lobd}, it is now easy to see that 
$$
\frac{T_m^\ijdoa(t,s)}{a(R_{b_m})R_{b_m}^{d-1}} \le \frac{J_{k,n}^\ijp(t,s)}{a(R_n)R_n^{d-1}} \le \frac{T_m^\ijupa(t,s)}{a(R_{c_m})R_{c_m}^{d-1}} 
$$
for all $n\in \N$. By Lemmas \ref{l:RV.seq.exp1} and \ref{l:first.second.moments.exp1} $(i)$, 
\begin{align}
\limsup_{n\to\infty}\frac{J_{k,n}^\ijp(t,s)}{a(R_n)R_n^{d-1}} &\le \limsup_{m\to\infty} \frac{a(R_{q_m})R_{q_m}^{d-1}}{a(R_{c_m})R_{c_m}^{d-1}} \cdot \frac{T_m^\ijupa (t,s)}{a(R_{q_m})R_{q_m}^{d-1}}  \label{e:rightmost.limsup.Jknijp} \\
&\le \limsup_{m\to\infty} \frac{T_m^\ijupa (t,s)}{a(R_{q_m})R_{q_m}^{d-1}} \notag \\
&\le \frac{\lambda^i}{i!}\, \xi_k^\ijp (t,s; \lambda) + \limsup_{m\to\infty} \frac{T_m^\ijupa (t,s)-\E\big[ T_m^\ijupa (t,s) \big]}{a(R_{q_m})R_{q_m}^{d-1}}, \ \ \text{a.s.},\notag 
\end{align}
and similarly, 
\begin{equation} \label{e:lower.bdd.Jknijp}
\liminf_{n\to\infty}\frac{J_{k,n}^\ijp(t,s)}{a(R_n)R_n^{d-1}} \ge \frac{\lambda^i}{i!}\, \xi_k^\ijp (t,s; \lambda)+ \liminf_{m\to\infty} \frac{T_m^\ijdoa (t,s)-\E\big[ T_m^\ijdoa (t,s) \big]}{a(R_{p_m})R_{p_m}^{d-1}}, \ \ \text{a.s.}
\end{equation}
For the application of the Borel-Cantelli lemma to the rightmost term in \eqref{e:rightmost.limsup.Jknijp}, we have, for every $\epsilon>0$, 
\begin{align*}
&\sum_{m=1}^\infty \P \Big( \big| \,T_m^\ijupa (t,s)-\E\big[ T_m^\ijupa (t,s) \big]  \, \big| >\epsilon a(R_{q_m})R_{q_m}^{d-1} \Big) \\
&\le \frac{1}{\epsilon^2} \sum_{m=1}^\infty \frac{\text{Var}\big( T_m^\ijupa (t,s) \big)}{\big( a(R_{q_m})R_{q_m}^{d-1} \big)^2}  \le C^* \sum_{m=1}^\infty \frac{1}{a(R_{q_m})R_{q_m}^{d-1}}, 
\end{align*}
where Lemma \ref{l:first.second.moments.exp1} $(i)$ is applied for the second inequality. 
Due to the constraint $\gamma\in \big(\tau/(d-\tau), 1  \big)$, one can choose $\delta_i>0$, $i=1,2$, so that 
$$
\gamma (d-\tau-\delta_1) \Big( \frac{1}{\tau} -\delta_2 \Big) > 1. 
$$
Observing that $a\in \text{RV}_{1-\tau}$, we have $a(R_{q_m})R_{q_m}^{d-1}\ge C^* R_{q_m}^{d-\tau-\delta_1}$ for all $m\in \N$. Moreover, it follows from $\psi^\leftarrow \in \text{RV}_{1/\tau}$ (see Proposition 2.6 in \cite{resnick:2007}) and \eqref{e:Rn.psi.inv} that 
$$
R_{q_m} \sim \psi^\leftarrow (\log q_m) \ge \psi^\leftarrow (\log v_m) \ge C^* (\log v_m)^{1/\tau - \delta_2/2} \ge C^* m^{\gamma (1/\tau -\delta_2)}. 
$$
Therefore, 
$$
a(R_{q_m}) R_{q_m}^{d-1} \ge C^* m^{\gamma (d-\tau -\delta_1)(1/\tau -\delta_2)}, 
$$
so that 
\begin{equation}  \label{e:aRqm.Rqmd}
\sum_{m=1}^\infty \frac{1}{a(R_{q_m})R_{q_m}^{d-1}} \le C^* \sum_{m=1}^\infty \frac{1}{m^{\gamma (d-\tau -\delta_1)(1/\tau -\delta_2)}} <\infty. 
\end{equation}
Now, the Borel-Cantelli lemma verifies that 
$$
\frac{T_m^\ijupa (t,s)-\E\big[ T_m^\ijupa (t,s) \big]}{a(R_{q_m})R_{q_m}^{d-1}} \to 0, \ \ m\to\infty, \ \ \text{a.s.}, 
$$
and hence, 
$$
\limsup_{n\to\infty}\frac{J_{k,n}^\ijp(t,s)}{a(R_n)R_n^{d-1}} \le \frac{\lambda^i}{i!}\, \xi_k^\ijp (t,s; \lambda),  \ \ \text{a.s.}
$$
Applying the similar argument to \eqref{e:lower.bdd.Jknijp}, we also get that 
$$
\liminf_{n\to\infty} \frac{J_{k,n}^\ijp(t,s)}{a(R_n)R_n^{d-1}} \ge \frac{\lambda^i}{i!}\, \xi_k^\ijp (t,s; \lambda),  \ \ \text{a.s.}; 
$$
hence, the proof of \eqref{e:suff.cond.exp1} has been completed. 

Our next task is to verify \eqref{e:suff.cond.exp2}.  Repeating the same analysis as in \eqref{e:Wmmt}, while using the same notation $V_{m,M}(t)$, we obtain that 
\begin{align}
\frac{\beta_{k,n}(t)-\beta^{(M)}_{k,n}(t)}{a(R_n)R_n^{d-1}} &\le \frac{1}{a(R_n)R_n^{d-1}}\sum_{i=M+1}^n i^{k+1}\sum_{\Y\subset\X_n, \lvert\Y\rvert=i}\ind{\cc(\Y, t) \text{ is connected},\, \M(\Y)\geq R_n} \notag \\ 
&\le \frac{1}{a(R_{c_m})R_{c_m}^{d-1}}\sum_{i=M+1}^\infty i^{k+1}\hspace{-10pt} \sum_{\Y\subset\X_{v_{m+1}}, \lvert\Y\rvert=i}\hspace{-10pt}\ind{\cc(\Y, t) \text{ is connected},\, \M(\Y)\geq R_{q_m}} \notag \\
&=:\frac{V_{m,M}(t)}{a(R_{c_m})R_{c_m}^{d-1} }. \notag
\end{align}
Since $V_{m,M}(t)$ is a non-decreasing function in $t$, we now need to show that 
$$
\lim_{M\to\infty} \limsup_{m\to\infty} \frac{V_{m,M}(1)}{a(R_{c_m})R_{c_m}^{d-1}} = 0, \ \ \text{a.s.}
$$
By Lemmas \ref{l:RV.seq.exp1} and \ref{l:first.second.moments.exp1} $(ii)$, it is enough to prove that for every $M\in \N$, 
$$
\frac{V_{m,M}(1)-\E\big[ V_{m,M}(1) \big]}{a(R_{q_m})R_{q_m}^{d-1}} \to 0, \ \ \ m\to\infty, \ \ \text{a.s.}
$$
To apply the Borel-Cantelli lemma, notice from Lemma \ref{l:first.second.moments.exp1} $(ii)$ that for every $\epsilon>0$, 
\begin{align*}
&\sum_{m=1}^\infty \P \Big(\big|\,  V_{m,M}(1)-\E[V_{m,M}(1)] \, \big| >\epsilon a(R_{q_m}) R_{q_m}^{d-1}  \Big) \\
&\le \frac{1}{\epsilon^2} \sum_{m=1}^\infty \frac{\text{Var}\big( V_{m,M}(1) \big)}{\big( a(R_{q_m})R_{q_m}^{d-1} \big)^2} \le C^* \sum_{m=1}^\infty \frac{1}{a(R_{q_m})R_{q_m}^{d-1}}. 
\end{align*}
The series above is convergent as shown in \eqref{e:aRqm.Rqmd}, and thus, \eqref{e:suff.cond.exp2} follows as desired. Finally, as in the proof of Theorem \ref{t:fslln.hvt} $(ii)$, we will skip the proof of \eqref{e:suff.cond.exp3}. 
\end{proof}
\bigskip

\begin{proof}[Proof of Theorem \ref{t:fslln.expdt} $(i)$]
By the same reasoning as in the proof of Theorem \ref{t:fslln.hvt} $(i)$, it is sufficient to demonstrate that 
\begin{align}
&\basicsup \Big| \, \frac{G_{k,n}(t)}{\eta_n} - \frac{\xi_k^{(k+2,1)}(t; 0)}{(k+2)!} \, \Big| \to 0, \ \ \ n\to\infty, \ \ \text{a.s.}, \label{e:suff.cond.crackle1.exp}\\
&\basicsup \Big| \, \frac{\beta_{k,n}(t)}{\eta_n} - \frac{G_{k,n}(t)}{\eta_n}  \, \Big| \to 0, \ \ \ n\to\infty, \ \ \text{a.s.}, \label{e:suff.cond.crackle2.exp}
\end{align}
where $G_{k,n}(t)$ is defined in \eqref{e:def.Gkn}. For \eqref{e:suff.cond.crackle1.exp}, we decompose $G_{k,n}(t)$ and $\xi_k^{(k+2,1)}(t; 0)$ as $G_{k,n}(t)=G_{k,n}^+(t)-G_{k,n}^-(t)$ and $\xi_k^{(k+2,1)}(t; 0) = \xi_k^{(k+2,1, +)}(t; 0) - \xi_k^{(k+2,1, -)}(t; 0)$ (see \eqref{e:def.Gknpmt} and \eqref{e:def.xikijpm} respectively). We discuss only the asymptotics of the ``$+$" part. We then recall that $G_{k,n}^+(t)$ and $\xi_k^{(k+2,1, +)}(t; 0)$ are both non-decreasing in $t$, and $\xi_k^{(k+2,1, +)}(t; 0)$ is a continuous function in $t\in [0,1]$. 
Now, because of Lemma \ref{l:pointwise.to.funcl} $(i)$, what needs to be shown is 
$$
\frac{G_{k,n}^+(t)}{\eta_n} \to \frac{\xi_k^{(k+2,1,+)}(t; 0)}{(k+2)!}\, \ \ \ n\to\infty, \ \ \text{a.s.}, 
$$
for every $t\in [0,1]$. 
Let $(v_m)$ be a sequence at \eqref{e:def.vm}. Then, for every $n\in \N$, there exists a unique $m=m(n)\in \N$ such that $v_m \le n < v_{m+1}$. 
Using $S_m^\uparrow (t)$ and $S_m^\downarrow (t)$ in \eqref{e:def.Smup} and \eqref{e:def.Smdown}, as well as the associated Poisson processes $\mathcal P_{v_m}$, $\mathcal P_{v_{m+1}}$, we shall derive the subsequential upper and lower bounds for $G_{k,n}^+(t)/\eta_n$. For the upper bound, as an analogue of \eqref{e:initial.upper.bdd.Gknrhon} we obtain that 
\begin{align*} 
\frac{G_{k,n}^+(t)}{\eta_n}  &\le \big( v_{m}^{k+2}a(R_{g_m})R_{g_m}^{d-1}f(R_{g_m})^{k+2} \big)^{-1} \sum_{\substack{\Y\subset \X_{v_{m+1}}, \\ |\Y|=k+2}} h_t^\ktop(\Y)\, \one \big\{ \M(\Y) \ge R_{q_m} \big\} \\
&= \frac{S_m^\uparrow(t)}{v_{m}^{k+2}a(R_{g_m})R_{g_m}^{d-1}f(R_{g_m})^{k+2}} +  \big( v_{m}^{k+2}a(R_{g_m})R_{g_m}^{d-1}f(R_{g_m})^{k+2}\big)^{-1} \\
&\qquad \qquad \qquad \qquad \qquad \times \bigg\{  \sum_{\substack{\Y\subset \X_{v_{m+1}}, \\ |\Y|=k+2}} h_t^\ktop(\Y)\, \one \big\{ \M(\Y) \ge R_{q_m} \big\}  -S_m^\uparrow(t)\bigg\}, \ \ n\ge 1, 
\end{align*}  
where $q_m$ and $g_m$ are defined in \eqref{e:def.pmqm} and \eqref{e:def.emgm} respectively.  
Repeating the calculations similar to those in \eqref{e:diff.binom.poisson} and \eqref{e:upper.Gkn.rhon}, while appealing to Lemmas \ref{l:RV.seq.exp1}, \ref{l:first.second.moments.exp2}, and \ref{l:difference.binom.poisson.exp}, we get that 
$$
\limsup_{n\to\infty}\frac{G_{k,n}^+(t)}{\eta_n} \le  \frac{\xi_k^\ktop(t; 0)}{(k+2)!} + \limsup_{m\to\infty} \frac{S_m^\uparrow(t)-\E[S_m^\uparrow(t)]}{v_{m+1}^{k+2}a(R_{q_m})R_{q_m}^{d-1}f(R_{q_m})^{k+2}}, \ \ \text{a.s.} 
$$
Similarly, exploiting the same lemmas, we have that 
$$
\liminf_{n\to\infty}\frac{G_{k,n}^+(t)}{\eta_n} \ge  \frac{\xi_k^\ktop(t; 0)}{(k+2)!} + \liminf_{m\to\infty} \frac{S_m^\downarrow(t)-\E[S_m^\downarrow(t)]}{v_{m}^{k+2}a(R_{p_m})R_{p_m}^{d-1}f(R_{p_m})^{k+2}}, \ \ \text{a.s.} 
$$
Now, according to the Borel-Cantelli lemma, what need to be shown are the following: for every $\epsilon>0$, 
\begin{align}
&\sum_{m=1}^\infty \P \big( S_m^\uparrow(t) >  \E[S_m^\uparrow(t)] + \epsilon v_{m+1}^{k+2} a(R_{q_m})R_{q_m}^{d-1} f(R_{q_m})^{k+2} \big) < \infty,   \label{e:BC.exp1}\\
&\sum_{m=1}^\infty \P \big( S_m^\downarrow(t) <   \E[S_m^\downarrow(t)] - \epsilon v_{m}^{k+2} a(R_{p_m})R_{p_m}^{d-1} f(R_{p_m})^{k+2}   \big) < \infty.\label{e:BC.exp2}
\end{align}
Here, we claim that 
\begin{equation}  \label{e:normalizing.faster.than.log}
n^{k+2}a(R_n)R_n^{d-1}f(R_n)^{k+2} =\Omega \big( (\log n)^\zeta \big), \ \ \ n\to\infty, 
\end{equation}
for some $\zeta>0$. If this is proven, one can establish \eqref{e:BC.exp1} and \eqref{e:BC.exp2} by repeating the same arguments as those for the proofs of \eqref{e:BC.upper} and \eqref{e:BC.lower}, with the aid of  Lemma \ref{l:RV.seq.exp1}, Lemma   \ref{l:first.second.moments.exp2}, and Proposition \ref{p:concentration}. To show \eqref{e:normalizing.faster.than.log}, choose $\zeta'>0$ such that 
$$
\frac{d-\tau}{\tau} -b(k+2) -\zeta' >0. 
$$
Notice that 
\begin{align*}
\eta_n&= C^{k+2} (\log n)^{-b(k+2)} a\big(\psi^\leftarrow (\log n + b\log \log n)\big)\psi^\leftarrow (\log n +b\log \log n)^{d-1}. 
\end{align*}
Since $a\in \text{RV}_{1-\tau}$ and $\psi^\leftarrow \in \text{RV}_{1/\tau}$, we see that $a\big( \psi^\leftarrow(z) \big)\psi^\leftarrow(z)^{d-1}$ is a regularly varying function (at infinity) of exponent $(d-\tau)/\tau$. Therefore, for all $n\in \N$, 
\begin{align*}
&a\big(\psi^\leftarrow (\log n + b\log \log n)\big)\psi^\leftarrow (\log n +b\log \log n)^{d-1} \\
&\ge C^* (\log n + b\log \log n)^{\frac{d-\tau}{\tau}-\zeta'} \ge C^* (\log n)^{\frac{d-\tau}{\tau}-\zeta'}. 
\end{align*}
From this, we get that $\eta_n \ge C^*  (\log n)^{\frac{d-\tau}{\tau}-b(k+2)-\zeta'}$. Now, \eqref{e:normalizing.faster.than.log} has been obtained by setting $\zeta=(d-\tau)/\tau -b(k+2)-\zeta'$, and the proof of \eqref{e:BC.exp1} and \eqref{e:BC.exp2} has been completed. 

Finally, by virtue of Lemma \ref{l:upper.lower.betti}, the proof of \eqref{e:suff.cond.crackle2.exp} is very analogous to the corresponding discussions in Section \ref{sec:proof.fslln.hvt}, so we  omit it here. 
\end{proof}

\section{Appendix}
In the Appendix, we provide a series of lemmas and propositions that will be used for the proofs of Theorems \ref{t:fslln.hvt} and \ref{t:fslln.expdt}. 
As in the last section, $C^*$ denotes a generic and positive constant independent of $n$. 

\subsection{Technical results commonly used for the proof of Theorems \ref{t:fslln.hvt} and \ref{t:fslln.expdt}} \label{sec:ingredients}

The result below allows us to develop a functional SLLN from its pointwise version. 

\begin{lemma} \label{l:pointwise.to.funcl}
$(i)$ Let $\big( X_n(t), \, n \ge 1 \big)$ be a sequence of random elements of $D[0,1]$ with non-decreasing sample paths. Suppose $a:[0,1]\to \R$ is a continuous and non-decreasing function. Suppose that 
$$
X_n(t) \to a(t), \ \ n\to\infty, \ \ \text{a.s.},
$$
for every  $t\in [0,1]$, then 
$$
X_n(t) \to a(t), \ \ n\to\infty,  \ \ \text{a.s.~in } D[0,1], 
$$
where $D[0,1]$ is endowed with the uniform topology. 
\vspace{3pt}

\noindent $(ii)$ Let $\big(X_n(t,s), \, n\in\N\big)$ be a sequence of random elements, such that for each $n \ge 1$, $X_n(t,s)$ has right continuous sample paths with left limits in each of the coordinates. Assume further that for every $n \ge 1$, $X_n(t,s)$ is non-decreasing in $t$ and non-increasing in $s$. Suppose  $a(t,s)$ is a real-valued, continuous function on $[0,1]^2$, which has the same monotonicity as $X_n(t,s)$ in each of the coordinates. If we have that 
\begin{equation}\label{e:condprop}
X_n(t,s) \to a(t,s), \quad n \to \infty, \quad \text{a.s.}
\end{equation}
for every $t,s \in [0,1]$, then, as $n\to\infty$, 
$$
X_n(t,t)\to a(t,t) \quad \text{a.s.} \ \text{in }\ D[0,1],
$$
where $D[0,1]$ is equipped with the uniform topology.
\end{lemma}

\begin{proof}
Part $(i)$ is proven in Proposition 4.2 of \cite{thomas:owada:2021a}. For Part $(ii)$, it is clear that $a(t,s)$ is uniformly continuous on $[0,1]^2$. Given $\epsilon >0$, we can choose $N=N(\epsilon)\in \N$ such that for all $(t_1, s_1)$, $(t_2, s_2) \in [0,1]^2$, 
\begin{equation}  \label{e:close.coordinate}
|t_1-t_2| + |s_1-s_2| \le \frac{2}{N} \  \ \text{ implies  } \ \big|  a(t_1,s_1) - a(t_2, s_2) \big| < \epsilon. 
\end{equation}
Then, we see that 
\begin{align*}
&\basicsup \big| X_n(t,t) - a(t,t) \big| \\
&\le \sup_{0 \le t, s \le 1} \big| X_n(t,s) -a(t,s) \big| \\
&=\max_{1 \le i \le N}\max_{1\le j \le N} \sup_{t \in [(i-1)/N, i/N]} \sup_{s \in [(j-1)/N, j/N]} \Big\{  \big( X_n(t,s)-a(t,s) \big) \vee \big( a(t,s)-X_n(t,s) \big) \Big\}\\
&\le \max_{1 \le i \le N}\max_{1\le j \le N} \Big\{ \Big( X_n\Big(\frac{i}{N}, \frac{j-1}{N}\Big) - a\Big(\frac{i-1}{N}, \frac{j}{N}\Big) \Big) \vee  \Big(  a\Big(\frac{i}{N}, \frac{j-1}{N}\Big) - X_n\Big(\frac{i-1}{N}, \frac{j}{N}\Big) \Big)  \Big\} \\
&\le \max_{1 \le i \le N}\max_{1\le j \le N} \Big\{ \Big( X_n\Big(\frac{i}{N}, \frac{j-1}{N}\Big) - a\Big(\frac{i}{N}, \frac{j-1}{N}\Big) \Big) \vee  \Big(  a\Big(\frac{i-1}{N}, \frac{j}{N}\Big) - X_n\Big(\frac{i-1}{N}, \frac{j}{N}\Big) \Big)  \Big\} + \epsilon. 
\end{align*}
In the above, the second inequality follows from the monotonicity of $X_n$ and $a$, and we have used \eqref{e:close.coordinate} for the third inequality. By virtue of \eqref{e:condprop}, the last expression converges to $\epsilon$ almost surely as $n \to \infty$. Since $\epsilon$ is arbitrary, the proof is complete.
\end{proof}

The next lemma provides the upper and lower bounds for $\beta_{k,n}(t)$ in \eqref{e:def.betsum}. We will make use of these bounds when estimating the difference between $\beta_{k,n}(t)$ and $G_{k,n}(t)$ (see \eqref{e:def.Gkn}) in the proof of Theorem \ref{t:fslln.hvt} $(i)$ and Theorem \ref{t:fslln.expdt} $(i)$. 
\begin{lemma} \label{l:upper.lower.betti}
For all $t \in [0,1]$, 
\begin{equation}  \label{e:upper.lower.betti.crackle}
J_{k,n}^{(k+2,1)}(t) \le \beta_{k,n}(t) \le J_{k,n}^{(k+2,1)} (t) + \binom{k+3}{k+1} L_{k,n}(t), 
\end{equation}
where 
$$
L_{k,n}(t) := \sum_{\substack{\Y\subset \X_n, \\ |\Y|=k+3}} \one \big\{  \cc (\Y, t) \text{ is connected}, \, \M(\Y)\ge R_n \big\}. 
$$
\end{lemma}

\begin{proof}
The inequality on the left hand side in \eqref{e:upper.lower.betti.crackle} is obvious due to  \eqref{e:def.betsum}. Owing to \eqref{e:def.betsum} again, the remaining inequality is equivalent to 
$$
\sum_{i=k+3}^n \sum_{j\ge 1} jJ_{k,n}^{(i,j)} (t) \le \binom{k+3}{k+1} L_{k,n}(t). 
$$
By the definition of $J_{k,n}^{(i,j)}(t)$, the left hand side above is equal to 
\begin{equation}  \label{e:equiv.repre.upper.betti}
\sum_{i=k+3}^n \sum_{\substack{\Y\subset \X_n, \\ |\Y|=i}} \beta_k \big( \cc(\Y,t) \big) \, \one \big\{ \M(\Y)\ge R_n \big\}\, \one \big\{  \cc(\Y, t) \text{ is a connected component of } \cc(\X_n, t) \big\}. 
\end{equation}
Note that $\beta_k \big( \cc(\Y,t) \big)$ is bounded by the number of $k$-simplices of $\cc(\Y,t)$. Suppose that for some $i\ge k+3$ and $\Y\subset \X_n$ with $|\Y|=i$, $\cc(\Y,t)$ is a connected component of $\cc(\X_n,t)$ with $\M(\Y)\ge R_n$.  Then, there  exists $\Z\subset \Y$ with $|\Z|=k+3$ such that $\cc(\Z,t)$ is a connected subcomplex of $\cc(\Y,t)$. Every time one finds such a connected subcomplex on $k+3$ points, it  can increase the $k$-simplex counts of $\cc(\Y,t)$ by at most $\binom{k+3}{k+1}$. In conclusion, 
$$
\beta_k\big( \cc(\Y,t) \big)\, \one \big\{ \M(\Y)\ge R_n \big\} \le \binom{k+3}{k+1} \sum_{\substack{\Z\subset \Y, \\ |\Z|=k+3}} \one \big\{ \cc(\Z, t) \text{ is connected}, \, \M(\Z)\ge R_n \big\}. 
$$
Substituting this bound back into \eqref{e:equiv.repre.upper.betti}, 
\begin{align*}
\sum_{i=k+3}^n \sum_{j\ge 1} jJ_{k,n}^{(i,j)} (t) &\le \binom{k+3}{k+1} \sum_{i=k+3}^n \sum_{\substack{\Y\subset \X_n, \\ |\Y|=i}}  \one \big\{  \cc(\Y, t) \text{ is a connected component of } \cc(\X_n, t) \big\}  \\
&\qquad \qquad \qquad \qquad \qquad \times \sum_{\substack{\Z\subset \Y, \\ |\Z|=k+3}} \one \big\{ \cc(\Z, t) \text{ is connected}, \, \M(\Z)\ge R_n \big\} \\
&= \binom{k+3}{k+1} L_{k,n}(t). 
\end{align*}
\end{proof}

The next result we introduce here is the concentration bound derived in \cite{bachmann:reitzner:2018} for a Poisson $U$-statistics. Let us rephrase the setup and assumptions of \cite{bachmann:reitzner:2018} in a way suitable for the current study. Let $\mathcal P$ denote a Poisson point process in $\R^d$ with finite intensity measure of no atoms. Let $s: (\R^d)^i \to \{ 0,1 \}$ be a \emph{symmetric} indicator function of order $i$ with the following properties. \\
$(i)$ There exists $c_1>0$ such that 
\begin{equation}  \label{e:indicator.condition1}  
s(x_1,\dots,x_i) =1 \text{ whenever diam}(x_1,\dots,x_i) < c_1.  
\end{equation}
$(ii)$ There is a constant $c_2>c_1$ such that 
\begin{equation}  \label{e:indicator.condition2}
s(x_1,\dots,x_i) = 0 \text{ whenever diam}(x_1,\dots,x_i) > c_2. 
\end{equation}
Finally, we define a \emph{Poisson $U$-statistics} $F(\mathcal P)$ of order $i$ by 
$$
F(\mathcal P)=\sum_{\Y\subset \mathcal P, \, |\Y|=i} s(\Y). 
$$
\begin{proposition} \label{p:concentration} [Theorem 3.1 in \cite{bachmann:reitzner:2018}]
Under the above conditions, there is a constant $C^*>0$, depending only on $i, d, c_1$, and $c_2$, such that for all $r>0$, 
\begin{align}
\P\big( F(\mathcal P) \ge \E[F(\mathcal P)] + r \big) &\le \exp \Big\{  -C^* \Big(  \big( \E[F(\mathcal P)]+r \big)^{1/(2i)} - \big( \E[F(\mathcal P)] \big)^{1/(2i)} \Big)^2 \Big\},  \notag\\
\P\big( F(\mathcal P) \le \E[F(\mathcal P)] - r \big) &\le \exp \bigg\{ -\frac{C^*r^2}{ \text{Var}(F(\mathcal P))} \bigg\}.  \notag
\end{align} 
\end{proposition}
\bigskip

\subsection{Technical lemmas for the proof of Theorem \ref{t:fslln.hvt}}  \label{sec:technical.lemma.heavy}

In this section, we offer a series of technical lemmas for the proof of Theorem \ref{t:fslln.hvt}. The first result below deals with asymptotic ratios of the regularly varying sequences as a function of $v_m$, $p_m$, $q_m$, $r_m$, and $s_m$ in \eqref{e:def.vm}, \eqref{e:def.pmqm}, and \eqref{e:def.rmsm}. 

\begin{lemma}  \label{l:RV.seq.heavy}
Let $u_m$ and $w_m$ be any of the sequences in \eqref{e:def.vm}, \eqref{e:def.pmqm}, and \eqref{e:def.rmsm}. Under the setup of Theorem \ref{t:fslln.hvt}, as $m\to\infty$, 
\begin{equation} \label{e:RV.seq.heavy}
\frac{u_m}{w_m} \to 1,  \ \ \  \frac{R_{u_m}}{R_{w_m}} \to 1, \ \ \ 
\frac{f(R_{u_m})}{f(R_{w_m})} \to 1.  
\end{equation}
\end{lemma}

\begin{proof}
By the definition of these five sequences, it is evident that 
$$
\frac{v_m}{v_{m+1}} \le \frac{u_m}{w_m} \le \frac{v_{m+1}}{v_m}
$$
for all $m\in \N$. Note that 
\begin{equation}  \label{e:ratio.vms}
1\le \frac{v_{m+1}}{v_m} \le \frac{e^{(m+1)^\gamma - m^\gamma}}{1-e^{-m^\gamma}} = \frac{e^{m^{\gamma-1}(\gamma+o(1))}}{1-e^{-m^\gamma}}, \ \ \ m\to\infty. 
\end{equation}
As $0<\gamma<1$, the rightmost term goes to $1$ as $m\to\infty$. 

For the proof of the second statement in \eqref{e:RV.seq.heavy}, let $\zeta>0$ be a regular variation exponent of $(R_n)$ (In the case of Theorem \ref{t:fslln.hvt} $(ii)$, we can take $\zeta=1/\alpha$). One can then rewrite the ratio as 
$$
\frac{R_{u_m}}{R_{w_m}} = \frac{R_{\lfloor w_m \frac{u_m}{w_m} \rfloor}}{R_{w_m}} - \Big( \frac{u_m}{w_m} \Big)^\zeta + \Big( \frac{u_m}{w_m} \Big)^\zeta. 
$$
Since $u_m/w_m \to 1$ as $m\to\infty$, we have that $u_m/w_m \in [1/2, 3/2]$ for sufficiently large $m$. By the uniform convergence of  regularly varying sequences (see, e.g., Proposition 2.4 in \cite{resnick:2007}), 
$$
\bigg|\,  \frac{R_{\lfloor w_m \frac{u_m}{w_m} \rfloor}}{R_{w_m}} - \Big( \frac{u_m}{w_m} \Big)^\zeta  \bigg| \le \sup_{\frac{1}{2} \le a\le \frac{3}{2}} \bigg|\,  \frac{R_{\lfloor a w_m \rfloor}}{R_{w_m}} - a^\zeta  \, \bigg| \to 0, \ \ \ m\to\infty. 
$$
Now, the proof is complete. Finally, since $\big( f(R_n), \, n\ge 1 \big)$ is also a regularly varying sequence,  the third statement in \eqref{e:RV.seq.heavy} can be shown in the same way as above. 
\end{proof}

The lemma below gives the asymptotic first and second moments for \eqref{e:def.upbd}, \eqref{e:def.lobd}, and \eqref{e:Wmmt}. 

\begin{lemma}\label{l:first.second.moments.heavy}
$(i)$ Under the assumptions of Theorem \ref{t:fslln.hvt} $(ii)$, for every $t, s\in[0,1]$, $i\geq k+2$, and $j\geq 1$, we have as $m\to\infty$, 
\begin{equation}\label{e:etmup}
R_{q_m}^{-d}\, \E \big[T_m^\ijupa(t,s)\big]\to\frac{\lambda^i}{i !}\, \mu_k^\ijp(t,s; \lambda),
\end{equation}
\begin{equation}\label{e:etmdown}
R_{p_m}^{-d}\, \E \big[T_m^\ijdoa(t,s)\big]\to\frac{\lambda^i}{i !}\, \mu_k^\ijp(t,s; \lambda). 
\end{equation}
Moreover,
\begin{equation}\label{e:vartmup}
\sup_{m\geq 1} R_{q_m}^{-d}\, \emph{Var}\big( T_m^\ijupa(t,s)\big)<\infty,
\end{equation}
\begin{equation}\label{e:vartmdown}
\sup_{m\geq 1} R_{p_m}^{-d}\, \emph{Var}\big(T_m^\ijdoa(t,s)\big)<\infty. 
\end{equation}
$(ii)$ Under the conditions of Theorem \ref{t:fslln.hvt} $(ii)$, for every $t\in [0,1]$ and $M\in \N$, we have as $m\to\infty$, 
\begin{equation}  \label{e:WmM.first}
R_{q_m}^{-d}\, \E \big[ V_{m,M}(t) \big] \to \sum_{i=M+1}^\infty i^{k+1} \frac{\lambda^i}{i!}\, \zeta_i (t) <\infty, 
\end{equation}
where 
\begin{equation}  \label{e:def.zetai}
\zeta_i(t) = \frac{s_{d-1}}{\alpha i-d}\, \int_{(\R^d)^{i-1}} \one \Big\{ \cc \big( \{ 0,\by \}, t \big) \text{ is connected} \Big\}d\by. 
\end{equation}
Furthermore, 
\begin{equation}  \label{e:WmM.second}
\sup_{m\geq 1} R_{q_m}^{-d}\, \emph{Var} \big( V_{m,M}(t) \big) < \infty. 
\end{equation}
\end{lemma}
\begin{proof}
For the proof of \eqref{e:etmup}, by the conditioning on $\Y$ we have 
\begin{align}
\E \big[ T_m^\ijupa(t,s)\big] &= \binom{v_{m+1}}{i} \E \Big[ h_t^\ijp (\Y)\, \one \big\{ \M(\Y)\ge R_{q_m} \big\} \P\big( \B (\Y; s/2)\cap \B(\X_{v_m}\setminus \Y; s/2) =\emptyset \, \big| \, \Y \big)\Big] \label{e:condtioning.expectation}\\
&={v_{m+1} \choose i}\int_{(\R^d)^i}h^\ijp_t(x_1,\dots,x_i)\, \ind{\M(x_1,\dots,x_i)\geq R_{q_m}}  \notag \\
&\qquad \qquad \qquad \qquad \qquad \times \big(1-I_s(x_1,\dots,x_i)\big)^{v_{m}-i}\prod_{\ell=1}^if(x_\ell)d\bx,\notag
\end{align}
where 
$$
I_s(x_1,\dots,x_i) :=\int_{\B\big( \{x_1,\dots,x_i\};s\big)}f(z)dz.
$$
Here, we consider the case in which a point set $\Y$ is contained in $\X_{v_m}$. We treat only this case, but the other cases (i.e., $\Y\cap (\X_{v_{m+1}}\setminus \X_{v_m})\neq \emptyset$) can be handled in the same way. 
Changing the variables $x_1= x$, $x_\ell= x+y_{\ell-1}$, $\ell\in \{2,\dots,i\}$ and using shift invariance condition \eqref{e:shift.invariance}, we have that 
\begin{align}
&\E  \big[T_m^\ijupa(t,s)\big]={v_{m+1} \choose i}\int_{\R^d}\int_{(\R^d)^{i-1}}h^\ijp_t(0,\by)\big(1-I_s(x,x+y_1,\cdots,x+y_{i-1})\big)^{v_{m}-i} \label{e:initial.change.variables}\\
&\qquad \qquad \qquad \qquad \qquad \times f(x)\, \ind{\| x\|\geq R_{q_m}} \prod_{\ell=1}^{i-1}f(x+y_\ell)\ind{\|x+y_\ell\|\geq R_{q_m}}d\by dx, \notag 
\end{align}
where $\by =(y_1,\dots,y_{i-1}) \in (\R^d)^{i-1}$. 
The polar coordinate transform $x\leftrightarrow (r,\theta)$ with $r\ge 0$, $\theta \in S^{d-1}$, gives that 
\begin{align}
\E  \big[T_m^\ijupa(t,s)\big] &= \binom{v_{m+1}}{i} \int_{S^{d-1}}J(\theta)\int_{R_{q_m}}^{\infty} r^{d-1}\int_{(\R^d)^{i-1}} h^\ijp_t(0,\by) \label{e:polar1}\\
&\qquad \qquad \times f(r) \prod_{\ell=1}^{i-1} f\big( \|r\theta +y_\ell\| \big)\, \one \big\{ \| r\theta + y_\ell \| \ge R_{q_m} \big\}  \notag \\
&\qquad   \qquad\qquad \times \big(1-I_s(r\theta, r\theta+y_1, \dots, r\theta + y_{i-1}) \big)^{v_{m}-i}d\by\, dr\, d\theta. \notag 
\end{align}
Furthermore, an additional change of variable $r=R_{q_m}\rho$ yields that 
\begin{align}
\E \big[ T_m^\ijupa(t,s)\big]=&{v_{m+1} \choose i}R_{q_m}^df(R_{q_m})^i\int_{S^{d-1}}J(\theta)\int_1^{\infty}\rho^{d-1}\int_{(\R^d)^{i-1}}h^\ijp_t(0,\by)\label{e:etmuptrans}\\
&\qquad  \times \frac{f(R_{q_m}\rho)}{f(R_{q_m})}  \prod_{\ell=1}^{i-1}\frac{f\big(R_{q_m}\|\rho\theta+y_\ell /R_{q_m} \|\big)}{f(R_{q_m})}\, \ind{\|\rho\theta+y_\ell/R_{q_m}\|\geq1} \notag \\
&\qquad \times \big(1-I_s(R_{q_m}\rho\theta, R_{q_m}\rho\theta+y_1, \dots, R_{q_m}\rho\theta + y_{i-1}) \big)^{v_{m}-i}d\by\, d\rho\, d\theta. \notag
\end{align}
By Lemma \ref{l:RV.seq.heavy}, we have 
$$
\binom{v_{m+1}}{i} f(R_{q_m})^i \sim \frac{\big( v_{m+1}f(R_{q_m}) \big)^i}{i!} \sim \frac{\big( q_{m}f(R_{q_m}) \big)^i}{i!}  \to \frac{\lambda^i}{i!}, \ \ \ m\to\infty. 
$$
By the regular variation assumption in \eqref{e:hvytl}, 
$$
\frac{f(R_{q_m}\rho)}{f(R_{q_m})}\prod_{\ell=1}^{i-1}\frac{f\big(R_{q_m}\|\rho\theta+y_\ell /R_{q_m} \|\big)}{f(R_{q_m})}\, \ind{\|\rho\theta+y_\ell/R_{q_m}\|\geq1}\to\rho^{-\alpha i} \qquad \text{as } n\to\infty
$$
for all $\rho\geq  1$, $\theta\in S^{d-1}$, and $y_1,\dots, y_{i-1}\in\R^d$.
As for the remaining term of the integrand in \eqref{e:etmuptrans}, we get that 
\begin{align}
&\lim_{m\to\infty} \big(1-I_s(R_{q_m}\rho\theta, R_{q_m}\rho\theta+y_1, \dots, R_{q_m}\rho\theta + y_{i-1}) \big)^{v_{m}-i} \label{e:compoent.part}\\
& =\lim_{m\to\infty} \Big(  1-\int_{\B\big( \{ 0,y_1,\dots,y_{i-1} \}; s \big)} f\big( R_{q_m} \| \rho\theta +z/R_{q_m} \| \big) dz\Big)^{v_m-i} \notag \\
& =\lim_{m\to\infty} \exp \Big\{  -v_m f(R_{q_m}) \int_{\B\big( \{ 0,y_1,\dots,y_{i-1} \}; s \big)} \frac{f\big( R_{q_m} \| \rho\theta +z/R_{q_m} \| \big)}{f(R_{q_m})} dz \Big\}. \notag
\end{align}
Now, \eqref{e:hvytl} and Lemma \ref{l:RV.seq.heavy} ensure that the last term in \eqref{e:compoent.part} converges to 
$$
\exp \Big\{ -\lambda \rho^{-\alpha} s^d \text{vol} \Big( \B\big( \{ 0,y_1,\dots,y_{i-1} \}; 1 \big)  \Big)  \Big\}. 
$$
Appealing to all of these convergence results and assuming temporarily that the dominated convergence theorem is applicable, we can obtain \eqref{e:etmup}. 

It now remains to find an integrable upper bound for the terms under the integral sign in \eqref{e:etmuptrans}. First it is evident that 
$$
\big(1-I_s(R_{q_m}\rho\theta, R_{q_m}\rho\theta+y_1, \dots, R_{q_m}\rho\theta + y_{i-1}) \big)^{v_{m}-i} \le 1. 
$$
Using Potter's bounds (see Proposition 2.6 in \cite{resnick:2007}), for every $\xi\in (0, \alpha-d)$, we have, for sufficiently large $m$, 
\begin{equation}\label{bd1}
\frac{f(R_{q_m}\rho)}{f(R_{q_m})}\leq(1+\xi)\rho^{-\alpha+\xi}
\end{equation}
and
\begin{equation}\label{bd2}
\prod_{\ell=1}^{i-1}\frac{f\big(R_{q_m}\|\rho\theta+y_\ell /R_{q_m} \|\big)}{f(R_{q_m})}\, \ind{\|\rho\theta+y_\ell/R_{q_m}\|\geq1}\leq(1+\xi)^{i-1}
\end{equation}
for all $\rho\geq 1$, $\theta\in S^{d-1}$ and $\by=(y_1,\dots,y_{i-1})\in (\R^d)^{i-1}$ such that $h_t^\ijp (0,\by)=1$. Combining all the  bounds derived above, together with $\int_1^{\infty}\rho^{d-1-\alpha+\xi}d\rho<\infty$, we can obtain an integrable upper bound, as desired. The proof of \eqref{e:etmdown} is similar, so we skip it here.

For the proof of \eqref{e:vartmup},  
\begin{align*}
\E \big[ T_m^\ijupa(t,s)^2\big] &=\sum_{\ell=0}^i\E\bigg[ \sum_{\substack {\Y \subset\X_{v_{m+1}} \\ \lvert \Y \rvert=i}} \sum_{\substack {\Y' \subset\X_{v_{m+1}} \\ \lvert \Y' \rvert=i, \, |\Y\cap \Y'|=\ell}}  h^\ijp_t(\Y)h^\ijp_t(\Y')\, \ind{\M(\Y\cup \Y')\geq R_{q_m}} \\
&\times \one \big\{\B(\Y;s/2)\cap\B(\X_{v_m}\setminus\Y;s/2)=\emptyset\big\} \, \one \big\{\B(\Y';s/2)\cap\B(\X_{v_m}\setminus\Y';s/2)=\emptyset\big\} \bigg] \\
&=:\sum_{\ell=0}^i\E[I_\ell]. 
\end{align*}
From this, $\text{Var}\big( T_m^\ijupa(t,s)\big)$ can be partitioned as $\text{Var}\big( T_m^\ijupa(t,s)\big)  = A_m +B_m$, where 
\begin{equation}  \label{e:def.AmBm.decomp}
A_m = \sum_{\ell=1}^i\E[I_\ell],  \ \ \ B_m = \E[I_0]-\Big\{ \E\big[T_m^\ijupa(t,s)\big] \Big\}^2. 
\end{equation}
For $\ell\in \{1,2,\dots,i\}$, 
\begin{align}
\E[I_\ell]\leq&\E \bigg[ \sum_{\substack {\Y \subset\X_{v_{m+1}} \\ \lvert \Y \rvert=i}} \sum_{\substack {\Y' \subset\X_{v_{m+1}} \\ \lvert \Y' \rvert=i, \, |\Y\cap \Y'|=\ell}}  h^\ijp_t(\Y)h^\ijp_t(\Y')\, \ind{\M(\Y\cup \Y')\geq R_{q_m}} \bigg] \label{e:EIell}\\
&={v_{m+1} \choose i}{i \choose \ell}{v_{m+1}-i \choose i-\ell} \notag \\
&\qquad \qquad \qquad \times \E\Big[ h^\ijp_t(\Y)h^\ijp_t(\Y')\, \one \big\{ \M(\Y\cup\Y')\ge R_{q_m}\big\}\Big] \, \one \big\{\lvert \Y\cap\Y' \rvert=\ell\big\} \notag\\
&\le  \cs v_{m+1}^{2i-\ell}\P\big( \cc(\Y\cup\Y', t) \text{ is connected}, \ \M(\Y\cup\Y')\ge R_{q_m}\big) \, \one \big\{\lvert \Y\cap\Y' \rvert=\ell\big\} \notag \\
&\le  \cs R_{q_m}^d, \notag 
\end{align}
where the last inequality comes from Lemma \ref{l:connectivity.prof.heavy} $(ii)$ below. This implies that $\sup_{m \ge 1}R_{q_m}^{-d}\, A_m<\infty$. 
In order to treat $B_m$ in \eqref{e:def.AmBm.decomp}, we derive an upper bound for $\E[I_0]$ by 
\begin{align}
\begin{split}  \label{e:EI0}
\E[I_0]\leq&\E \bigg[ \sum_{\substack {\Y \subset\X_{v_{m+1}} \\ \lvert \Y \rvert=i}} \sum_{\substack {\Y' \subset\X_{v_{m+1}} \\ \lvert \Y' \rvert=i  }}  h^\ijp_t(\Y)h^\ijp_t(\Y')\, \ind{\M(\Y\cup \Y')\geq R_{q_m}} \\
&\qquad \qquad \qquad \qquad \times\one \Big\{\B(\Y\cup\Y';s/2)\cap\B\big(\X_{v_m}\setminus(\Y\cup\Y'); s/2\big)=\emptyset\Big\}\bigg] \, \one \big\{ |\Y\cap \Y'|=0 \big\}  \\
&\le {v_{m+1} \choose i}^2\E \Big[h^\ijp_t(\Y)h^\ijp_t(\Y')\, \ind{\M(\Y\cup \Y')\geq R_{q_m}}  \\
&\qquad \qquad \qquad \qquad \qquad \times \Big(1-\int_{\B(\Y\cup\Y';s)}f(z)dz\Big)^{v_m-2i} \Big]\, \one \big\{ |\Y\cap \Y'|=0 \big\}, 
\end{split}
\end{align}
where the second inequality is obtained from an obvious relation $\binom{v_{m+1}-i}{i} \le \binom{v_{m+1}}{i}$ as well as 
the conditioning on $\Y\cup\Y'$ as in \eqref{e:condtioning.expectation}. Although we here consider the case when all the points in $\Y \cup \Y'$ belong to $\X_{v_m}$, the other cases (i.e., $(\Y \cup \Y')\cap (\X_{v_{m+1}}\setminus \X_{v_m})\neq \emptyset$) can be treated in the same manner. 
By the calculation similar to the above, we have 
\begin{align}
\Big\{ \E\big[T_m^\ijupa(t,s)\big] \Big\}^2 &= {v_{m+1} \choose i}^2\E\Big[ h^\ijp_t(\Y)h^\ijp_t(\Y')\, \ind{\M(\Y\cup \Y')\geq R_{q_m}} \label{e:subtract.first.moment} \\
&\qquad \quad \times\Big( 1-\int_{\B(\Y;s)}f(z)dz\Big)^{v_m-i}\Big(1-\int_{\B(\Y';s)}f(z)dz\Big)^{v_m-i}\Big] \, \one \big\{ |\Y\cap \Y'|=0 \big\} . \notag 
\end{align}
By \eqref{e:EI0} and \eqref{e:subtract.first.moment}, we get that $B_m \le v_{m+1}^{2i} \E[\Xi_m] \one \big\{ |\Y\cap \Y'|=0 \big\}$, where
\begin{align*}
\Xi_m := &h^\ijp_t(\Y)h^\ijp_t(\Y')\, \ind{\M(\Y\cup \Y')\ge R_{q_m}}
\\&\times \Big\{\Big(1-\int_{\B(\Y\cup\Y';s)}f(z)dz \Big)^{v_m-2i}-\Big(1-\int_{\B(\Y;s)}f(z)dz\Big)^{v_m-i}\Big(1-\int_{\B(\Y';s)}f(z)dz\Big)^{v_m-i}\Big\}. 
\end{align*}
Furthermore, $\Xi_m$ can be decomposed as $\Xi_m=C_m+D_m$, where
$$C_m=\Xi_m\ind{\B(\Y;s)\cap\B(\Y';s)=\emptyset},$$
$$D_m=\Xi_m\ind{\B(\Y;s)\cap\B(\Y';s)\neq\emptyset}.$$
Since $\text{vol}\big(\B(\Y\cup\Y';s)\big)=\text{vol}\big(\B(\Y;s)\big)+\text{vol}\big(\B(\Y';s)\big)$ whenever $\B(\Y;s)\cap\B(\Y';s)=\emptyset$, it is straightforward to check that 
\begin{equation}  \label{e:ECm}
v_{m+1}^{2i}\E [C_m] \one \big\{ |\Y\cap \Y'|=0 \big\} = o(R_{q_m}^d), \ \ \ m\to\infty, 
\end{equation}
by   the same arguments as in \eqref{e:condtioning.expectation}, \eqref{e:etmuptrans}, and \eqref{e:compoent.part}. Moreover, by Lemma \ref{l:connectivity.prof.heavy} $(ii)$, 
\begin{align}
&v_{m+1}^{2i}\E[D_m] \, \one \big\{ |\Y\cap \Y'|=0 \big\} \label{e:EDm}\\
&\le v_{m+1}^{2i}\P\big( \cc(\Y\cup\Y', 2s) \text{ is connected}, \, \M(\Y\cup\Y')\geq R_{q_m}\big)  \one \big\{ |\Y\cap \Y'|=0 \big\} \le \cs R_{q_m}^d. \notag 
\end{align}
It thus follows that $\sup_{m\ge1} R_{q_m}^{-d}\, B_m <\infty$, and hence \eqref{e:vartmup} has been established. Since the proof of \eqref{e:vartmdown} is very similar to that of \eqref{e:vartmup}, we will omit it. 

Next, turning to \eqref{e:WmM.first} we apply Fubini's theorem to obtain that 
\begin{align*}
R_{q_m}^{-d}\, \E\big[ V_{m,M}(t) \big] &= \sum_{i=M+1}^\infty i^{k+1} \binom{v_{m+1}}{i} R_{q_m}^{-d}\, \\
&\qquad \qquad  \times \P \Big( \cc \big( \{ X_1,\dots,X_i \}, t \big) \text{ is connected}, \, \M(X_1,\dots,X_i) \ge R_{q_m} \Big). 
\end{align*}
Taking $\delta>0$ so small that $\lambda(1+\delta) e\omega_d <1$, Lemma \ref{l:connectivity.prof.heavy} $(ii)$ demonstrates that 
\begin{equation}  \label{e:after.Stirling}
R_{q_m}^{-d}\, \E\big[ V_{m,M}(t) \big]  \le C^* \sum_{i=M+1}^\infty i^{k+1} \cdot \frac{\lambda^i(1+\delta)^ii^{i-2}\omega_d^{i-1}}{i!}, 
\end{equation}
where $C^*$ is a positive constant independent of $i$ and $m$. 
By Stirling's formula $i! \ge (i/e)^i$ for sufficiently large $i$, \eqref{e:after.Stirling} is further bounded by $C^*\sum_{i=M+1}^\infty i^{k-1} \big( \lambda (1+\delta) e\omega_d \big)^i$, which is finite due to the constraint $\lambda(1+\delta)e\omega_d<1$. By Lemma \ref{l:connectivity.prof.heavy} $(i)$ and the dominated convergence theorem, one can obtain \eqref{e:WmM.first} as required. 

Finally, \eqref{e:WmM.second} can be established by combining Lemma \ref{l:connectivity.prof.heavy} and the argument similar to that for the variance asymptotics of \eqref{e:vartmup}, so we will skip the  detailed discussions. 
\end{proof}

The next lemma is complementary to the proof of Lemma \ref{l:first.second.moments.heavy}. 
\begin{lemma}\label{l:connectivity.prof.heavy}
$(i)$ Under the assumptions of Theorem \ref{t:fslln.hvt} $(ii)$, for every $i\ge k+2$ and $t\in [0,1]$, 
\begin{equation}  \label{e:conv.connectivity.prob}
v_{m+1}^i R_{q_m}^{-d} \, \P \Big( \cc\big( \{ X_1,\dots,X_i \},  t\big) \emph{ is connected}, \, \M(X_1,\dots,X_i)\ge R_{q_m}\Big) \to \lambda^i \zeta_i(t), \ \ m\to\infty, 
\end{equation}
where $\zeta_i(t)$ is given in  \eqref{e:def.zetai}.  \\
\noindent $(ii)$ Moreover, let $\delta>0$ be a constant so small that $\lambda (1+\delta) e\omega_d < 1$. Then, for each $i\ge k+2$ and $t\in [0,1]$, there exists $N\in\N$ such that for all $m \ge N$, 
\begin{align}
&v_{m+1}^i R_{q_m}^{-d} \,  \P\Big( \cc\big( \{ X_1,\dots,X_i \},  t\big) \emph{ is connected}, \, \M(X_1,\dots,X_i)\ge R_{q_m}\Big) \le  C^*\lambda^i(1+\delta)^ii^{i-2}\omega_d^{i-1},  \label{e:bound.connectivity.prob} 
\end{align}
where $C^*>0$ is a constant independent of $i$ and $m$. 
\end{lemma}

\begin{proof}
We first prove \eqref{e:bound.connectivity.prob}. 
By the change of variables $x_1=x$, $x_\ell=x+y_{\ell-1}$ for $\ell\in \{2,\dots,i\}$, which is followed by an additional change of variables $x\leftrightarrow(R_{q_m}\rho,\theta)$ as in \eqref{e:polar1} and \eqref{e:etmuptrans}, we have that 
\begin{align}
&v_{m+1}^i R_{q_m}^{-d}\, \P\Big( \cc\big( \{ X_1,\dots,X_i \},  t\big) \text{ is connected}, \, \M(X_1,\dots,X_i)\ge R_{q_m}\Big) \label{e:expii-2} \\
&=v_{m+1}^i f(R_{q_m})^i\int_{S^{d-1}}J(\theta)\int_1^\infty\rho^{d-1}\int_{(\R^d)^{i-1}}\ind{\cc\big(\{0,\by\},t\big) \text{ is connected}}\notag \\
&\qquad \qquad \times \frac{f(R_{q_m}\rho)}{f(R_{q_m})}\prod_{\ell=1}^{i-1}\frac{f\big(R_{q_m}\|\rho\theta+y_\ell /R_{q_m} \|\big)}{f(R_{q_m})}\, \ind{\|\rho\theta+y_\ell/R_{q_m}\|\geq1}d\by\, d\rho\, d\theta.\notag
\end{align}
Employing Potter's bound as in \eqref{bd1} and \eqref{bd2} with $\xi=\min\{(\alpha-d)/2,\delta/2\}$, there exists $N_1\in \N$, such that for all $m\ge N_1$, 
\begin{equation}  \label{e:Potter.connectivity}
\frac{f(R_{q_m}\rho)}{f(R_{q_m})}\prod_{\ell=1}^{i-1}\frac{f\big(R_{q_m}\|\rho\theta+y_\ell /R_{q_m} \|\big)}{f(R_{q_m})}\, \ind{\|\rho\theta+y_\ell/R_{q_m}\|\geq1}\leq(1+\xi)^i\rho^{-\alpha+\xi}. 
\end{equation}
By the definition of $\xi$ above  and Lemma \ref{l:RV.seq.heavy}, one can find $N_2\in\N$, such that for all $m\geq N_2$, $(1+\xi)v_{m+1}f(R_{q_m})\leq\lambda(1+\delta)$. Thus, for all  $m \ge N:=N_1\vee N_2$, 
\begin{align*}
&v_{m+1}^iR_{q_m}^{-d}\, \P\Big( \cc\big( \{ X_1,\dots,X_i \},  t\big) \text{ is connected}, \, \M(X_1,\dots,X_i)\ge R_{q_m}\Big) \\
&\le \frac{s_{d-1}\big((1+\xi)v_{m+1}f(R_{q_m})\big)^i}{\alpha-d-\xi}\int_{(\R^d)^{i-1}}\ind{\cc\big(\{0,\by\}, t\big) \text{ is connected}}d\by\\
&\le \frac{2s_{d-1}\lambda^i (1+\delta)^i}{ \alpha-d}\int_{(\R^d)^{i-1}}\ind{\cc\big(\{0,\by\}, 1\big) \text{ is connected}}d\by\\
&\le \frac{2s_{d-1}\lambda^i(1+\delta)^ii^{i-2}\omega_d^{i-1}}{\alpha-d}. 
\end{align*}
The last inequality above follows from an elementary fact that there exist $i^{i-2}$ spanning trees on a set of $i$ points. 

For the proof of \eqref{e:conv.connectivity.prob}, returning to \eqref{e:expii-2} we see that $v_{m+1}^i f(R_{q_m})^i\to \lambda$ as $m\to\infty$ by Lemma \ref{l:RV.seq.heavy}. By the dominated convergence theorem with the regular variation assumption \eqref{e:hvytl} and an integrable bound obtained in   \eqref{e:Potter.connectivity}, we can get \eqref{e:conv.connectivity.prob}. 
\end{proof}

All the lemmas below will apply to the proof of Theorem \ref{t:fslln.hvt} $(i)$. We provide the asymptotic moments of $S_m^\uparrow (t)$, $S_m^\downarrow(t)$, and $W_m(t)$; see \eqref{e:def.Smup}, \eqref{e:def.Smdown}, and \eqref{e:def.Wm}. 

\begin{lemma}  \label{l:first.second.asymptotics}
Under the assumptions of Theorem \ref{t:fslln.hvt} $(i)$, for every $t\in [0,1]$ we have that as $m\to\infty$, 
\begin{align*}
\big( v_{m+1}^{k+2} R_{q_m}^d f(R_{q_m})^{k+2} \big)^{-1} \E \big[S_m^\uparrow (t)\big] &\to \frac{\mu_k^\ktop (t; 0)}{(k+2)!}, \\
\big( v_{m}^{k+2} R_{p_m}^d f(R_{p_m})^{k+2} \big)^{-1} \E \big[S_m^\downarrow (t)\big] &\to \frac{\mu_k^\ktop (t; 0)}{(k+2)!}, 
\end{align*}
and also, 
$$
\big( v_{m+1}^{k+3} R_{q_m}^d f(R_{q_m})^{k+3} \big)^{-1} \E \big[ W_m(t) \big] \to \frac{s_{d-1}}{(k+3)!\big( \alpha (k+3)-d \big)}\,  \int_{(\R^d)^{k+2}} \hspace{-10pt}\one \Big\{  \cc \big( \{ 0,\by \}, t \big) \text{ is connected} \Big\} d\by, 
$$
where $\by=(y_1,\dots,y_{k+2})\in (\R^d)^{k+2}$. 
Moreover, 
\begin{align*}
&\sup_{m\ge 1}\big( v_{m+1}^{k+2} R_{q_m}^d f(R_{q_m})^{k+2} \big)^{-1}  \text{Var}\big( S_m^\uparrow (t) \big) <\infty, \\
&\sup_{m\ge 1} \big( v_{m}^{k+2} R_{p_m}^d f(R_{p_m})^{k+2} \big)^{-1}\text{Var}\big( S_m^\downarrow (t) \big) <\infty. 
\end{align*}
\end{lemma}

\begin{proof}
The proof here is mostly the same as those for Lemma \ref{l:first.second.moments.heavy}, so we provide only the sketch of proof for the first statement. Appealing to \emph{Palm theory for Poisson processes} (see, e.g., Theorem 1.6 in \cite{penrose:2003}), 
$$
\E\big[S_m^\uparrow (t) \big] = \frac{v_{m+1}^{k+2}}{(k+2)!}\, \int_{(\R^d)^{k+2}} h_t^\ktop (x_1,\dots,x_{k+2}) \, \one \big\{ \M(x_1,\dots,x_{k+2}) \ge R_{q_m}\big\}\prod_{\ell=1}^{k+2} f(x_\ell) d\bx. 
$$ 
By the same change of variables as in \eqref{e:etmuptrans} with $i=k+2$ and $j=1$, 
\begin{align*}
\E \big[ S_m^\uparrow(t)\big]=&\frac{v_{m+1}^{k+2}R_{q_m}^d f(R_{q_m})^{k+2}}{(k+2)!}\,   \int_{S^{d-1}}J(\theta)\int_1^{\infty}\rho^{d-1}\int_{(\R^d)^{k+1}}h^\ktop_t(0,\by)\\
&\qquad  \times \frac{f(R_{q_m}\rho)}{f(R_{q_m})}  \prod_{\ell=1}^{k+1}\frac{f\big(R_{q_m}\|\rho\theta+y_\ell /R_{q_m} \|\big)}{f(R_{q_m})}\, \ind{\|\rho\theta+y_\ell/R_{q_m}\|\geq1} d\by\, d\rho\, d\theta. 
\end{align*}
The rest of our discussion is completely the same as the argument after \eqref{e:etmuptrans}. 
\end{proof}

The next result justifies that with a proper scaling, the asymptotic behaviors of $S_m^\uparrow(t)$ and $S_m^\downarrow(t)$ will remain unchanged, even if the Poisson point process is replaced with the corresponding binomial process. 

\begin{lemma}  \label{l:difference.binom.poisson}
Under the assumptions of Theorem \ref{t:fslln.hvt} $(i)$, for every $t\in [0,1]$ we have, as $m\to\infty$, 
\begin{align*}
&\big( v_{m+1}^{k+2} R_{q_m}^df(R_{q_m})^{k+2} \big)^{-1} \bigg\{  \sum_{\substack{\Y \subset \X_{v_{m+1}}, \\ |\Y|=k+2}} h_t^\ktop(\Y)\, \one \big\{ \M(\Y) \ge R_{q_m} \big\} -S_m^\uparrow(t) \bigg\} \to 0, \ \ \text{a.s.}, \\
&\big( v_{m}^{k+2} R_{p_m}^df(R_{p_m})^{k+2} \big)^{-1} \bigg\{  \sum_{\substack{\Y \subset \X_{v_{m}}, \\ |\Y|=k+2}} h_t^\ktop(\Y)\, \one \big\{ \M(\Y) \ge R_{p_m} \big\} -S_m^\downarrow(t) \bigg\} \to 0, \ \ \text{a.s.}, 
\end{align*}
and further, 
$$
\big( v_{m+1}^{k+3} R_{q_m}^d f(R_{q_m})^{k+3}\big)^{-1} \bigg\{ \sum_{\substack{\Y\subset \X_{v_{m+1}}, \\ |\Y|=k+3}} \one \big\{ \cc(\Y, t) \text{ is connected}, \, \M(\Y)\ge R_{q_m} \big\} - W_m(t)\bigg\} \to 0, \ \ \text{a.s.}
$$
\end{lemma}
\begin{proof}
The proofs of these statements are essentially the same, so we show only the first result. By the Borel-Cantelli lemma and Markov's inequality, it suffices to demonstrate that 
\begin{align}
&\sum_{m=1}^\infty \big( v_{m+1}^{k+2} R_{q_m}^df(R_{q_m})^{k+2} \big)^{-1} \E\bigg[\,  \bigg|  \sum_{\substack{\Y \subset \X_{v_{m+1}}, \\ |\Y|=k+2}} h_t^\ktop(\Y)\, \one \big\{ \M(\Y) \ge R_{q_m} \big\}  -S_m^\uparrow(t) \, \bigg| \, \bigg] < \infty. \label{e:BC.with.Markov}
\end{align}
Recall that $|\mathcal P_{v_{m+1}}|$ (i.e., the cardinality of $\mathcal P_{v_{m+1}}$) is Poisson distributed with parameter $v_{m+1}$.  By the conditioning on the values of $|\mathcal P_{v_{m+1}}|$, we get that 
\begin{align}
&\E\bigg[\,  \bigg|  \sum_{\substack{\Y \subset \X_{v_{m+1}}, \\ |\Y|=k+2}} h_t^\ktop(\Y)\, \one \big\{ \M(\Y) \ge R_{q_m} \big\} -S_m^\uparrow(t)\, \bigg|\,  \bigg] \label{e:after.Markov} \\
&=\sum_{\ell=0}^\infty \E\bigg[\,  \bigg|  \sum_{\Y \subset \X_{v_{m+1}}, \, |\Y|=k+2} \hspace{-10pt} h_t^\ktop(\Y)\, \one \big\{ \M(\Y) \ge R_{q_m} \big\} \notag \\
& \qquad \qquad \qquad \qquad \qquad - \sum_{\substack{\Y \subset \X_\ell, \\ |\Y|=k+2}} h_t^\ktop(\Y)\, \one \big\{ \M(\Y) \ge R_{q_m} \big\} \bigg|\,  \bigg] \P\big( |\mathcal P_{v_{m+1}}|=\ell \big) \notag \\
&= \sum_{\ell=0}^\infty \Big| \binom{\ell}{k+2} -\binom{v_{m+1}}{k+2}  \Big| \, \E \big[ h_t^\ktop(X_1,\dots,X_{k+2})\, \notag \\
&\qquad \qquad \qquad \qquad \qquad \qquad\qquad \qquad \times \one \big\{ \M(X_1,\dots,X_{k+2})\ge R_{q_m}  \big\}  \big] \P\big( |\mathcal P_{v_{m+1}}|=\ell \big), \notag 
\end{align}
where $X_1,\dots, X_{k+2}$ are i.i.d random variables with density $f$. 
Proceeding as in the proof of Lemma \ref{l:first.second.moments.heavy}, we can derive that  
$$
\E \big[ h_t^\ktop(X_1,\dots,X_{k+2})\, \one \big\{ \M(X_1,\dots,X_{k+2})\ge R_{q_m}  \big\}  \big] \sim C^* R_{q_m}^d f(R_{q_m})^{k+2}, \ \ \ m\to\infty. 
$$
Referring this back into \eqref{e:after.Markov}, the left hand side in \eqref{e:BC.with.Markov} is now bounded by a constant multiple of 
\begin{align}
&\sum_{m=1}^\infty \frac{1}{v_{m+1}^{k+2}}\, \sum_{\ell=0}^\infty \Big| \binom{\ell}{k+2} -\binom{v_{m+1}}{k+2}  \Big| \,  \P\big( |\mathcal P_{v_{m+1}}|=\ell \big) \label{e:CS.inequ} \\
&=\sum_{m=1}^\infty \frac{1}{v_{m+1}^{k+2}}\, \E \bigg[ \Big| \binom{|\mathcal P_{v_{m+1}}|}{k+2} -\binom{v_{m+1}}{k+2}  \Big|  \bigg] \notag \\
&\le \sum_{m=1}^\infty \frac{1}{v_{m+1}^{k+2}}\, \bigg\{ \E \bigg[  \binom{|\mathcal P_{v_{m+1}}|}{k+2}^2 \bigg] - 2\binom{v_{m+1}}{k+2} \E \bigg[  \binom{|\mathcal P_{v_{m+1}}|}{k+2} \bigg] + \binom{v_{m+1}}{k+2}^2  \bigg\}^{1/2}, \notag
\end{align}
where the last relation is due to the Cauchy-Schwarz inequality. It is elementary to check that there are constants $c_j$, $j=1,2,\dots,2k+4$,  with $c_{2k+4}=1$, such that 
\begin{align*}
 &\E \bigg[  \binom{|\mathcal P_{v_{m+1}}|}{k+2} \bigg] = \frac{v_{m+1}^{k+2}}{(k+2)!}, \ \ \text{and } \  \E \bigg[  \binom{|\mathcal P_{v_{m+1}}|}{k+2}^2 \bigg]  =\frac{1}{\big((k+2)!\big)^2}\, \sum_{j=1}^{2k+4}c_j v_{m+1}^j. 
\end{align*}
Therefore, the last expression in \eqref{e:CS.inequ} can be written as 
\begin{equation}  \label{e:last.part.negligibility}
\sum_{m=1}^\infty \frac{1}{v_{m+1}^{k+2}} \cdot \frac{1}{(k+2)!}\Big(  \sum_{j=1}^{2k+3}c_j' v_{m+1}^j \Big)^{1/2}
\end{equation}
for some constants $c_j'$, $j=1,\dots,2k+3$ (note that $v_{m+1}^{2k+4}$ has disappeared here). Finally, \eqref{e:last.part.negligibility} is further bounded by 
$$
C^*\sum_{m=1}^\infty \frac{1}{v_{m+1}^{1/2}} \le C^*\sum_{m=1}^\infty e^{-m^\gamma/2} < \infty. 
$$
\end{proof}

\subsection{Technical lemmas for the proof of Theorem \ref{t:fslln.expdt}}  \label{sec:technical.lemma.exp}

In the below we deduce various auxiliary results for the proof of Theorem \ref{t:fslln.expdt}, all of which are analogous to the corresponding lemmas in Section \ref{sec:technical.lemma.heavy}. Among them, Lemma \ref{l:RV.seq.exp1} below analyzes asymptotic ratios of the  sequences as a function of $v_m$, $p_m$, $q_m$, $b_m$, $c_m$, $e_m$, and $g_m$ defined in \eqref{e:def.vm}, \eqref{e:def.pmqm}, \eqref{e:def.bmcm}, and \eqref{e:def.emgm}. Moreover, Lemmas \ref{l:first.second.moments.exp1} and \ref{l:first.second.moments.exp2} are used for calculating the asymptotic moments of various quantities appearing in the proof. Before stating these lemmas, recall the notations $T_m^\ijupa (t,s)$, $T_m^\ijdoa(t,s)$, $V_{m,M}(t)$, $S_m^\uparrow(t)$, $S_m^\downarrow(t)$, and $W_m(t)$, which are  defined respectively at  \eqref{e:def.upbd}, \eqref{e:def.lobd}, \eqref{e:Wmmt}, \eqref{e:def.Smup}, \eqref{e:def.Smdown}, and \eqref{e:def.Wm}. 

\begin{lemma}   \label{l:RV.seq.exp1}
Under the setup of Theorem \ref{t:fslln.expdt}, let $u_m$, $w_m$ be any of the sequences in \eqref{e:def.vm}, \eqref{e:def.pmqm},  \eqref{e:def.bmcm}, and \eqref{e:def.emgm}.  Then, as $m\to\infty$, 
$$
\frac{R_{u_m}}{R_{w_m}} \to 1, \ \ \
\frac{a(R_{u_m})}{a(R_{w_m})} \to 1, \ \ \ 
\frac{f(R_{u_m})}{f(R_{w_m})} \to 1. 
$$
\end{lemma}
\begin{proof}
Because of Proposition 2.6 in \cite{resnick:2007}, we see that $\psi^\leftarrow \in \text{RV}_{1/\tau}$. 
Under the setup of Theorem \ref{t:fslln.expdt} $(i)$, we have $\psi (R_n) =\log n + b\log \log n \sim \log n$ as  $n\to\infty$. 
In the case of Theorem \ref{t:fslln.expdt} $(ii)$, we have $nf(R_n)\to \lambda\in (0,\infty)$, $n\to\infty$, which again implies  $\psi(R_n)\sim \log n$ as $n \to \infty$. In both cases, by the uniform convergence of regularly varying sequences (see, e.g., Proposition 2.4 in \cite{resnick:2007}), 
$$
\frac{\psi^\leftarrow \big( \psi(R_n) \big)}{\psi^\leftarrow (\log n)} \sim \bigg( \frac{\psi(R_n)}{\log n} \bigg)^{1/\tau} \to 1, \ \ n\to\infty. 
$$
Since $\psi^\leftarrow \big( \psi(R_n) \big)= R_n$, we conclude that 
\begin{equation}  \label{e:Rn.psi.inv}
R_n\sim \psi^\leftarrow (\log n), \ \ n\to\infty. 
\end{equation}
We are now ready to show the first statement. By the uniform convergence of regularly varying sequences, 
$$
\frac{R_{u_m}}{R_{w_m}} \sim \frac{\psi^\leftarrow(\log u_m)}{\psi^\leftarrow (\log w_m)} \sim \Big( \frac{\log u_m}{\log w_m} \Big)^{1/\tau} \to 1, \ \ \ m\to\infty, 
$$
where the last convergence is obtained from $u_m/w_m\to 1$ as $m\to\infty$ (see \eqref{e:ratio.vms}). 

Next, $a\in \text{RV}_{1-\tau}$ implies that 
$$
\frac{a(R_{u_m})}{a(R_{w_m})} \sim \Big( \frac{R_{u_m}}{R_{w_m}} \Big)^{1-\tau} \to 1, \ \ \text{as } m\to\infty. 
$$
As for the last statement, the result is obvious under the setup of Theorem \ref{t:fslln.expdt} $(ii)$. In the case of Theorem \ref{t:fslln.expdt} $(i)$, it is easy to calculate that 
$$
\frac{f(R_{u_m})}{f(R_{w_m})} = \frac{w_m}{u_m}\cdot \Big( \frac{\log w_m}{\log u_m} \Big)^b \to 1, \ \ \ m\to\infty. 
$$
\end{proof}

The following two lemmas will be applied for the proof of Theorem \ref{t:fslln.expdt} $(ii)$. 

\begin{lemma}  \label{l:first.second.moments.exp1}
$(i)$ Under the setup of Theorem \ref{t:fslln.expdt} $(ii)$, for every $t, s\in [0,1]$, $i\ge k+2$, and $j\ge 1$, we have as $m\to\infty$, 
\begin{align}
&\big( a(R_{q_m})R_{q_m}^{d-1} \big)^{-1} \E \big[ T_m^\ijupa (t,s) \big] \to \frac{\lambda^i}{i!}\, \xi_k^\ijp (t,s; \lambda),  \label{e:first.moment1.exp.weak.core}\\
&\big( a(R_{p_m})R_{p_m}^{d-1} \big)^{-1} \E \big[ T_m^\ijdoa (t,s) \big] \to \frac{\lambda^i}{i!}\, \xi_k^\ijp (t,s; \lambda). \notag
\end{align}
Moreover, 
\begin{align}
&\sup_{m\ge 1} \big( a(R_{q_m})R_{q_m}^{d-1} \big)^{-1}  \text{Var} \big(T_m^\ijupa (t,s)  \big) <\infty, \label{e:second.moment1.exp.weak.core}\\
&\sup_{m\ge 1} \big( a(R_{p_m})R_{p_m}^{d-1} \big)^{-1}  \text{Var} \big(T_m^\ijdoa (t,s)  \big) <\infty. \notag
\end{align}
$(ii)$ Furthermore, for every $M\in \N$ and $t\in [0,1]$, we have as $m\to\infty$, 
\begin{equation}  \label{e:first.moment.VmM.exp.weak.core}
\big( a(R_{q_m})R_{q_m}^{d-1} \big)^{-1}  \E \big[ V_{m,M}(t) \big] \to \sum_{i=M+1}^\infty i^{k+1}\frac{\lambda^i}{i!}\, \zeta'_i(t) <\infty, 
\end{equation}
where 
\begin{align}
\zeta'_i(t)&:= \int_{S^{d-1}}\int_0^\infty \int_{(\R^d)^{i-1}} \one \Big\{ \cc\big( \{  0,\by\}, t \big) \text{ is connected} \Big\}\, e^{-\rho i - c^{-1} \sum_{\ell=1}^{i-1} \langle \theta, y_\ell \rangle} \label{e:def.zeta.prime.it}\\
&\qquad \qquad \qquad \qquad\qquad\qquad \times \prod_{\ell=1}^{i-1} \one \big\{  \rho+c^{-1} \langle \theta, y_\ell \rangle \ge 0 \big\} \, d\by \, d\rho\, J(\theta) \, d\theta,\notag 
\end{align}
and also, 
$$
\sup_{m\ge 1} \big( a(R_{q_m}) R_{q_m}^{d-1} \big)^{-1} \text{Var}\big( V_{m,M}(t) \big) < \infty. 
$$
\end{lemma}

\begin{proof}
Here, we show only \eqref{e:first.moment1.exp.weak.core}, \eqref{e:second.moment1.exp.weak.core}, and \eqref{e:first.moment.VmM.exp.weak.core}. 
For \eqref{e:first.moment1.exp.weak.core}, by the same change of variables as in \eqref{e:initial.change.variables} after conditioning on $\Y$ as in \eqref{e:condtioning.expectation}, we get that 
\begin{align}
&\E  \big[T_m^\ijupa(t,s)\big]={v_{m+1} \choose i}\int_{\R^d}\int_{(\R^d)^{i-1}}h^\ijp_t(0,\by)\big(1-I_s(x,x+y_1,\cdots,x+y_{i-1})\big)^{v_{m}-i} \label{e:second.time.change.variables}\\
&\qquad \qquad \qquad \qquad \qquad \times f(x)\, \ind{\| x\|\geq R_{q_m}} \prod_{\ell=1}^{i-1}f(x+y_\ell)\ind{\|x+y_\ell\|\geq R_{q_m}}d\by dx, \notag 
\end{align}
where $\by =(y_1,\dots,y_{i-1}) \in (\R^d)^{i-1}$. Note that \eqref{e:second.time.change.variables} completely agrees with \eqref{e:initial.change.variables}. Moreover, by the polar coordinate transform $x\leftrightarrow (r,\theta)$ with $r\ge 0$, $\theta \in S^{d-1}$, it turns out that $\E\big[  T_m^\ijupa (t,s)\big]$ above  becomes the right hand side of \eqref{e:polar1}. Next, we make an additional change of variable $r=\aRrho$ to get that 
\begin{align}
\E\big[  T_m^\ijupa (t,s)\big] &= \binom{v_{m+1}}{i} a(R_{q_m}) R_{q_m}^{d-1} f(R_{q_m})^i \label{e:final.form.exp}\\
&\quad \times \int_{S^{d-1}} J(\theta) \int_0^\infty \Big(  1+\frac{a(R_{q_m})}{R_{q_m}}\, \rho\Big)^{d-1}  \int_{(\R^d)^{i-1}} 
 h_t^\ijp (0,\by)\, \frac{f\big( R_{q_m}+a(R_{q_m}) \rho \big)}{f(R_{q_m})} \notag\\
&\quad \times \prod_{\ell=1}^{i-1} \frac{f\big( \big\| \big( R_{q_m}+a(R_{q_m})\rho \big)\theta +y_\ell  \big\| \big)}{f(R_{q_m})}\, \one \Big\{ \big\| \big(  R_{q_m}+a(R_{q_m})\rho\big)\theta +y_\ell  \big\| \ge R_{q_m} \Big\} \notag \\
&\quad \times \Big( 1-I_s \big( (\aRrho)\theta, (\aRrho)\theta+y_1, \notag \\
&\qquad \qquad \qquad \qquad \qquad \qquad \qquad  \dots, (\aRrho)\theta  + y_{i-1} \big) \Big)^{v_m-i} d\by \, d\rho\, d\theta. \notag
\end{align}
Lemma \ref{l:RV.seq.exp1} gives that
$$
\binom{v_{m+1}}{i}f(R_{q_m})^i \sim \frac{\big( v_m f(R_{v_m}) \big)^i}{i!} \to \frac{\lambda^i}{i!}, \ \  \text{as } m\to\infty. 
$$
In the following, we calculate the limit of each of the terms under the integral sign of \eqref{e:final.form.exp}. By  Proposition 2.5 in \cite{resnick:2007}, we have $a'\in \text{RV}_{-\tau}$, which implies $a(z)/z\to 0$ as $z\to\infty$, and hence, 
\begin{equation}  \label{e:ratio.plus.1}
\Big(  1+\frac{a(R_{q_m})}{R_{q_m}}\, \rho\Big)^{d-1}  \to 1, \ \ \ m\to\infty, 
\end{equation}
for every $\rho>0$. Furthermore, \eqref{e:ratio.plus.1} is bounded by $2(1\vee \rho)^{d-1}$ for sufficiently large $m$. For the untreated term in the second line of \eqref{e:final.form.exp}, we have 
\begin{align}
\frac{f\big( R_{q_m}+a(R_{q_m}) \rho \big)}{f(R_{q_m})} &= \exp \Big\{  -\psi \big( \aRrho \big) +\psi(R_{q_m}) \Big\} \label{e:main.ratio.f}\\
&=\exp\bigg\{  -\int_0^\rho \frac{a(R_{q_m})}{a\big(R_{q_m}+a(R_{q_m})r\big)}dr \bigg\}. \notag
\end{align}
Since $a\in \text{RV}_{1-\tau}$, the uniform convergence of regularly varying sequences, together with $a(z)/z\to0$, $z\to\infty$, indicates that 
\begin{equation}  \label{e:a.ratio.conv.0}
\frac{a(R_{q_m})}{a\big(R_{q_m}+a(R_{q_m})r\big)} \to 1, \ \ \ m\to\infty, 
\end{equation}
for every $r> 0$. It thus follows that for every $\rho>0$, 
\begin{equation}  \label{e:conv.main.ratio}
\frac{f\big( R_{q_m}+a(R_{q_m}) \rho \big)}{f(R_{q_m})} \to e^{-\rho}, \ \ \text{as } m\to\infty. 
\end{equation}
In order to find an upper bound of \eqref{e:main.ratio.f}, we define an array $\big( s_\ell(m), \, \ell \ge 0, \, m\ge 0 \big)$ by 
$$
s_\ell(m) = \frac{\psi^\leftarrow \big( \psi(R_{q_m}) +\ell \big)-R_{q_m}}{a(R_{q_m})}, 
$$
which is equivalent to $\psi \big(R_{q_m} +a(R_{q_m})s_\ell(m)\big) = \psi(R_{q_m})+\ell$. According to Lemma 5.2 in \cite{balkema:embrechts:2004}, for $0<\epsilon<1/d$, there exists $N=N(\epsilon)\in \N$ such that $s_\ell(m)\le \epsilon^{-1} e^{\ell\epsilon}$ for all $m\ge N$ and $\ell\ge 0$. Since $\psi$ is increasing, the upper bound of \eqref{e:main.ratio.f} is given by 
\begin{align}
&\exp \Big\{ -\psi \big( \aRrho \big) +\psi(R_{q_m})  \Big\}\, \one \{\rho>0 \} \label{e:balkema.bound}\\
&= \sum_{\ell=0}^\infty \one \big\{ s_\ell(m) <\rho \le s_{\ell+1}(m)  \big\}\, \exp \Big\{ -\psi \big( \aRrho \big) +\psi(R_{q_m})  \Big\} \notag \\
&\le \sum_{\ell=0}^\infty  \one \big\{ 0 <\rho \le \epsilon^{-1} e^{(\ell+1)\epsilon} \big\}\, e^{-\ell}. \notag 
\end{align} 

Subsequently, we calculate the limit of the term in the third line of \eqref{e:final.form.exp}. By an expansion of $\big\| \big( \aRrho \big)\theta + y_\ell  \big\|$ for $\ell\in \{ 1,\dots,i-1 \}$, we can define 
\begin{align}
\gamma_m(\rho, \theta, y_\ell) &:= \big\|  \big( \aRrho \big)\theta +y_\ell  \big\| - \big( \aRrho + \langle \theta, y_\ell \rangle \big) \label{e:gamma.m.rho.theta.yl}\\
&=\frac{\|y_\ell\|^2 - \langle \theta, y_\ell \rangle^2}{\big\|  \big( \aRrho \big)\theta +y_\ell  \big\|  + \aRrho + \langle \theta, y_\ell \rangle}. \notag
\end{align}
If $\big\| \big( \aRrho \big)\theta + y_\ell  \big\|\ge R_{q_m}$, it then holds that 
\begin{equation} \label{e:conv.gamma.m}
\big| \gamma_m(\rho, \theta, y_\ell) \big| \le \frac{\big| \|y_\ell\|^2 - \langle \theta, y_\ell \rangle^2  \big|}{2R_{q_m} + \langle \theta, y_\ell \rangle} \to 0, \ \ \ m\to\infty, 
\end{equation}
uniformly for all $\rho>0$, $\theta\in S^{d-1}$, and $y_\ell\in \R^d$ with $\|y_\ell\|\le Lt$ ($L$ is given in \eqref{e:locally.determined}). Let 
\begin{align*}
A_m &= \Big\{ z\in \R^d: \big\|  \big( \aRrho \big)\theta +z \big\| \ge R_{q_m}\Big\}= \Big\{  z\in \R^d: \rho + \zeta_m (\rho, \theta, z) \ge0 \Big\}, 
\end{align*}
where 
\begin{equation}  \label{e:zeta.m.rho.theta.z}
\zeta_m (\rho, \theta, z) := \frac{\langle \theta, z \rangle + \gamma_m (\rho, \theta, z)}{a(R_{q_m})}. 
\end{equation}
Then, from \eqref{e:gamma.m.rho.theta.yl} and \eqref{e:zeta.m.rho.theta.z} we have 
\begin{align}
&\prod_{\ell=1}^{i-1}\frac{f\big( \big\| \big( R_{q_m}+a(R_{q_m})\rho \big)\theta +y_\ell  \big\| \big)}{f(R_{q_m})}\, \one \{ y_\ell\in A_m \} \label{e:third.line.ratio}\\
&=\prod_{\ell=1}^{i-1}\exp \Big\{ -\psi \big( \aRrho + \langle \theta, y_\ell \rangle + \gamma_m(\rho, \theta, y_\ell) \big) +\psi(R_{q_m}) \Big\}\, \one \{ y_\ell\in A_m \} \notag \\
&=\prod_{\ell=1}^{i-1}\exp \bigg\{  -\int_0^{\rho + \zeta_m(\rho, \theta, y_\ell)} \frac{a(R_{q_m})}{a\big( R_{q_m} + a(R_{q_m})r \big)} \, dr\bigg\}\, \one\{y_\ell \in A_m\}. \notag
\end{align}
From \eqref{e:limit.a} and \eqref{e:conv.gamma.m} it follows that, for every $\ell\in\{ 1,\dots,i-1 \}$, 
$$
\zeta_m(\rho, \theta, y_\ell) \to c^{-1} \langle \theta, y_\ell \rangle, \ \ \text{as } m\to\infty, 
$$
uniformly for all $\rho>0$, $\theta\in S^{d-1}$, and $y_\ell\in \R^d$ with $\|y_\ell\|\le Lt$. 
Thus, by \eqref{e:a.ratio.conv.0}, we have as $m\to\infty$, 
\begin{align}  
&\prod_{\ell=1}^{i-1}\frac{f\big( \big\| \big( R_{q_m}+a(R_{q_m})\rho \big)\theta +y_\ell  \big\| \big)}{f(R_{q_m})}\, \one \{ y_\ell\in A_m \} \label{e:conv.third.line.ratio} \\
&\to e^{ -\rho(i-1)-c^{-1}\sum_{\ell=1}^{i-1} \langle \theta, y_\ell\rangle } \prod_{\ell=1}^{i-1}\one \big\{  \rho+c^{-1}\langle \theta, y_\ell \rangle \ge 0\big\}. \notag
\end{align}
Notice further that \eqref{e:third.line.ratio} is bounded by $1$. 
For the remaining term in \eqref{e:final.form.exp},  we use \eqref{e:conv.third.line.ratio} and Lemma \ref{l:RV.seq.exp1} to ensure that 
\begin{align*}
&\lim_{m\to\infty}\Big( 1-I_s \big( (\aRrho)\theta, (\aRrho)\theta+y_1,  \dots, (\aRrho)\theta  + y_{i-1} \big) \Big)^{v_m-i} \\
&=\lim_{m\to\infty} \bigg(  1-\int_{\B\big( \{ 0,y_1,\dots,y_{i-1} \}; s \big)} f\Big( \big\| \big( \aRrho \big)\theta +z  \big\| \Big)\, dz \bigg)^{v_m-i}\\
&=\lim_{m\to\infty} \exp \bigg\{  -v_m f(R_{q_m}) \int_{\B\big( \{ 0,y_1,\dots,y_{i-1} \}; s \big)}  \frac{f\big( \big\| \big( \aRrho \big)\theta +z  \big\| \big)}{f(R_{q_m})}\, dz \bigg\}\\
&= \exp \Big\{  -\lambda  e^{-\rho}\int_{\B\big( \{ 0,y_1,\dots,y_{i-1} \}; s \big)} e^{-c^{-1} \langle \theta, z\rangle}  dz \Big\}. 
\end{align*}
Multiplying all of the upper bounds derived thus far, one can bound the triple integral in \eqref{e:final.form.exp} by 
\begin{align*}
&\int_{S^{d-1}}J(\theta)\int_0^\infty 2(1\vee \rho)^{d-1} \int_{(\R^d)^{i-1}} h_t^\ijp (0,\by) \sum_{\ell=0}^\infty \one \big\{ 0<\rho \le \epsilon^{-1} e^{(\ell+1)\epsilon} \big\}\, e^{-\ell} d\by \, d\rho\, d\theta \\
&=2s_{d-1} \int_{(\R^d)^{i-1}} h_t^\ijp (0,\by) d\by \int_0^\infty \sum_{\ell=0}^\infty \one \big\{ 0<\rho \le \epsilon^{-1} e^{(\ell+1)\epsilon} \big\}\, e^{-\ell} (1\vee \rho)^{d-1}d\rho \\
&\le C^* \Big( \frac{e^\epsilon}{\epsilon} \Big)^d \sum_{\ell=0}^\infty e^{-(1-\epsilon d)\ell}. 
\end{align*}
Since the last term is finite due to $0 <\epsilon<1/d$, one can apply the dominated convergence theorem to  obtain \eqref{e:first.moment1.exp.weak.core}. 

For the variance asymptotics in \eqref{e:second.moment1.exp.weak.core}, we write $\text{Var}\big( T_m^\ijupa (t,s) \big) = A_m+B_m$, where $A_m$ and $B_m$ are defined in \eqref{e:def.AmBm.decomp}. Our argument here is nearly the same as that for the proof of \eqref{e:vartmup}. More concretely, by virtue of Lemma \ref{l:connectivity.prob.exp} $(ii)$ below, one can replace \eqref{e:EIell}, \eqref{e:ECm}, and \eqref{e:EDm} respectively by
$$
\E[I_\ell] \le C^* a(R_{q_m})R_{q_m}^{d-1}, \  \ \ell\in \{ 1,2,\dots,i \}, 
$$
and
$$
v_{m+1}^{2i} \E[C_m]\, \one \big\{ |\Y\cap \Y'|=0 \big\} = o \big( a(R_{q_m}) R_{q_m}^{d-1} \big), \ \ \ m\to\infty, 
$$
and 
$$
v_{m+1}^{2i} \E[D_m] \, \one \big\{|\Y\cap \Y'|=0  \big\} \le C^* a(R_{q_m})R_{q_m}^{d-1}. 
$$
Since the rest of the argument is totally the same as that for \eqref{e:vartmup}, we now conclude that 
$$
\sup_{m\ge1} \big( a(R_{q_m})R_{q_m}^{d-1} \big)^{-1} A_m <\infty, \ \  \text{ and } \ \  \sup_{m\ge1} \big( a(R_{q_m})R_{q_m}^{d-1} \big)^{-1} B_m <\infty,
$$
as desired. 

Finally, for \eqref{e:first.moment.VmM.exp.weak.core} we have 
\begin{align*}
\big(a(R_{q_m})R_{q_m}^{d-1}\big)^{-1}\, \E\big[ V_{m,M}(t) \big] &= \sum_{i=M+1}^\infty i^{k+1} \binom{v_{m+1}}{i} \big(a(R_{q_m})R_{q_m}^{d-1}\big)^{-1}\, \\
&\quad  \times \P \Big( \cc \big( \{ X_1,\dots,X_i \}, t \big) \text{ is connected}, \, \M(X_1,\dots,X_i) \ge R_{q_m} \Big). 
\end{align*}
Proceeding as in \eqref{e:after.Stirling} and using Lemma \ref{l:connectivity.prob.exp} $(ii)$, it turns out that 
$$
\big(a(R_{q_m})R_{q_m}^{d-1}\big)^{-1}\, \E\big[ V_{m,M}(t) \big] \le C^*\sum_{i=M+1}^\infty i^{k-1} \big( \lambda(1+\delta) e\omega_d\big)^i <\infty. 
$$
Therefore, combining Lemma \ref{l:connectivity.prob.exp} $(i)$ and the dominated convergence theorem yields \eqref{e:first.moment.VmM.exp.weak.core}. 
\end{proof}

The result below is analogous to Lemma \ref{l:connectivity.prof.heavy} when the density has an exponentially decaying tail. 

\begin{lemma} \label{l:connectivity.prob.exp}
$(i)$ Under the assumptions of Theorem \ref{t:fslln.expdt} $(ii)$, for every $i\ge k+2$ and $t\in [0,1]$, 
\begin{align}  
&v_{m+1}^i \big(  a(R_{q_m})R_{q_m}^{d-1}\big)^{-1} \, \P \Big( \cc\big( \{ X_1,\dots,X_i \},  t\big) \emph{ is connected}, \, \notag \\
&\qquad \qquad \qquad\qquad \qquad \qquad \qquad\M(X_1,\dots,X_i)\ge R_{q_m}\Big) \to \lambda^i \zeta_i'(t), \ \ m\to\infty, \notag
\end{align}
where $\zeta_i'(t)$ is defined in  \eqref{e:def.zeta.prime.it}.  \\
\noindent $(ii)$ Moreover, let $\delta>0$ be a constant so small that $\lambda (1+\delta) e\omega_d < 1$. Then, for each $i\ge k+2$ and $t\in [0,1]$, there exists $N\in\N$ such that for all $m \ge N$, 
\begin{align}
&v_{m+1}^i \big(  a(R_{q_m})R_{q_m}^{d-1}\big)^{-1} \, \P\Big( \cc\big( \{ X_1,\dots,X_i \},  t\big) \emph{ is connected}, \, \label{e:bound.connectivity.prob.exp}  \\
&\qquad \qquad \qquad \qquad \qquad \qquad  \M(X_1,\dots,X_i)\ge R_{q_m}\Big) \le  C^*\lambda^i(1+\delta)^ii^{i-2}\omega_d^{i-1},  \notag 
\end{align}
where $C^*>0$ is a constant independent of $i$ and $m$. 
\end{lemma}

\begin{proof}
We first show \eqref{e:bound.connectivity.prob.exp}. 
Performing the same change of variables as in \eqref{e:second.time.change.variables} and \eqref{e:final.form.exp}, we have 
\begin{align}
&v_{m+1}^i \big(  a(R_{q_m})R_{q_m}^{d-1}\big)^{-1} \, \P\Big( \cc\big( \{ X_1,\dots,X_i \},  t\big) \emph{ is connected}, \, \M(X_1,\dots,X_i)\ge R_{q_m}\Big) \label{e:change.variables.connectivity.prob}\\
&=v_{m+1}^i f(R_{q_m})^i \int_{S^{d-1}} J(\theta) \int_0^\infty \Big( 1+\frac{a(R_{q_m})}{R_{q_m}}\rho  \Big)^{d-1} \notag \\
&\qquad\qquad   \times \int_{(\R^d)^{i-1}} \one \Big\{ \cc \big( \{0,\by\}, t \big) \text{ is connected} \Big\}\, \frac{f\big( R_{q_m}+a(R_{q_m}) \rho \big)}{f(R_{q_m})} \notag \\
&\qquad\qquad  \times \prod_{\ell=1}^{i-1} \frac{f\big( \big\| \big( R_{q_m}+a(R_{q_m})\rho \big)\theta +y_\ell  \big\| \big)}{f(R_{q_m})}\, \one \Big\{ \big\| \big(  R_{q_m}+a(R_{q_m})\rho\big)\theta +y_\ell  \big\| \ge R_{q_m} \Big\} \, d\by\, d\rho \, d\theta. \notag 
\end{align}
By Lemma \ref{l:RV.seq.exp1}, there exists $N_1 \in \N$ so that for all $m\ge N_1$, we have $v_{m+1}f(R_{q_m}) \le \lambda (1+\delta)$. Employing the bounds derived in the proof of \eqref{e:first.moment1.exp.weak.core}, there exists $N_2\in \N$ such that for all $m\ge N:=N_1 \vee N_2$, one can bound \eqref{e:change.variables.connectivity.prob} by 
\begin{align*}
&\big(\lambda(1+\delta)  \big)^i \int_{S^{d-1}} J(\theta)\int_0^\infty 2 (1\vee \rho)^{d-1} \int_{(\R^d)^{i-1}} \one \Big\{ \cc \big( \{0,\by\}, t \big) \text{ is connected} \Big\} \\
&\qquad\qquad\qquad\qquad\qquad\qquad \times \sum_{\ell=0}^\infty \one \big\{ 0<\rho \le \epsilon^{-1} e^{(\ell+1)\epsilon} \big\} e^{-\ell} \, d\by\, d\rho\, d\theta \\
&= C^* \big(\lambda(1+\delta)  \big)^i  \int_{(\R^d)^{i-1}} \one \Big\{ \cc \big( \{0,\by\}, t \big) \text{ is connected} \Big\} d\by \\
&\le C^*  \big(\lambda(1+\delta)  \big)^i  i^{i-2} \omega_d^{i-1}, 
\end{align*}
where $\epsilon\in (0,1/d)$ is determined in \eqref{e:balkema.bound} and $C^*>0$ is a constant independent of $i$ and $m$. 

Finally, combining the convergences \eqref{e:ratio.plus.1}, \eqref{e:conv.main.ratio}, and \eqref{e:conv.third.line.ratio}, together with the dominated convergence theorem, can show that \eqref{e:change.variables.connectivity.prob} tends to $\lambda^i \zeta_i'(t)$ as $m\to\infty$, as required. 
\end{proof}

The last two lemmas below will be used for the proof of Theorem \ref{t:fslln.expdt} $(i)$. The proof of these lemmas are considerably similar to those of Lemmas \ref{l:first.second.asymptotics} and \ref{l:difference.binom.poisson}, so we skip their proofs. 

\begin{lemma}   \label{l:first.second.moments.exp2}
Under the assumptions of Theorem \ref{t:fslln.expdt} $(i)$, for every $t\in [0,1]$, as $m\to\infty$, 
\begin{align*}
\big( v_{m+1}^{k+2} a(R_{q_m})R_{q_m}^{d-1} f(R_{q_m})^{k+2} \big)^{-1} \E \big[S_m^\uparrow (t)\big] &\to \frac{\xi_k^\ktop (t; 0)}{(k+2)!}, \\
\big( v_{m}^{k+2} a(R_{p_m})R_{p_m}^{d-1} f(R_{p_m})^{k+2} \big)^{-1} \E \big[S_m^\downarrow (t)\big] &\to \frac{\xi_k^\ktop (t; 0)}{(k+2)!}, 
\end{align*}
and also, 
\begin{align*}
&\big( v_{m+1}^{k+3} a(R_{q_m})R_{q_m}^{d-1} f(R_{q_m})^{k+3} \big)^{-1} \E \big[ W_m(t) \big] \\
&\qquad \qquad \to \frac{1}{(k+3)!}\, \int_{S^{d-1}}\int_0^\infty \int_{(\R^d)^{k+2}} \one \Big\{ \cc\big( \{ 0,\by \}, t \big) \text{ is connected} \Big\} \\
&\qquad\qquad \qquad \qquad \qquad  \times e^{-\rho (k+3) - c^{-1}\sum_{\ell=1}^{k+2} \langle \theta, y_\ell \rangle} \prod_{\ell=1}^{k+2} \one \big\{  \rho +c^{-1}\langle \theta, y_\ell \rangle \ge 0 \big\}\, d\by\, d\rho\, J(\theta) \, d\theta, 
\end{align*}
where $\by=(y_1,\dots,y_{k+2})\in (\R^d)^{k+2}$. 
Moreover, 
\begin{align*}
&\sup_{m\ge 1}\big( v_{m+1}^{k+2} a(R_{q_m})R_{q_m}^{d-1} f(R_{q_m})^{k+2} \big)^{-1}  \text{Var}\big( S_m^\uparrow (t) \big) <\infty, \\
&\sup_{m\ge 1} \big( v_{m}^{k+2}a(R_{p_m}) R_{p_m}^{d-1} f(R_{p_m})^{k+2} \big)^{-1}\text{Var}\big( S_m^\downarrow (t) \big) <\infty. 
\end{align*}
\end{lemma}

\begin{lemma}  \label{l:difference.binom.poisson.exp}
Under the assumptions of Theorem \ref{t:fslln.expdt} $(i)$, for every $t\in [0,1]$ we have, as $m\to\infty$, 
\begin{align*}
&\big( v_{m+1}^{k+2} a(R_{q_m})R_{q_m}^{d-1}f(R_{q_m})^{k+2} \big)^{-1} \bigg\{  \sum_{\substack{\Y \subset \X_{v_{m+1}}, \\ |\Y|=k+2}} h_t^\ktop(\Y)\, \one \big\{ \M(\Y) \ge R_{q_m} \big\} -S_m^\uparrow(t) \bigg\} \to 0, \ \ \text{a.s.}, \\
&\big( v_{m}^{k+2} a(R_{p_m})R_{p_m}^{d-1}f(R_{p_m})^{k+2} \big)^{-1} \bigg\{  \sum_{\substack{\Y \subset \X_{v_{m}}, \\ |\Y|=k+2}} h_t^\ktop(\Y)\, \one \big\{ \M(\Y) \ge R_{p_m} \big\} -S_m^\downarrow(t) \bigg\} \to 0, \ \ \text{a.s.}, 
\end{align*}
and further, 
\begin{align*}
&\big( v_{m+1}^{k+3} a(R_{q_m})R_{q_m}^{d-1} f(R_{q_m})^{k+3}\big)^{-1} \\
&\qquad \qquad \qquad \times \bigg\{ \sum_{\substack{\Y\subset \X_{v_{m+1}}, \\ |\Y|=k+3}} \one \big\{ \cc(\Y, t) \text{ is connected}, \, \M(\Y)\ge R_{q_m} \big\} - W_m(t)\bigg\} \to 0, \ \ \text{a.s.}
\end{align*}
\end{lemma}


\end{document}